\documentclass[11pt]{amsart}
\usepackage{amssymb,amsmath,amsthm}
\usepackage{amsfonts}
\usepackage{enumerate}
\usepackage{dsfont}
\usepackage{cite}
\usepackage[colorlinks, citecolor=red]{hyperref}
\usepackage{mathrsfs}
\usepackage{epsfig}
\usepackage{lscape}
\usepackage{subfigure}
\usepackage{epstopdf}
\usepackage{caption}
\usepackage{algorithm}
\usepackage{algpseudocode}
\usepackage{multirow}
\usepackage{geometry}
\usepackage{paralist}
\usepackage{enumerate}
\usepackage{url}
\usepackage{graphicx,cite}
\usepackage{longtable}
\usepackage{tikz}
\usepackage{adjustbox}
\usepackage{CJKutf8}
\usepackage{hyperref}
\usepackage{url}

\newtheorem{definition}{Definition}[section]

\newtheorem{prop}[definition]{Proposition}
\newtheorem{thm}[definition]{Theorem}
\newtheorem{lemma}[definition]{Lemma}

\newtheorem{assumption}{Assumption}[section]
\newtheorem{REMK}[definition]{Remark}
\date{}

\begin{document}
	\begin{CJK}{UTF8}{gbsn}
		\baselineskip 18pt
		\bibliographystyle{plain}

\title[ASDCD for linearly constrained convex optimization]{Stochastic dual coordinate descent with adaptive heavy ball momentum for linearly constrained convex optimization}

\author{Yun Zeng}
\address{School of Mathematical Sciences, Beihang University, Beijing, 100191, China. }
\email{zengyun@buaa.edu.cn}

\author{Deren Han}
\address{LMIB of the Ministry of Education, School of Mathematical Sciences, Beihang University, Beijing, 100191, China. }
\email{handr@buaa.edu.cn}

\author{Yansheng Su}
\address{School of Mathematical Sciences, Beihang University, Beijing, 100191, China. }
\email{suyansheng@buaa.edu.cn}

\author{Jiaxin Xie}
\address{LMIB of the Ministry of Education, School of Mathematical Sciences, Beihang University, Beijing, 100191, China. }
\email{xiejx@buaa.edu.cn}

\begin{abstract}
	The problem of finding a solution to the linear system $Ax = b$ with certain minimization properties arises in numerous scientific and engineering areas.
	In the era of big data, the stochastic optimization algorithms  become increasingly significant due to their scalability for problems of unprecedented size.
	This paper focuses on the problem of minimizing a strongly convex function subject to linear constraints.
	We consider the dual formulation of this problem and adopt the stochastic coordinate descent to solve it.
	The proposed algorithmic framework, called adaptive stochastic dual coordinate descent, utilizes sampling matrices sampled from user-defined distributions to extract gradient information. Moreover, it employs Polyak's heavy ball momentum acceleration with adaptive parameters learned through iterations, overcoming the limitation of the heavy ball momentum method that it requires prior knowledge of certain parameters, such as the singular values of a matrix. With these extensions, the framework is able to recover many well-known methods in the context, including the randomized sparse Kaczmarz method, the randomized regularized Kaczmarz method,  the linearized Bregman iteration, and a variant of the conjugate gradient (CG) method. Additionally, we introduce an equivalent formulation that, in certain cases, substantially reduces the need for full-dimensional vector operations introduced by the momentum term. We prove that, with strongly admissible objective function, the proposed method converges linearly in expectation. Numerical experiments are provided to confirm our results.
\end{abstract}

\maketitle

\let\thefootnote\relax

\let\thefootnote\relax

\section{Introduction}
\label{sec1}

Consider the following linearly constrained convex optimization problem
\begin{equation}
	\label{main-prob}
	\min\limits_{x\in\mathbb{R}^n} \ f(x) \ \ \text{subject to} \ \  Ax=b,
\end{equation}
%where $A\in\mathbb{R}^{m\times n},b\in\mathbb{R}^m$, and $f(x)$ is  strongly convex but possibly nonsmooth which is used to select a solution with desired feature.
where $A\in\mathbb{R}^{m\times n},b\in\mathbb{R}^m$, and $f$ is  strongly convex but possibly nonsmooth.
The problem depicts a solution to the linear system $ Ax=b $ that possesses certain properties.
%which is used to select a solution with desired feature.
%Problem \eqref{main-prob} arises
It arises in many areas of scientific computing, such as compressed sensing \cite{candes2006robust,donoho2006compressed,cai2009linearized},  low-rank matrix recovery \cite{recht2010guaranteed,cai2010singular}, image processing \cite{chambolle2016introduction}, and machine learning \cite{lin2020accelerated}.

In this paper, we  consider applying the coordinate descent method to the dual problem of \eqref{main-prob}.
%a  dual coordinate descent method to solve \eqref{main-prob}. This method is based on
%which is well-known in optimization community.
We here provide a brief derivation of the method and the related convex analysis basics will be presented in Subsection \ref{subsection2.2}.
%; Refer to Section  \ref{section-2} for detailed analysis.
%and the tools from convex analysis used below will
% be presented in Section \ref{section-2}.
The associated Lagrangian function of \eqref{main-prob} is
$$
L(x,\lambda)=f(x)-\langle \lambda,Ax-b\rangle, \ x\in\mathbb{R}^n \ \text{and} \ \lambda\in\mathbb{R}^m,
$$
which induces the dual function
$$
\inf\limits_{x\in\mathbb{R}^n}\left\{ f(x)-\langle \lambda,Ax-b\rangle\right\}=-f^*(A^\top \lambda)+\langle\lambda,b\rangle,
$$
where $A^\top$ denotes the transport of $A$ and $f^*$ denotes the Legendre-Fenchel conjugate of $f$. Thus the corresponding dual problem of \eqref{main-prob} is
\begin{equation}\label{dual-prob}
	\min\limits_{\lambda\in\mathbb{R}^m}
	g(\lambda):=f^*(A^\top \lambda)-\langle\lambda,b\rangle.
\end{equation}
Since $f$ is strongly convex, $f^*$ is continuous differentiable and so is the function $g$. One may apply the coordinate descent to solve \eqref{dual-prob},
\begin{equation}\label{iter-org-cd}
	\lambda^{k+1}=\lambda^k-\alpha_k e_{i_k}e_{i_k}^\top\nabla g(\lambda^k),%\left(A\nabla f^*(A^\top\lambda^k)-b\right),
\end{equation}
where $\alpha_k>0$ is the stepsize, the index $i_k$ belongs to $ [m]:=\{1,\ldots,m\}$, $e_{i_k}$ denotes the ${i_k}$-th unit coordinate vector in $\mathbb{R}^m$,  and $\nabla g$ denotes the gradient of $g$.
%In practice, the index $i$ can be chosen in a random order or a deterministic order.
Since $\nabla g(\lambda^k)=A\nabla f^*(A^\top\lambda^k)-b$, one has $e_{i_k}^\top\nabla g(\lambda^k)=a_{{i_k}}^\top \nabla f^*(A^\top\lambda^k)-b_{i_k}$, where $a_{i_k}$ denotes the ${i_k}$-th row of $A$ and $b_{i_k}$ denotes the ${i_k}$-th entry of $b$. The method \eqref{iter-org-cd} can be rewritten as
$$
\lambda^{k+1}=\lambda^k-\alpha_k \left(a_{i_k}^\top \nabla f^*(A^\top\lambda^k)-b_{i_k}\right) e_{i_k}.%\left(A\nabla f^*(A^\top\lambda^k)-b\right),
$$
Denoting $x^k:=\nabla f^*(A^\top\lambda^k)$ and $z^k:=A^\top \lambda^k$, one obtains the following equivalent iteration strategy of  \eqref{iter-org-cd},
\begin{equation}\label{iter-cd}
	\begin{aligned}
		&z^{k+1}=z^k-\alpha_k(a_{i_k}^\top x^k-b_{i_k})a_{i_k},\\
		&x^{k+1}=\nabla f^*(z^{k+1}).
	\end{aligned}
\end{equation}
Particularly,
if the index $i_k$ is chosen randomly, it can recover several well-known methods.
%	Suppose that the index ${i_k}$ is chosen randomly.
When $f(x)=\frac{1}{2}\|x\|^2_2$, this iteration scheme \eqref{iter-cd} becomes the randomized Kaczmarz (RK) method \cite{Str09} for solving linear systems. When $f(x)=\mu\|x\|_1+\frac{1}{2}\|x\|^2_2$ with parameter $\mu>0 $, it becomes the randomized sparse Kaczmarz (RSK) method \cite{schopfer2019linear} for solving sparse signal recovery problems. %For the general $f$, the method \eqref{iter-cd} is called the regularized Kaczmarz investigated in.

	\subsection{Our contribution}

In this paper, we present a generic algorithmic framework, named the stochastic dual coordinate descent (SDCD) method, for solving the linearly constrained optimization problem \eqref{main-prob} via solving its unconstrained dual reformulation \eqref{dual-prob} by stochastic algorithms. Noting that $ e_{i_k} $ in \eqref{iter-org-cd} acts as the role that extracts partial information of the gradient, we extend $ e_{i_k} $ to a general \textit{sampling matrix} $ S_k \in \mathbb{R}^{m\times q_k} $ and apply the following iteration format,
\begin{equation}
	\label{iter-org-sdcd}\lambda^{k+1}=\lambda^k-\alpha_k S_k S_k^\top \nabla g(\lambda^k).
\end{equation}
The matrix $S_k $ is sampled from some probability spaces $ (\Omega_k,\mathcal{F}_k,P_k) $ which may vary across iterations. Although it is actually an extended version of the primal stochastic dual coordinate descent method, we refer to it as SDCD for the sake of convenience.

We further incorporate the Polyak's heavy ball momentum technique \cite{polyak1964some} into SDCD, resulting in the following adaptive SDCD (ASDCD) algorithmic framework
$$
	\lambda^{k+1}=\lambda^k-\alpha_k S_k S_k^\top \nabla g(\lambda^k)+\beta_k(\lambda^k-\lambda^{k-1}),
$$
where both $ \alpha_{k} $ and $ \beta_{k} $ are determined adaptively.
%The proposed fast stochastic dual coordinate descent (FSDCD) algorithmic framework initializes with the points $\lambda^0$ and $\lambda^1$, and updates the next iterate using the following formula
%Moreover, to accelerate the optimization process,  Polyak's heavy ball momentum technique \cite{polyak1964some} is incorporated into our framework. %Specifically, the heavy ball momentum method is utilized to accelerate the gradient descent algorithm by introducing a momentum term into the update rule.
%The proposed fast stochastic dual coordinate descent (FSDCD) algorithmic framework initializes with the points $\lambda^0$ and $\lambda^1$, and updates the next iterate using the following formula
%\begin{equation}\label{fsdcd-o}
%where $\alpha_k$ is the stepsize, $\beta_k$ is the momentum parameter, and $S_k$ is a random matrix sampled from the probability space $(\Omega,\mathcal{F},P)$.
%By using the similar argument as above, we can get the following equivalent iteration strategy of
Similarly, we can derive an equivalent iteration format,
%for the proposed FSDCD algorithmic framework
%Using the same procedure, we can derive an equivalent iteration strategy for the proposed FSDCD algorithmic framework
\begin{equation}\label{fsdcd-o}
	\begin{aligned}
		&z^{k+1}=z^k-\alpha_kA^\top S_kS_k^\top(A x^k-b)+\beta_k(z^k-z^{k-1}),\\
		&x^{k+1}=\nabla f^*(z^{k+1}).
	\end{aligned}
\end{equation}
Note that when $\beta_k=0$ and $S_k=e_{i_k}$, \eqref{fsdcd-o} reduces to \eqref{iter-cd}. %the dual coordinate descent method  \eqref{iter-cd}.
We now comment on the main contributions of this work.

	\begin{itemize}
	\item[1.] We develop a framework of the stochastic dual coordinate descent (SDCD) method for solving the linearly constrained convex optimization problem. At each iteration, a sampling matrix $ S_k $ is drawn to extract partial information of the matrix $ A $.
	In addition, instead of relying on a fixed probability space $(\Omega, \mathcal{F}, P)$, we utilize a class of probability spaces $\{(\Omega_k, \mathcal{F}_k, P_k)\}_{k\geq0}$ to generate the random  matrix $S_k$ at each iteration. This framework is flexible and can recover a wide range of popular algorithms, including the linearized Bregman iteration, the randomized sparse Kaczmarz method, and their variants. Furthermore, it also enables us to design more versatile hybrid algorithms with improved performance, accelerated convergence, and better scalability.
	%This alternative approach can enable us to design more versatile hybrid algorithms that have the potential to improve optimization performance, accelerate convergence rates, and better scalability. Furthermore, by the manipulation of the probability space, we can recover a wide range of popular algorithms as special cases, including the linearized Bregman iteration, the randomized sparse Kaczmarz method, and its variants, while also deriving completely new methods.
	\item[2.] The Polyak's heavy ball momentum (HBM) method  has attracted much attention in recent years  due to its ability to improve the convergence of the gradient descent (GD) method. Recently, a fruitful line of research has been dedicated to extending this acceleration technique to enhance the performance of the  stochastic gradient descent (SGD) method \cite{loizou2020momentum,barre2020complexity,sebbouh2021almost,han2022pseudoinverse}. However, the resulting stochastic heavy ball momentum  (SHBM) method has a drawback that it requires prior knowledge of certain problem parameters, such as the singular values of the coefficient matrix \cite{loizou2020momentum,han2022pseudoinverse,polyak1964some,ghadimi2015global,bollapragada2022fast}. Hence, it is an open problem whether one can design an adaptive scheme for obtaining the parameters $\alpha_k$ and $\beta_k$ to get rid of any of these problem parameters \cite{barre2020complexity,bollapragada2022fast}. This paper answers the problem for a class of unconstrained convex optimization problems that are reformulated from linearly constrained optimization problems. We adopt the HBM technique to accelerate the convergence of  the SDCD method and obtain the adaptive SDCD (ASDCD) method. Particularly, based on the majorization technique \cite{li2016majorized,chen2017efficient}, we propose a novel strategy for the ASDCD method to  learn the parameters adaptively and prove that the method converges linearly in expectation.
	        \item[3.] We develop an equivalent formulation of the ASDCD method that, in certain cases, largely avoids the full-dimensional vector operations introduced by the momentum term, inspired by the concept of variable transformation in \cite{lee2013efficient,fercoq2015accelerated}. In particular, when $f(x)=\frac{1}{2}\|x\|_{2}^{2}$, since the deterministic version of ASDCD coincides with a variant of the conjugate gradient (CG) method, this reformulation offers an efficient implementation strategy for CG-type methods in solving linear systems with sparse cofficient matrices.
			%        
			%        While our ASDCD algorithm incorporates a momentum term to accelerate convergence, the acceleration comes at the cost of extra computation. For sparse coefficient matrices, the momentum term disrupts the matrix's sparse structure, and the computational cost becomes substantial. %which conflicts with the core principle of coordinate descent. 
			%        In the revised manuscript, we have addressed this concern by providing an efficient implementation of ASDCD that circumvents this issue for sparse problems with specific properties. Besides, we discuss two alternative strategies to reduce the additional computational cost introduced by the momentum term. We are particularly grateful to Referee 1 for drawing this issue to our attention.
\end{itemize}

	\subsection{Related work}

There exist various approaches for solving problems of the form \eqref{main-prob}, such as
%Classical approaches include
the (accelerated) proximal gradient method \cite{beck2009fast,lan2020first,luo2022differential}, the primal-dual method \cite{chambolle2011first,condat2013primal}, the augmented Lagrangian method \cite{bertsekas2014constrained,he2022fast,luo2022primal,luo2021accelerated}, and the alternating direction method of multipliers (ADMM) \cite{boyd2011distributed,han2022survey}.
However, since these approaches require whole matrix-vector products, they are typically unavailable when the matrix $A$ is extremely huge that it is impossible to be stored entirely in the RAM.
%,  these approaches are usually unavailable as they require  whole matrix-vector products.
%However, these approaches are usually unavailable when the size of $A$ is very large as they require  whole matrix-vector products.
To deal with such issues, there emerge iterative methods that only requires partial information of $A$ at each step, for instance, the  Kaczmarz method \cite{Kac37} and the coordinate descent method \cite{Cha08,bai2021matrix}, their randomized variants \cite{Str09,Lev10}, and the corresponding modifications and extensions \cite{Bai18Gre,Gow19,liu2016accelerated,han2022pseudoinverse,loizou2020momentum,Nec19,needell2014paved,moorman2021randomized,Gow15,zeng2023randomized}. Moreover, in recent years, primal-dual coordinate descent (PDCD) \cite{alacaoglu2020random, chambolle2018stochastic, fercoq2019coordinate, zhang2017stochastic}, a randomized coordinate variant of the primal-dual method, has also been proposed to solve large-scale problems.
At each iteration, PDCD processes a randomly selected subset of coordinates and updates the corresponding variables, thereby reducing memory requirements and per-iteration computational costs.

\subsubsection{Kaczmarz method}
The Kaczmarz method \cite{Kac37}, also known as the algebraic reconstruction technique (ART) \cite{herman1993algebraic,gordon1970algebraic}, is an iterative method for solving large-scale linear systems $Ax=b$.
Starting from  $x^0\in\mathbb{R}^n$, the Kaczmarz method constructs $x^{k+1}$ by
$$
x^{k+1}=x^k-\frac{\langle a_{i_k},x^k\rangle-b_{i_k}}{\|a_{i_k}\|^2_2}a_{i_k},
$$
where $i_k$ is selected from $[m] $ according to some selection rules, including cyclic rules \cite{Kac37,censor1981row}, greedy rules \cite{Gri12}, or random rules \cite{Str09}.
Notably, Strohmer and Vershynin \cite{Str09} showed that if the index $ i_k$ is selected randomly with probability proportional to $\|a_{i_k}\|^2_2$, then the resulting \emph{randomized  Kaczmarz} (RK) method converges linearly in expectation. The iteration scheme apparently shows that it only requires a single row of the matrix $A\in\mathbb{R}^{m\times n}$ at each iteration, endowing the method with low RAM occupation and fast data transfer. These features make the Kaczmarz method a practically efficient iterative solver to linear systems, especially for the mentioned case where $ A $ is too large to be stored entirely in the RAM. Therefore, a large amount of researches on the refinements and extensions of the Kaczmarz method have been studied. We refer to \cite{bai2023randomized} for a recent survey on them. %\cite{Bai18Gre,Gow19,liu2016accelerated,han2022pseudoinverse,zeng2023randomized,loizou2020momentum,Nec19,needell2014paved,moorman2021randomized,Gow15}.
%It can be seen that the Kaczmarz method only requires processing one row of the coefficient matrix $A\in\mathbb{R}^{m\times n}$ and the corresponding entry of the right-hand side vector $b\in\mathbb{R}^m$ at each iteration.  This feature allows for small memory occupation and data transfer, making the Kaczmarz method a practical and efficient iterative solver for linear systems, especially when the coefficient matrix $A$ is too large to be stored entirely in the computer memory.

%In the literature, there are empirical evidences that selecting a working row randomly from the matrix $A$ can often lead to better convergence of the Kaczmarz method compared to choosing them sequentially \cite{herman1993algebraic,natterer2001mathematics,feichtinger1992new}.
%In the seminal paper \cite{Str09}, Strohmer and Vershynin showed that if the index $ i_k$ is selected randomly with probability proportional to $\|a_{i_k}\|^2_2$, then the resulting \emph{randomized  Kaczmarz} (RK) method converges linearly in expectation. Subsequently, there is a large amount of work on the refinements and extensions of the RK method have been studied \cite{Bai18Gre,Gow19,liu2016accelerated,han2022pseudoinverse,zeng2023randomized,loizou2020momentum,Nec19,needell2014paved,moorman2021randomized,Gow15}.
%We refer to \cite{bai2023randomized} for a recent survey  about the Kaczmarz method.

Recently, Tondji and Lorenz \cite{tondji2022faster} proposed a new variant of the RK method, named the randomized sparse Kaczmarz method with averaging (RSKA), for approximating sparse solutions to linear systems. Let $\mathcal{J}_k$ consist of $\eta$ indexes sampled from $[m]$ and let $\omega_i\geq0$ represent the weight corresponding to the $i$-th row. The RSKA update is given by
\begin{equation}\label{rska}
	\begin{aligned}
		&z^{k+1}=z^k-\frac{1}{\eta} \sum\limits_{i\in \mathcal{J} _k}\omega_i\frac{a_{i}^\top x^k-b_i}{\|a_{i}\|^2_2}a_{i},\\
		&x^{k+1}=S_\mu(z^{k+1}),
	\end{aligned}
\end{equation}
where $S_\mu(\cdot)$ is the soft shrinkage operator defined as \eqref{soft-op}.
%$(S_\lambda(x))_i=\max\{|x_i|-\lambda,0\}\cdot \text{sign}(x_i)$.
If $\mathcal{J}_k$ is a singleton  and the weights are chosen as $\omega_i=1$ for $i\in[m]$, it reduces to  the standard randomized sparse Kaczamrz (RSK) method \cite{schopfer2019linear}. We note that our SDCD framework can recover an adjusted RSKA method, where instead of using a constant stepsize as in \eqref{rska}, an adaptive stepsize is employed; See  Remark \ref{remark-xie-0607-1}.
In practice, the methods with well-designed adaptive stepsizes typically perform better than those with constant ones \cite{lorenz2014linearized,necoara2022stochastic}.

\subsubsection{Stochastic mirror descent }
The stochastic mirror descent (SMD) method as well as its variants \cite{beck2003mirror,lan2012validation,nemirovski2009robust} is one of the most widely used algorithms in stochastic optimization for non-smooth Lipschitz continuous convex  functions. Enlightened by the pioneering work \cite{nemirovskij1983problem}, SMD has been studied in the context of convex programming \cite{nemirovski2009robust}, saddle-point problems \cite{mertikopoulos2018optimistic}, and monotone variational inequalities \cite{mertikopoulos2018stochastic}. % was recently studied by Ryan et al \cite{d2021stochastic}.

%
%Recall that
The SMD method for solving the finite-sum problem
\begin{equation} \label{FSP}
	\begin{aligned}
		\min\frac{1}{m}\sum_{i=1}^m h_i(x)
	\end{aligned}
\end{equation}
utilizes the update
\begin{equation}\label{SMD-iter}
	x^{k+1}=\arg\min_{x\in\mathbb{R}^n}\left\{t_k\left\langle \nabla h_{i_k}(x^k),x-x^k\right\rangle+D_{\psi, z^k}(x^k, x)\right\},
\end{equation}
where $t_k$ is the stepsize, $i_k$ is selected randomly, $ \psi $ is the mirror map that is $ \mu_\psi$-strongly convex,  $ z^k\in\partial \psi(x^k)$, and $D_{\psi, z}$ is the Bregman distance associated to $\psi$ that is defined later (Definition \ref{bregma-dis}).
%$f$ is a  strictly convex function,  $ z^k\in\partial f(x^k)$, and $D_{f, z}$ is the Bregman distance associated to $f$ that is defined later (Definition \ref{bregma-dis}). %We demonstrate our SDCD method can be reinterpreted as mirror descent with a novel adaptive stepsize.
When $\psi(x)=\frac{1}{2}\|x\|^2_2$, it reduces to the \emph{stochastic gradient descent} (SGD) \cite{hardt2016train,robbins1951stochastic,ma2017stochastic} method.
%In recent years, there has been a surge of interest in the development and analysis of SMD.
Recently, Ryan et al. \cite{d2021stochastic} studied  the SMD method for solving \eqref{FSP} with \textit{mirror stochastic Polyak stepsize}
\begin{equation}\label{polyak-stepsize}
	t_k=\frac{\mu_{\psi}(h_{i_k}(x^k)-\widehat{h}_{i_k})}{c\left\|\nabla h_{i_k}(x^k)\right\|^2_2},
\end{equation}
where $c>0$ is a fixed constant and $\widehat{h}_i= \inf_{x \in \mathbb{R}^{n}} h_{i}(x)$. It provides a more reliable approach to determine $ t_k $ than typical hyperparameter tuning.
The method is proved to be convergent for lower bounded convex functions $ h_i $, if the \emph{interpolation} condition holds, i.e. there exists $\widehat{x}\in\mathbb{R}^n$ such that $h_i(\widehat{x})=\widehat{h}_i$ for all $i=1,\ldots,m$.
%it has been proven that the iterates of this method converge.
%Despite these efforts, identifying adaptive stepsizes that do not require hyperparameter tuning remains a challenging problem.
Although this assumption seems restrictive, it can be satisfied under certain circumstances, e.g. the stochastic optimization problem reformulated from the linear constraint \eqref{reformulate-ls}.
We establish the connection between our SDCD framework and the SMD method, and show that the adaptive stepsize in our SDCD method framework is in actual a kind of the mirror stochastic Polyak stepsize; See Remark \ref{remark-xie-0429}.

%our SDCD framework can be reinterpreted in the SMD perspective

%; See Remark \ref{remark-xie-0429}.

% We demonstrate our SDCD method can be reinterpreted as SMD with mirror stochastic Polyak stepsize; See Remark \ref{remark-xie-0429}.
%Despite these efforts, identifying adaptive stepsizes that do not require hyperparameter tuning remains a challenging problem.
%In this context, the \emph{mirror stochastic Polyak stepsize}
%\begin{equation}\label{polyak-stepsize}
%t_k=\frac{\mu_{f}(h_{i_k}(x^k)-\hat{h}_{i_k})}{c\left\|\nabla h_{i_k}(x^k)\right\|^2_2}
%\end{equation}
%was proposed in \cite{d2021stochastic} provided that $f$ is $\mu_{f}$-strongly convex (Definition \ref{strongly-convex-def}), where $c>0$ is a fixed constant and $\hat{h}_i=\inf h_i$.
%For convex lower bounded functions $h_i$, if the \emph{interpolation} condition holds, i.e. there  exists $\hat{x}\in\mathbb{R}^n$ such that $h_i(\hat{x})=\hat{h}_i$ for all $i=1,\ldots,m$,  it has been proven that the iterates of this method converge.
%Although this assumption seems restrictive, it can be satisfied under certain circumstances, e.g. the stochastic optimization problem reformulated from the linear constraint; See \eqref{reformulate-ls}. %  \eqref{main-prob}.
%For example, for the stochastic optimization problem \eqref{SOP1}, the optimal value is zero, and furthermore, the optimal value of $f_{S}$ is also zero for any $S$.
% We demonstrate our SDCD method can be reinterpreted as SMD with mirror stochastic Polyak stepsize; See Remark \ref{remark-xie-0429}.

\subsubsection{Heavy ball momentum method}

The  heavy ball momentum (HBM) method is a modification of the classic gradient descent (GD) method, which was introduced in $1964$ by Polyak  \cite{polyak1964some}. For minimizing $ g(\lambda), $
it introduces the momentum term $ \beta(\lambda^k-\lambda^{k-1}) $ to the original GD iteration format, writing as
\[
\lambda^{k+1} = \lambda^k - \alpha \nabla g(\lambda^k) + \beta(\lambda^k-\lambda^{k-1}).
\]
The local convergence of the HBM method was originally established for twice differentiable, strongly convex, and smooth functions $ g $, showing that it converges at an accelerated rate with appropriate parameters $\alpha$ and $\beta$ \cite{polyak1964some}.
While only recently, a global sublinear convergence of the HBM method for smooth and convex functions was given in \cite{ghadimi2015global}.
%the authors in \cite{ghadimi2015global} showed that the  HBM method converged globally and sublinearly for smooth and convex functions.
Inspired by its success, several recent works extend the HBM technique to speed up the stochastic version of the GD method (SGD), called the stochastic HBM (SHBM) method \cite{loizou2020momentum,barre2020complexity,sebbouh2021almost,han2022pseudoinverse,P2017Stochastic,
	loizou2021revisiting,morshed2020stochastic}.
% In this paper, we will use the HBM method to accelerate our DSCD method.

However, it is well-known that one limitation of the HBM method is that $ \alpha $  and  $\beta$ may rely on certain problem parameters that are generally inaccessible. For instance, the optimal  choices of the parameters for the SHBM method for solving the linear system $Ax=b$ require knowledge of the largest and smallest nonzero singular values of the matrix $A$ \cite{loizou2020momentum,polyak1964some,ghadimi2015global,bollapragada2022fast}. Therefore, a strategy that learns the parameters $ \alpha $ and $ \beta $ adaptively would be especially beneficial to the practical performance of the SHBM method  \cite{barre2020complexity,bollapragada2022fast}.
%In fact, it is an open problem whether one can  design a theoretically supported adaptive SHBM method \cite{barre2020complexity,bollapragada2022fast}.
Recently, Zeng et al. have provided a solution in the context of solving linear systems \cite{zeng2024adaptive}. They showed that the proposed adaptive SHBM (ASHBM) method
converges with an improved rate. 
%In this paper, we also combine the SDCD framework with heavy ball momentum acceleration, where the parameters are adaptively updated via iteration information. 
While our work also integrates HBM into the stochastic dual coordinate descent (SDCD) framework with adaptive parameter updates, it differs from ASHBM in three key aspects. First, our method determines the parameters via a majorization technique and employs an incremental scheme to ensure practical computability, whereas ASHBM derives them through orthogonal projections. Second, in the special case where \( f(x) = \frac{1}{2}\|x\|_2^2 \), our approach admits a more efficient implementation that significantly reduces the full-dimensional operations required by the momentum term. Finally, we establish a linear convergence rate under weaker, more general assumptions than those required by ASHBM. A recent paper \cite{lorenz2023minimal}, published online  around the same time as our working paper \cite{yun2023fast}, presented an algorithm closely related to the adaptive  strategy  presented here. Their convergence results are slightly different from ours. Beyond the investigations in \cite{lorenz2023minimal}, we consider the relationship between our framework, and the SMD method and the conjugate gradient method. In addition, we provide a geometric interpretation of our approach.

\subsection{Organization}
The remainder of the paper is organized as follows.
After introducing some preliminaries in Section 2, we present and analyze the SDCD method with adaptive stepsizes in Section 3. In Section 4, we propose the adaptive SDCD (ASDCD) method and show its linear convergence rate.
In Section 5, we perform some numerical experiments to show
the effectiveness of the proposed method. We conclude the paper in Section 6. Proofs of all main results are provided in the appendix.

\section{Preliminaries}
\label{section-2}
\subsection{Notations}
Throughout the paper, for any random variables $\zeta$, we use $\mathbb{E}[\zeta]$ to denote the expectation of $\zeta$. For an integer $m\geq 1$, let $[m]:=\{1,\ldots,m\}$.
For any vector $x\in\mathbb{R}^n$, we use $x_i,x^\top$,  $\|x\|_1$, and $\|x\|_2$ to denote the $i$-th entry, the transpose,  the $\ell_1$-norm, and the  $\ell_2$-norm of $x$, respectively.
For any matrix $A\in\mathbb{R}^{m\times n}$, we use $a_{i}^\top,A^\top,\|A\|_2,\|A\|_F$,  and $\mbox{Range}(A)$ to denote the $i$-th row, the transpose, the spectral norm, the Frobenius norm, and the column space, respectively.
For a given index set $\mathcal{I}$, we use $A_{\mathcal{I}}$ to denote the row submatrix of the matrix $A$ indexed by $\mathcal{I}$. The cardinality
of the set $\mathcal{I}$ is denoted by $| \mathcal{I}|$.
%The nonzero singular values of a matrix $A$ are $\sigma_1(A)\geq\sigma_2(A)\geq\ldots\geq\sigma_{r}(A):=\sigma_{\min}(A)>0$, where $r$ is the rank of $A$,
%	$\sigma_{\max}(A)$ and $\sigma_{\min}(A)$ denotes the largest and the smallest nonzero singular values of $A$.  We note that $\|A\|_2=\sigma_{1}(A)$ and $\|A\|_F=\sqrt{\sum_{i=1}^r \sigma^2_i(A)}$.
We use $\sigma_{\min}(A)$ to denote the smallest nonzero singular value of $A$, and use $\lambda_{\max}(A^\top A)$ and $\lambda_{\min}(A^\top A)$ to denote  the largest and smallest eigenvalues of $A^\top A$, respectively. In addition, for any positive difinite matirx $H \in \mathbb{R}^{n \times n}$, we define the $H$-inner product and the induced $H$-norm by
		$
		\langle x,y \rangle_{H}= \langle x, Hx \rangle$ and $ \|x\|_{H}=\sqrt{\langle x, x \rangle_{H}}
		$, respectively.
The soft thresholding operator (also known as shrinkage) $S_\mu(\cdot)$ is defined componentwise as
\begin{equation}\label{soft-op}
	\left(S_\mu(x)\right)_i=\max\{|x_i|-\mu,0\}\cdot \text{sign}(x_i),
\end{equation}
where $x\in\mathbb{R}^n$ and $\text{sign}(\cdot)$ is the signum function which returns the sign of a nonzero number and zero otherwise. %In addition, the indicator function $\mathbb{I}_C$ of set $ C $ is defined as
%\begin{equation}
%\nonumber
%\mathbb{I}_C=
%\left\{\begin{array}{ll}
	%1, \;\; \text{if} \; x \in C;
	%\\
	%0, \;\; \text{otherwise}.
	%\end{array}
	%\right.
	%\end{equation}
	
	\subsection{Convex optimization basics} \label{subsection2.2}
	
	This subsection aims to recall some concepts and properties about convex functions and Bregman distance. We refer readers to \cite{rockafellar1997convex,beck2017first} for more detailed analysis.
	
	\begin{definition}[subdifferential]  For a convex function $f: \mathbb{R}^{n} \rightarrow \mathbb{R}$, its subdifferential at $x \in \mathbb{R}^{n}$ is defined as
		$$
		\partial f(x):=\left\{z\in\mathbb{R}^n \mid f(y) \geq f(x)+\langle z, y-x\rangle, \ \forall \ y \in \mathbb{R}^{n}\right\}.
		$$
	\end{definition}
	
	\begin{definition}[$\gamma$-strong convexity] \label{strongly-convex-def}
		A function $f: \mathbb{R}^{n} \rightarrow \mathbb{R}$ is called $\gamma$-strongly convex for a given $\gamma>0$ if the following inequality holds for any $x, y \in \mathbb{R}^{n}$ and $z \in \partial f(x)$,
		$$
		f(y) \geq f(x)+\langle z,y-x\rangle+\frac{\gamma}{2}\|y-x\|_{2}^2.
		$$
	\end{definition}
	
	As an example, the function $f(x)=\frac{1}{2}\|x\|_{2}^2$ is differentiable and $1$-strongly convex. Moreover, it is easy to show that the function $h(x)+\frac{1}{2}\|x\|_{2}^2$ is $1$-strongly convex if $h(x)$ is convex.
	
	% \begin{definition}[$L$-smoothness]
		% 	Let  $g: \mathbb{R}^{n} \rightarrow \mathbb{R}$ be a differentiable convex function. Then $g$ is $L$-smooth if and only if for all $x, y \in \mathbb{R}^{n}$, it holds that
		% 	$$
		% 	g(y) \leq g(x)+\left\langle\nabla g(x), y-x\right\rangle+\frac{L}{2 }\|y-x\|_{2}^2.
		% 	%\|\nabla g(x)-\nabla g(y)\|_2\leq L\|x-y\|_2.
		% 	$$
		% \end{definition}
	\begin{definition}[$L$-smoothness]
		Let $g: \mathbb{R}^{n} \rightarrow \mathbb{R}$ be a differentiable function. Then $g$ is $L$-smooth if there exists a constant \( L > 0 \) such that for all \( x, y \in \mathbb{R}^{n} \),
		$$
		\| \nabla g(x)- \nabla g(y) \|_{2} \leq L\|x-y\|_{2}.
		$$
	\end{definition}
	If \( g \) is \( L \)-smooth, then for all \( x, y \in \mathbb{R}^{n} \), the following inequality holds \cite[Lemma 5.7]{beck2017first}:
	$
	g(y) \leq g(x) + \langle \nabla g(x), y - x \rangle + \frac{L}{2} \| y - x \|_{2}^{2}.
	$
	
	\begin{definition}[conjugate function] The conjugate function of $f: \mathbb{R}^{n} \rightarrow \mathbb{R}$ at $y \in\mathbb{R}^{n}$ is defined as
		$$
		f^*(y):=\sup _{x\in \mathbb{R}^{n}}\{\langle y, x\rangle-f(x)\}.
		$$
	\end{definition}
	%$ f^* $ is always convex whether $ f $ is convex or not.
	If $f$ is convex, it can be shown that \cite{rockafellar1997convex,beck2017first}
	$$
	z \in \partial f(x) \Leftrightarrow x \in \partial f^*(z).
	$$
	Besides, if $f$ is $\gamma$-strongly convex, then its conjugate function $f^*$ is differentiable and $\frac{1}{\gamma}$-smooth, i.e. for any $x, y \in \mathbb{R}^{n}$,
	% and for any $x, y \in \mathbb{R}^{n}$, the following inequality holds \cite{lorenz2014linearized}
	\begin{equation}\label{strongly-convex}
		f^*(y) \leq f^*(x)+\left\langle\nabla f^*(x), y-x\right\rangle+\frac{1}{2 \gamma}\|y-x\|_{2}^2.
	\end{equation}
	%which implies that $f^*(x)$ is $\frac{1}{\gamma}$-smooth convex.
	%For a convex function $h(x)$, the conjugate function of $\mu h(x)+\frac{1}{2}\|x\|_{2}^2$ is differentiable. Its gradient involves the proximal mapping of $h(x)$, and it holds that
	%$$
	%\nabla\left(\mu h(\cdot)+\frac{1}{2}\|\cdot\|_{2}^2\right)^*(x)=\operatorname{prox}_{\mu h}(x):=\underset{y \in \mathbb{R}^{n}}{\operatorname{argmin}}\left\{\mu h(y)+\frac{1}{2}\|x-y\|_{2}^2\right\}.
	%$$
	%It can be shown that the soft thresholding operator is in fact the proximal operator of the $\ell_1$-norm, i.e.,
	%$S_\mu(x)=\operatorname{prox}_{\mu \|\cdot\|_1}(x)$.
	
	\begin{definition}[Bregman distance]
		\label{bregma-dis}
		For a strictly convex function $f: \mathbb{R}^{n} \rightarrow \mathbb{R}$, the Bregman distance between $x$ and $y$ with respect to $f$ and $z \in \partial f(x)$ is defined as
		$$
		D_{f, z}(x, y):=f(y)-f(x)-\langle z,y-x\rangle.
		$$
	\end{definition}

	Since if $z \in \partial f(x)$, it holds that $\langle z, x\rangle=f(x)+f^*(z)$,  one has
	\begin{equation}\label{breg-dist}
		D_{f,z}(x,y)=f(y)+f^*(z)-\langle z, y\rangle.
	\end{equation}
	If $f$ is $\gamma$-strongly convex, it holds that
	$$
	D_{f, z}(x,y) \geq \frac{\gamma}{2}\|x-y\|_{2}^2 .
	$$
	
	\begin{definition}[restricted strong convexity, \cite{lai2013augmented,schopfer2016linear}] Let $f: \mathbb{R}^{n} \rightarrow \mathbb{R}$ be convex differentiable with a nonempty minimizer set $X_f$. The function $f$ is called restricted  $ \mu $-strongly convex on $\mathbb{R}^n$, if there exists $ \mu>0 $ such that for all $x \in \mathbb{R}^{n}$ the following inequality holds,
		$$
		\left\langle\nabla f\left(\operatorname{Proj}_{X_f}(x)\right)-\nabla f(x), \operatorname{Proj}_{X_f}(x)-x\right\rangle \geq \mu\left\|\operatorname{Proj}_{X_f}(x)-x\right\|_{2}^2,
		$$
		where $\operatorname{Proj}_{X_f}(x)$ denotes the orthogonal projection of $x$ onto $X_f$.
	\end{definition}
	
	\begin{definition}[strong admissibility] Let $f: \mathbb{R}^{n} \rightarrow \mathbb{R}$ be strongly convex. The function $f$ is called strongly admissible if the function $g(y):=f^*\left(A^\top y\right)-\left\langle b, y\right\rangle$ is restricted strongly convex on $\mathbb{R}^{n}$ for all $A \in \mathbb{R}^{m\times n}$ and $b \in \mathbb{R}^{m}$.
	\end{definition}

	As an example, the function $f(x)=\mu\|x\|_1+\frac{1}{2}\|x\|^2_2$ is strongly admissible (see \cite[Example 3.7]{chen2021regularized} and \cite[Lemma 4.6]{lai2013augmented}). We refer readers to \cite{schopfer2016linear} for more examples of strongly admissible
	functions.  The following property of strongly admissible functions is key for proving linear convergence rate of the algorithms.
	
	\begin{lemma}[\cite{chen2021regularized}, Lemma 3.6]\label{adm-key}
		Let $\widehat{x}$ be the solution of \eqref{main-prob}. If $f$ is strongly admissible, then there exists a constant $\nu>0$ such that
		\begin{equation}\label{nu-cons}
			D_{f, z}(x, \widehat{x}) \leq \frac{1}{\nu}\|A(x-\widehat{x})\|_{2}^2,
		\end{equation}
		for all $x \in \mathbb{R}^{n}$ and $z \in \partial f(x) \cap\text{Range}\left(A^\top\right)$.
	\end{lemma}
	
	We note that the constant $\nu$ in Lemma \ref{adm-key} depends on the matrix $A$ and the function $f$. For example, if we let the objective function $f(x)=\frac{1}{2}\|x\|^2_2$, then $\nu=2\sigma_{\min}^2(A)$ \cite{chen2021regularized}. For the case where $f(x)=\mu\|x\|_1+\frac{1}{2}\|x\|^2_2$, we refer readers to \cite[Lemma 7]{lai2013augmented} for an explicit computation of $\nu$. In general, it is hard to quantify $\nu$.
	
	%Lemma \ref{adm-key} The following property of strongly admissible functions is key for our analysis.
	
	%\subsection{Probability basics}
	
	\section{Stochastic dual coordinate descent}
	
	In this section, we examine the stochastic dual coordinate descent (SDCD) method for solving the linearly constrained optimization problem \eqref{main-prob}. As discussed in Section \ref{sec1}, at each iteration, we first draw a sampling matrix $ S_k $ from the probability space $(\Omega_k, \mathcal{F}_k, P_k)$. Then the iterate is updated with the following iteration strategy
	$$
	\begin{aligned}
		&z^{k+1}=z^k-\alpha_kA^\top S_kS_k^\top(A x^k-b),\\
		&x^{k+1}=\nabla f^*(z^{k+1}).
	\end{aligned}
	$$
	Here $\alpha_k$ is the stepsize defined by
	\begin{equation}
		\label{alp}
		\alpha_k=
		\left\{\begin{array}{ll}
			(2-\zeta)L_{\text{adap}}^{k,\gamma}, \;\; \text{if} \; S_k^\top (Ax^k-b)\neq 0;
			\\
			0, \qquad\qquad\;\;\;\;\; \text{otherwise},
		\end{array}
		\right.
	\end{equation}
	where $\zeta \in (0,2)$ is the relaxation parameter and
	\begin{equation}
		\label{Lk}
		L_{\text{adap}}^{k,\gamma}=\frac{\gamma\left\|S_{k}^\top(Ax^k-b)\right\|^2_2}{\left\|  A^\top S_kS_k^\top(Ax^k-b) \right\|_2^2}.
	\end{equation}
	The following lemma shows that this stepsize is well-defined.
	\begin{lemma}
		\label{lemma-non}
		Assume that the linear system $Ax=b$ is consistent. Then for any matrix $S \in \mathbb{R}^{m\times q}$ and any vector $\tilde{x} \in \mathbb{R}^{n}$, it holds that $A^\top SS^\top(A\tilde{x}-b) \neq 0$ if and only if $S^\top(A\tilde{x}-b) \neq 0$.
	\end{lemma}
	\begin{proof}
		Suppose that $A\widehat{x}=b$, then we know that $S^\top(A\tilde{x}-b) =0$ if and only if
		$$
		(A\tilde{x}-b)^\top S S^\top(A\tilde{x}-b)=(\tilde{x}-\widehat{x})^\top A^\top S S^\top A(\tilde{x}-\widehat{x})=0,
		$$
		which is equivalent to $A^\top S S^\top A(\tilde{x}-\widehat{x})=A^\top SS^\top(A\tilde{x}-b)=0$. This completes the proof of this lemma.
	\end{proof}
	
	% From Lemma \ref{lemma-non}, we know that $L_{\text{adap}}^{k,\gamma}$ is well-defined as
	Therefore, $S_k^\top (Ax^k-b)\neq 0$ implies that $A^\top S_k S_k^\top (Ax^k-b)\neq0$.
	%We note that it will become clear in the proof of Theorem \ref{thm-rdcd} how those are defined.
	We emphasize that when $S_k^\top (Ax^k-b)= 0$, then $A^\top S_kS_k^\top(Ax^k-b)=0$, and it holds that $x^{k+1}=x^k$ for any choices of $\alpha_k$. So we set $\alpha_k = 0$ to avoid extraneous computation.
	%We emphasize that for the case $S_k^\top (Ax^k-b)= 0$, we set $\alpha_k$ to be zero since any constant value  can be chosen. This is because we now have $A^\top S_kS_k^\top(Ax^k-b)=0$, and thus, regardless of the choice of $\alpha_k$, it holds that $x^{k+1}=x^k$.
	The stochastic dual coordinate descent (SDCD) method is formally described in Algorithm \ref{RDCD}.
	We make the following assumption on the probability spaces $\left\{(\Omega_k, \mathcal{F}_k, P_k)\right\}_{k\geq0}$.
	\begin{assumption}
		\label{Ass}
		%Let $\left\{(\Omega_k, \mathcal{F}_k, P_k)\right\}_{k\geq0}$  be the class of probability space used
		%The probability spaces $\left\{(\Omega_k, \mathcal{F}_k, P_k)\right\}_{k\geq0}$ satisfy that for any $k\geq0$, $\mathop{\mathbb{E}}\limits_{S_k \in \Omega_k} \left[S_k S_k^\top\right]$ is a positive definite matrix.
		Let $\left\{(\Omega_k, \mathcal{F}_k, P_k)\right\}_{k\geq0}$  be  probability spaces from which the sampling matrices are drawn. We assume that  for any $k\geq0$, $\mathop{\mathbb{E}}_{S_k\in\Omega_k} \left[S_k S_k^\top\right]$ is a positive definite matrix.
	\end{assumption}		
	
	\begin{algorithm}[htpb]
		\caption{ Stochastic dual coordinate descent (SDCD)\label{RDCD}}
		\begin{algorithmic}
			\Require
			$A\in \mathbb{R}^{m\times n}$, $b\in \mathbb{R}^m$, probability spaces $\{(\Omega_k, \mathcal{F}_k, P_k)\}_{k\geq0}$, $\zeta\in(0,2)$, $k=0$ and initial points $z^0\in\text{Range}(A^\top)$, $x^0=\nabla f^*(z^{0})$.
			\begin{enumerate}
				\item[1:] Randomly select a sampling matrix $S_{k}\in\Omega_k$.
				%Draw sample $S_k\thicksim\mathcal{D}_k$.
				\item[2:] Compute the stepsize $\alpha_k$ in \eqref{alp}.
				\item[3:] Compute
				$$
				z^{k+1}=z^k-\alpha_k A^\top S_{k} S_{k}^\top(Ax^k-b).
				$$
				\item[4:] Compute
				$$
				x^{k+1}=\nabla f^*(z^{k+1}).
				$$
				\item[5:] If the stopping rule is satisfied, stop and go to output. Otherwise, set $k=k+1$ and go to Step $1$.
			\end{enumerate}
			
			\Ensure
			The approximate solution $x^k$.
		\end{algorithmic}
	\end{algorithm}

		We now consider the connections between the SDCD framework and other methods.
	\begin{REMK}\label{remark-xie-0429}
		When the probability spaces are fixed, i.e. $(\Omega_k, \mathcal{F}_k, P_k) \equiv  (\Omega, \mathcal{F}, P) $,  Algorithm \ref{RDCD} can be regarded as a kind of the stochastic mirror descent (SMD) method using  mirror stochastic Polyak stepsize. Consider the following optimization problem
		\begin{equation}\label{reformulate-ls}
			\mathop{\min}\limits_{x \in \mathbb{R}^n} \mathop{\mathbb{E}}\limits_{S\in\Omega}\left[ h_S(x)\right],
		\end{equation}
		where $h_S(x):=\frac{1}{2} \left\|S^\top(Ax-b)\right\|^2_2$.
		%with $S$ being a random variable in $(\Omega_k, \mathcal{F}_k, P_k)$.
		%With $S_k$ being drawn from the sample space $\Omega_k$,
		In fact,  the problem \eqref{reformulate-ls} can be viewed as a stochastic reformulation of solving the linear system $Ax=b$, and Assumption \ref{Ass} guarantees that the stochastic reformulation \eqref{reformulate-ls} is \textit{exact}, i.e. the set of minimizers of the problem \eqref{reformulate-ls} is identical to the set of solutions of the linear system $Ax=b$; See \cite[Lemma 2.2]{zeng2024adaptive}.
		
		We employ the SMD method \eqref{SMD-iter} to solve \eqref{reformulate-ls}
		$$
		x^{k+1}=\arg\min_{x\in\mathbb{R}^n}\left\{t_k\left\langle \nabla h_{S_k}(x^k),x-x^k\right\rangle+D_{f, z^k}(x^k, x)\right\}, \ z^k\in\partial f(x^k),
		$$
		which yields the following update
		$$
		\begin{aligned}
			z^{k+1}&=z^k-t_k\nabla h_{S_k}(x^k)=z^k-t_k A^\top S_{k} S_{k}^\top(Ax^k-b),\\
			x^{k+1}&=\nabla f^*(z^{k+1}).
		\end{aligned}
		$$
		This is exactly the SDCD method. Let $\widehat{x}$ be the solution of \eqref{main-prob}, then  $b=A\widehat{x}$ and   $$\hat{h}_{S}:=\inf_{x \in \mathbb{R}^{n}} h_{S}(x)=\inf_{x \in \mathbb{R}^{n}} \frac{1}{2} \left\|S^\top(Ax-b)\right\|^2_2=\left\|S^\top(A\widehat{x}-b)\right\|^2_2=0$$ for all $S \in \Omega$. Hence, the mirror stochastic Polyak stepsize \eqref{polyak-stepsize} is simply
		$$
		t_k=\frac{\gamma h_{S_k}(x^k)}{c\|\nabla h_{S_k}(x^k)\|^2_2}=\frac{\gamma\left\|S_{k}^\top(Ax^k-b)\right\|^2_2}{c\left\|  A^\top S_kS_k^\top(Ax^k-b) \right\|_2^2}.
		$$
		%Therefore, Algorithm \ref{RDCD} can be viewed as a kind of the SMD with  mirror stochastic Polyak stepsize.
		%Now we have ready derive a connection between  Algorithm \ref{RDCD} and the SMD with  mirror stochastic Polyak stepsize.
		%Now we have ready derive a connection between  Algorithm \ref{RDCD} and the SMD with  mirror stochastic Polyak stepsize.
		Now we have arrived at the connection between Algorithm \ref{RDCD} and the SMD with  mirror stochastic Polyak stepsize.
	\end{REMK}
	
	\begin{REMK}\label{remark-xie-0607-1}
		%We consider the randomized average block Kaczmarz (RABK) method proposed by Necoara \cite{Nec19}.
		Consider the following iteration
		\begin{equation}\label{BKM-g}
			\begin{aligned}
				z^{k+1}&=z^k-\alpha_k\bigg(\sum\limits_{i\in \mathcal{J} _k}\omega^k_i\frac{a_{i}^\top x^k-b_i}{\|a_{i}\|^2_2}a_{i}\bigg),\\
				x^{k+1}&=\nabla f^*(z^{k+1}),
			\end{aligned}
		\end{equation}
		where the weights $\omega^k_i\in[0,1]$ such that $\sum\limits_{i\in \mathcal{J} _k}\omega^k_i=1$,  $\mathcal{J} _k\subseteq[m]$, and the stepsize $\alpha_k>0$. We note that the iteration scheme \eqref{BKM-g} can be viewed as a special case of the SDCD method.
		Indeed, let $I_{\mathcal{J}_k}$ denote a column concatenation of the columns of the $m\times m$ identity matrix $I$ indexed by $\mathcal{J} _k$, and the diagonal matrix  $D_{\mathcal{J} _k}:=\mbox{diag}(\sqrt{\omega^k_i}/\|a_{i}\|_2,i\in \mathcal{J} _k)$.
		Then the iteration scheme \eqref{BKM-g} can be rewritten as
		%\begin{equation}\label{BKMs}
		$$
		\begin{aligned}
			z^{k+1}&=z^k-\alpha_k A^\top S_k S_k^\top (Ax^k-b),\\
			x^{k+1}&=\nabla f^*(z^{k+1}),
		\end{aligned}$$
		%\end{equation}
		where $S_k = I_{\mathcal{J} _k}D_{\mathcal{J} _k}$, which can be viewed as a sampling matrix selected from  a certain probability space $(\Omega_k, \mathcal{F}_k, P_k)$. Finally, let us discuss some special cases of the iteration scheme \eqref{BKM-g}.
		\begin{itemize}
			\item[(1)] If $f(x)=\frac{1}{2}\|x\|^2_2$, then $f^*(z)=\frac{1}{2}\|z\|^2_2$ and hence the iteration scheme \eqref{BKM-g}
			reduces to the randomized average block Kaczmarz (RABK) method proposed by Necoara \cite{Nec19}.
			%Furthermore,
			\item[(2)] If $f(x)=\mu\|x\|_1+\frac{1}{2}\|x\|^2_2$, then the iteration scheme \eqref{BKM-g} derives a new type of the RSKA method, where instead of using a constant stepsize as \eqref{rska}, an adaptive stepsize is employed.
			%the iteration scheme \eqref{BKM-g}, unlike the iteration scheme \eqref{rska} where a constant stepsize is employed,  derives a new type of RSKA method.
			Furthermore, if the sample spaces $\Omega_k=\{I\}$ for any $k\geq0$, then \eqref{BKM-g} reduces to the linearized Bregman method \cite{cai2009convergence,cai2009linearized}.
			\item[(3)] If $\mathcal{J}_k$ is a singleton, then the iteration scheme \eqref{BKM-g} reduces to  the randomized regularized Kaczmarz  method proposed in \cite{chen2021regularized}.
		\end{itemize}
	\end{REMK}
	
	\begin{REMK}
			
			We consider the stochastic dual coordinate ascent (SDCA) method by Shalev-Schwartz et al. \cite{shalev2013stochastic} for the regularized loss minimization problem 
			$$
				\min_{x \in \mathbb{R}^{n}} f(x) = \frac{1}{m} \sum_{i=1}^{m} \phi_{i}(a_i^\top x) + \frac{\gamma}{2} \|x\|_{2}^{2},
			$$
			where \(\phi_i\) are convex functions, \(a_i^\top\) are rows of \(A\), and \(\gamma > 0\). Its dual problem is
			\begin{equation} \label{dual-reg-20}
				\min_{\lambda \in \mathbb{R}^{m}} g(\lambda) = \frac{1}{m} \sum_{i=1}^{m} \phi_{i}^{*}(-\lambda_{i}) + \frac{\gamma}{2} \left\| \frac{1}{\gamma m} A^\top \lambda \right\|_{2}^{2},
			\end{equation}
			with optimal solutions satisfying \(\widehat{x} = \frac{1}{\gamma m} A^\top \widehat{\lambda}\). The SDCA iteration \cite{shalev2013stochastic} is
			\begin{equation} \label{SDCA}
				\begin{aligned}
					\triangle \lambda^*_{i_{k}} &= \arg\min_{\triangle \lambda} \ \frac{1}{m} \phi_{i_k}^{*}\left( -(\lambda^{k})_{i_{k}} + \triangle \lambda \right) + \frac{\gamma}{2} \left\| x^k - \frac{\triangle \lambda}{\gamma m} a_{i_k} \right\|_{2}^{2}, \\
					\lambda^{k+1} &= \lambda^{k} - \triangle \lambda^*_{i_{k}} e_{i_{k}}, \\
					x^{k+1} &= x^{k} - \frac{1}{\gamma m} \triangle \lambda^*_{i_{k}} a_{i_{k}}.
				\end{aligned}
			\end{equation}
			Applying the SDCD method \eqref{iter-org-sdcd} with $S_k=e_{i_k}$ to the dual problem \eqref{dual-reg-20} yields
			\begin{equation} \label{SDCD-lg}
				\begin{aligned}
					\triangle \widetilde{\lambda}_{i_{k}} &= -\frac{1}{m} \nabla \phi^*_{i_k}( -(\lambda^k)_{i_k} ) + \frac{1}{m} a^\top_{i_k} x^k, \\
					\lambda^{k+1} &= \lambda^{k} - \alpha_k \triangle \widetilde{\lambda}_{i_{k}} e_{i_k}, \\
					x^{k+1} &= x^k - \frac{\alpha_k}{\gamma m} \triangle \widetilde{\lambda}_{i_{k}} a_{i_k},
				\end{aligned}
			\end{equation}
			where \(x^k = \frac{1}{\gamma m} A^\top \lambda^k\). Comparing the update for \(\triangle \widetilde{\lambda}_{i_{k}}\) in \eqref{SDCD-lg} with the exact minimization in \eqref{SDCA}, we observe that the former can be viewed as a single gradient descent step for solving the subproblem in the latter. However, we note that SDCD offers greater flexibility through its choice of the sampling matrix \(S_k\), leading to more versatile algorithmic variants.
	\end{REMK}

	%\begin{remark}
	%Furthermore, the versatility of our framework and the general convergence theorem (Theorem \ref{thm-rdcd}) enable one to customize the probability spaces $\{(\Omega_k, \mathcal{F}_k, P_k)\}_{k\geq0}$ to other specific problems. For example, random sparse matrices or sparse Rademacher matrices may be suitable for a particular set of problems.
	Finally, we note that the flexibility of our framework and the general convergence theorem (Theorem \ref{thm-rdcd}) allow for customization of the probability spaces $\{(\Omega_k, \mathcal{F}_k, P_k)\}_{k\geq0}$ to address other specific problems. For instance, random sparse matrices or sparse Rademacher matrices may be appropriate for a particular set of problems.
	%\end{remark}
	
	\subsection{Convergence analysis}
	\label{sect-31}
	To establish the convergence of Algorithm \ref{RDCD}, % let us first introduce some notations and lemmas.
	the following lemma is necessary.
	\begin{lemma}[\cite{zeng2024adaptive}, Lemma 2.5]\label{positive}
		Let $S\in\mathbb{R}^{m\times q}$ be a real-valued random variable defined on a probability space $(\Omega,\mathcal{F},P)$. Suppose that
		$
		D=\mathbb{E}\left[SS^\top\right]
		$
		is a positive definite matrix.  Then
		$$
		\mathbb{E}\left[\frac{SS^\top}{\|S\|^2_2}\right]
		$$
		is also positive definite, here we define $\frac{0}{0}=0$.
	\end{lemma}
	To state conveniently, we define
	\begin{equation}
		\label{matrix-H}
		H_k=
		\left\{\begin{array}{ll}
			\mathbb{E}_{S \in \Omega_k}[SS^\top], \quad\;\;\, \text{if} \; \Omega_k \; \text{is} \; \text{bounded};
			\\
			\mathbb{E}_{S \in \Omega_k}\left[\frac{SS^\top}{\|S\|_2^2}\right], \;\;\, \ \text{otherwise},
		\end{array}
		\right.
	\end{equation}
	and
	\begin{equation}
		\label{lambda-max}
		\lambda_{\max}^{(k)}=
		\left\{\begin{array}{ll}
			\sup\limits_{S\in \Omega_k} \lambda_{\max}\left(A^\top SS^\top A\right), \;\;\,\ \text{if} \; \Omega_k \; \text{is} \; \text{bounded};
			\\
			\sup\limits_{S\in \Omega_k} \lambda_{\max}\left(\frac{A^\top SS^\top A}{\|S\|_2^2}\right), \quad\;\ \text{otherwise}.
		\end{array}
		\right.
	\end{equation}
	It follows from Assumption \ref{Ass} and Lemma \ref{positive} that $H_k$ in \eqref{matrix-H} is well-defined and positive definite.
	
	At the $k$-th iteration, we consider the product probability space $(\mathop{\Pi}_{i=0}^{k} \Omega_{i}, \mathop{\otimes}_{i=0}^{k} \mathcal{F}_{i}, \tilde{P})$, where $\otimes$ denotes the product of $\sigma$-algebras and $\tilde{P}$ denotes the corresponding product measure \cite[Section 5]{athreya2006measure}. Let $\mathcal{B}_{k}:=(S_{0}, S_{1}, \cdots, S_{k-1})$ be a random variable in this probability space, where $\mathcal{B}_{0}$ denotes the empty sequence. We denote the conditional expectation with respect to $\mathcal{B}_{k}$ as
	$$
	\mathbb{E}_{k} [\cdot]:=\mathbb{E}[\cdot | \mathcal{B}_{k}].
	$$
	We have the following convergence result for Algorithm \ref{RDCD}. The detailed proof is provided in the Appendix \ref{secA1}.
	
	\begin{thm}
		\label{thm-rdcd}
		Let $f$ be $\gamma$-strongly convex and strongly admissible. Suppose that the probability spaces $\{(\Omega_k, \mathcal{F}_k, P_k)\}_{k\geq 0}$ satisfy Assumption \ref{Ass}. Let $\{x^k\}_{k\geq0}$ and $\{z^k\}_{k\geq0}$ be the sequences of iterates  generated by Algorithm \ref{RDCD}. Then
		%For any given linear system $Ax=b$, let $\{x^k\}_{k=0}^{\infty}$ be the iteration sequence generated by Algorithm \ref{rrbk} with $x^0=0$. Then
		$$
		\mathbb{E}{_k}\left[D_{f,z^{k+1}}(x^{k+1},\widehat{x})\right]\leq \left(1-\frac{\gamma\zeta(2-\zeta)\nu\lambda_{\min}\left(H_k\right) }{2\lambda_{\max}^{(k)}} \right)D_{f,z^{k}}(x^{k},\widehat{x}),
		$$
		where  $\widehat{x}$ is the solution of \eqref{main-prob}, $\nu$, $H_k$, and $\lambda_{\max}^{(k)}$ are given by \eqref{nu-cons}, \eqref{matrix-H}, and \eqref{lambda-max}, respectively. Furthermore, it holds that
				$$
						\mathop{\mathbb{E}} \left[\| x^{k}-\hat{x}\|_2^2\right] \leq \frac{2 D_{f,z^{0}}(x^{0},\widehat{x})}{\gamma} \prod \limits_{i=0}^{k-1} \left(1-\frac{\gamma\zeta(2-\zeta)\nu\lambda_{\min}\left(H_{i}\right) }{2\lambda_{\max}^{(i)}} \right).
				$$
	\end{thm}

	\begin{REMK}
	If we choose $$\text{Prob}\left(S_k=\frac{e_i}{\|a_i\|_2}\right)=\frac{\|a_i\|^2_2}{\|A\|_F^2},$$
	then Theorem \ref{thm-rdcd} recovers the convergence result for the regularized randomized Kaczmarz proposed in \cite[Theorem 3.9]{chen2021regularized}. Particularly, if $f(x)=\frac{1}{2}\|x\|^2_2$, then Theorem \ref{thm-rdcd} recovers the convergence result for the randomized Kaczmarz method.
	\end{REMK}
	
		\begin{REMK}\label{remark-sampling}
			We analyze the effect of the sample size on the convergence rate of SDCD. In particular, we consider a partition-based sampling strategy, which has been extensively studied in the literature \cite{tropp2009column, Nec19, necoara2022stochastic, xie2021subset}. Let \(\varpi\) be a uniform random permutation on \([m]\). The index set \([m]\) is partitioned into blocks \(\mathcal{I}_1, \dots, \mathcal{I}_t\) as follows
			\begin{equation}\label{partition-sample}
			\begin{aligned}
				\mathcal{I}_i &= \left\{ \varpi(k): k = (i-1)\tau+1, (i-1)\tau+2, \ldots, i\tau \right\}, \quad i = 1, 2, \ldots, t-1, \\
				\mathcal{I}_t &= \left\{ \varpi(k): k = (t-1)\tau+1, (t-1)\tau+2, \ldots, m \right\}, \qquad |\mathcal{I}_t| \leq \tau,
			\end{aligned}
			\end{equation}
			where \(\tau\) is the block size.
			At each iteration, we randomly select a block index \(i_k \in [t]\) with probability 
			$
			\operatorname{Prob}(i_k = i) = \|A_{\mathcal{I}_i}\|_F^2/\|A\|_F^2,
			$
			and set the sampling matrix as 
			$
			S_k = (I_{\mathcal{I}_{i_k}})^\top/\|A_{\mathcal{I}_{i_k}}\|_F.
			$
			Under this strategy, the parameters in Theorem \ref{thm-rdcd} simplify to 
			$
			H_k = \frac{1}{\|A\|_F^2} I$ and $ \lambda_{\max}^{(k)} = \max_{i \in [t]} \frac{\|A_{\mathcal{I}_i}\|_2^2}{\|A_{\mathcal{I}_i}\|_F^2}.
			$
			Then,  SDCD with \(\zeta = 1\) satisfies the following convergence bound
			$$
				\mathbb{E} [\| x^k - \hat{x} \|_2^2] \leq \frac{2 D_{f,z^0}(x^0,\hat{x})}{\gamma} 
				\left(1 - \frac{\gamma \nu}{2 \|A\|_F^2 \cdot \max_{j \in [t]} \frac{\|A_{\mathcal{I}_j}\|_2^2}{\|A_{\mathcal{I}_j}\|_F^2}} \right)^k.
			$$
			We now compare two extreme cases: \(\tau = 1\) and \(\tau = m\). The corresponding convergence factors are \(1 - \frac{\gamma \nu}{2\|A\|_F^2}\) and \(1 - \frac{\gamma \nu}{2\|A\|_2^2}\), respectively. Using the inequality \(1 - \iota \leq e^{-\iota}\) for any \(\iota \in (0,1)\), SDCD with \(\tau = 1\) and \(\tau = m\) requires
			\[
			\mathcal{O}\left(\frac{\|A\|_F^2}{\gamma \nu} \log\left(\frac{1}{\varepsilon}\right)\right) \quad \text{and} \quad \mathcal{O}\left(\frac{\|A\|_2^2}{\gamma \nu} \log\left(\frac{1}{\varepsilon}\right)\right)
			\]
			iterations, respectively, to achieve an accuracy of \(\varepsilon\) in terms of the expected error norm. Since updating \(z^k\) with \(\tau = m\) requires approximately \(m\) times more computation than with \(\tau = 1\), a fair comparison should be made between \(\mathcal{O}\left(\frac{\|A\|_F^2}{\gamma \nu} \log(1/\varepsilon)\right)\) and \(\mathcal{O}\left(\frac{m\|A\|_2^2}{\gamma \nu} \log(1/\varepsilon)\right)\). Given that \(\|A\|_F^2 \leq m \|A\|_2^2\), SDCD with \(\tau = 1\) converges faster in theory than with \(\tau = m\). 
		Now consider a special case where the rows within each block \(A_{\mathcal{I}_i}\) are orthonormal, i.e.,
		\[
		\langle a_\ell, a_j \rangle = 
		\begin{cases}
			1, & \text{if } \ell = j \in \mathcal{I}_i, \\
			0, & \text{if } \ell \neq j \in \mathcal{I}_i.
		\end{cases}
		\]
		In this case, the convergence factor becomes \(1 - \frac{\gamma \nu \tau}{2m}\), and the corresponding number of iterations to achieve \(\varepsilon\)-accuracy is \(\mathcal{O}\left(\frac{m}{\gamma \nu \tau} \log(1/\varepsilon)\right)\). Since each iteration with block size \(\tau\) requires roughly \(\tau\) times more computation than with \(\tau=1\), a fair comparison of the total computational cost yields \(\mathcal{O}\left(\frac{m}{\gamma \nu} \log(1/\varepsilon)\right)\), which is independent of the block size $\tau$. This indicates that SDCD with \(\tau = 1\) performs comparably to larger block sizes in this orthonormal setting.

			The above analysis also applies to uniform sampling, where \(\tau\) distinct indices are selected uniformly at random from \([m]\) to form \(\mathcal{I}\), with \(|\mathcal{I}| = \tau\) in each sampling.
			However, in practice, parallelization techniques can be used to accelerate SDCD in terms of total runtime. This observation is also supported by the numerical results in Section~\ref{section5-1}.
		\end{REMK}

	\section{Acceleration by adaptive heavy-ball momentum}
	
	This section aims to enrich the SDCD method with adaptive heavy-ball momentum.
	It was originally proposed by Polyak \cite{polyak1964some}, where a (heavy ball) momentum term is introduced to improve the convergence rate of the gradient descent method. To solve the problem \eqref{dual-prob}, the iteration scheme of the proposed adaptive SDCD (ASDCD) method reads as
	$$
	\lambda^{k+1}=\lambda^{k}-\alpha_{k} S_kS_k^\top\nabla g\big(\lambda^{k}\big)+\beta_k\big(\lambda^{k}-\lambda^{k-1}\big),
	$$
	where $S_k$ is randomly chosen from $\Omega_k$, $\alpha_k$ is the stepsize, and $\beta_k$ is the momentum parameter. Ideally, we would like to  choose  $\alpha_k$ and $\beta_k$ to obtain a sufficient reduction of the objective function $g(\lambda)$, and hence we may consider the following optimization problem
	\begin{equation}\label{opt-prob}
		%{\boxed{
				\begin{aligned}
					\min\limits_{\lambda}& \ \ g(\lambda)=f^*(A^\top\lambda)- \langle b,\lambda\rangle\\
					\text{subject to}& \ \ \lambda=\lambda^{k}-\alpha S_kS_k^\top\nabla g\big(\lambda^{k}\big)+\beta\big(\lambda^{k}-\lambda^{k-1}\big), \ \alpha,\beta\in\mathbb{R}.
				\end{aligned}
				%}}
	\end{equation}
	However, finding the optimal vaules of $\alpha$ and $\beta$  may be difficult in practice. Actually, we can use the majorization technique \cite{li2016majorized,chen2017efficient} to find an approximate solution of the optimization problem \eqref{opt-prob}.
	To state conveniently, we set $z:=A^\top\lambda$, $x:=\nabla f^*(A^\top\lambda)=\nabla f^*(z)$, and
	$$
	d^k:=A^\top S_{k} S_{k}^\top\nabla g\big(\lambda^{k}\big)=A^\top S_{k} S_{k}^\top(Ax^k-b).%\sum\limits_{i\in J_k}\omega_{i,k}\frac{\langle a_i,x^k\rangle-b_i}{\|a_i\|^2_2}a_i
	$$
	Let $\widehat{x}$ be the solution of \eqref{main-prob}, then  $b=A\widehat{x}$. For the objective function in \eqref{opt-prob}, we have
	\begin{equation}\label{MT-xie}
		\begin{aligned}
			g(\lambda)=&f^*(z)-\langle A\widehat{x},\lambda\rangle\\
			=&f^*(z)-\langle z,\widehat{x}\rangle\\
			\leq& f^*(z^k)+\langle \nabla f^*(z^k), z-z^k\rangle+\frac{1}{2\gamma}\|z-z^k\|^2_2-\langle z,\widehat{x}\rangle
			\\
			=& f^*(z^k)-\left\langle x^k,\alpha A^\top S_{k} S_{k}^\top(Ax^k-b)-\beta(z^k-z^{k-1}) \right\rangle
			\\
			&+\frac{1}{2\gamma}\left\|\alpha A^\top S_{k} S_{k}^\top(Ax^k-b)-\beta(z^k-z^{k-1})\right\|^2_2\\
			&-
			\left\langle z^k-\alpha A^\top S_{k} S_{k}^\top(Ax^k-b)+\beta(z^k-z^{k-1}),\widehat{x}\right\rangle\\
			=&f^*(z^k)-\langle z^k,\widehat{x}\rangle+\frac{1}{2\gamma}\left\|\alpha d^k-\beta(z^k-z^{k-1})\right\|^2_2
			\\&- \left\langle x^k-\widehat{x},\alpha d^k -\beta(z^k-z^{k-1}) \right\rangle,
		\end{aligned}
	\end{equation}
	where the first inequality follows from \eqref{strongly-convex}. Let
	\begin{equation} \nonumber
		\begin{aligned}
			h^k(\alpha,\beta):=& \frac{1}{2\gamma}\left\|\alpha d^k
			-\beta(z^k-z^{k-1})\right\|^2_2
			- \left\langle x^k-\widehat{x},\alpha d^k -\beta(z^k-z^{k-1}) \right\rangle.
		\end{aligned}
	\end{equation}
	We now consider solving the following majorized optimization problem of \eqref{opt-prob}
	\begin{equation}\label{xie-mo}
		\min_{\alpha,\beta\in\mathbb{R}} h^k(\alpha,\beta).
	\end{equation}
	%whose solution is given by
	By taking the derivative of \eqref{xie-mo} with respect to $\alpha$ and $\beta$, we obtain
			\begin{equation}\nonumber
				\left\{
				\begin{array}{ll}
					\alpha \|d^{k}\|^{2}_{2}-\beta \langle d^{k}, z^{k}-z^{k-1} \rangle=\gamma \langle d^{k}, x^{k}-\widehat{x} \rangle, \\[0.3cm]
					\alpha \langle d^{k}, z^{k}-z^{k-1} \rangle-\beta \|z^{k}-z^{k-1}\|_{2}^{2}=\gamma \langle z^k-z^{k-1},x^k-\widehat{x}\rangle.
				\end{array}
				\right.
			\end{equation}
			Therefore, the minimizers of \eqref{xie-mo} are given by
	\begin{equation}\label{adap-s-m}
		\left\{
		\begin{array}{ll}
			\alpha_k=\gamma\frac{\langle d^{k},x^k-\widehat{x}\rangle\|z^k-z^{k-1}\|^2_2-\langle d^k,z^k-z^{k-1}\rangle\langle z^k-z^{k-1},x^k-\widehat{x}\rangle}
			{\|d^k\|^2_2\|z^k-z^{k-1}\|^2_2-\langle d^k,z^k-z^{k-1}\rangle^2}, \\[0.3cm]
			\beta_k=\gamma\frac{-\|d^{k}\|^2_2\langle x^k-\widehat{x},z^k-z^{k-1}\rangle+\langle d^k,z^k-z^{k-1}\rangle \langle d^k,x^k-\widehat{x}\rangle}
			{\|d^k\|^2_2\|z^k-z^{k-1}\|^2_2-\langle d^k,z^k-z^{k-1}\rangle^2},
		\end{array}
		\right.
	\end{equation}
	provided that $\|d^k\|^2_2\|z^k-z^{k-1}\|^2_2-\langle d^k,z^k-z^{k-1}\rangle^2\neq0$.
	We can see that in order to compute $\alpha_k$ and $\beta_k$, we need to calculate $\langle d^{k},\widehat{x}\rangle$ and $\langle z^k-z^{k-1}, \widehat{x}\rangle$. By the definition of $d^k$, we know that
	$$\langle d^k, \widehat{x}\rangle=\langle S_{k} S_{k}^\top(Ax^k-b), A\widehat{x}\rangle=\langle S_{k} S_{k}^\top (Ax^k-b), b\rangle$$
	is calculable. Next, we show that we can compute $\langle z^k-z^{k-1},\widehat{x}\rangle$ by an incremental method.
	%However, such expressions of $\alpha_k$ and $\beta_k$ is incalculable, since they need to the compute $d^\top_{k}\widehat{x}$ and $(z^k-z^{k-1})^\top \widehat{x}$ while the solution $\widehat{x}$ of \eqref{main-prob} is unknown.
	From \eqref{opt-prob} and the definition of $z^k$, we know that
	$$
	z^k-z^{k-1}=-\alpha_{k-1}d^{k-1}+\beta_{k-1}(z^{k-1}-z^{k-2}).
	$$
	Hence, we have
	$$
	\begin{aligned}
		\langle z^k-z^{k-1},\widehat{x}\rangle&=-\alpha_{k-1}\langle d^{k-1}, \widehat{x}\rangle+\beta_{k-1}\langle z^{k-1}-z^{k-2}, \widehat{x}\rangle
		\\
		&=-\alpha_{k-1}\langle S_{k-1} S_{k-1}^\top (Ax^{k-1}-b), b\rangle+\beta_{k-1}\langle z^{k-1}-z^{k-2}, \widehat{x}\rangle,
	\end{aligned}
	$$
	which means that if the value of $\langle z^{k-1}-z^{k-2},\widehat{x}\rangle$ is available, then we are able to compute $\langle z^k-z^{k-1}, \widehat{x}\rangle$. % as $\alpha_{k-1}$ and $\beta_{k-1}$ have been already calculated in the previous step.
	Let $\rho_k:=\langle z^k-z^{k-1}, \widehat{x}\rangle$.
	%We can calculate $\tau_1$ as $\tau_1=\xi^0^\top b$ by choosing $z^1-z^0\in\text{Range}(A)$, i.e., $z^1-z^0=A\xi^0$ with $\xi^0\in\mathbb{R}^{n}$ being initially given.
	If we choose $z^1-z^0\in\text{Range}(A^\top)$, i.e. $z^1-z^0=A^\top\xi^0$ with an initialized  $\xi^0\in\mathbb{R}^{m}$, then $\rho_1=\langle\xi^0,A\widehat{x}\rangle=\langle\xi^0, b\rangle$ is calculable. Consequently, using the recursive relationship
	$$
	\rho_k=-\alpha_{k-1}\langle S_{k-1} S_{k-1}^\top(Ax^{k-1}-b), b\rangle+\beta_{k-1}\rho_{k-1},
	$$
	we know that  $\{\rho_k\}_{k\geq 1}$ is available.
	Thus, \eqref{adap-s-m}  can be computed by
	%
	%given the initial points $z^0$, $\xi^0$, and $\rho_1=\xi^0^\top b$, we can use the following iteration scheme to compute $\alpha_k$ and $\beta_k$:
	%\eqref{adap-s-m} can be rewritten as
	\begin{equation}\label{adap-s-mm}
		\left\{
		\begin{array}{ll}
			\alpha_k=\gamma\frac{\|S^\top_k(Ax^k-b)\|^2_2\|z^k-z^{k-1}\|^2_2-\langle d^k,z^k-z^{k-1}\rangle\left(\langle z^k-z^{k-1}, x^k\rangle-\rho_k\right)}
			{\|d^k\|^2_2\|z^k-z^{k-1}\|^2_2-\langle d^k,z^k-z^{k-1}\rangle^2}, \\[0.3cm]
			\beta_k=\gamma\frac{-\|d^{k}\|^2_2\left(\langle z^k-z^{k-1}, x^k\rangle-\rho_k\right)+\langle d^k,z^k-z^{k-1}\rangle\|S^\top_k(Ax^k-b)\|^2_2}
			{\|d^k\|^2_2\|z^k-z^{k-1}\|^2_2-\langle d^k,z^k-z^{k-1}\rangle^2}.
			%\\
			%\rho_{k+1}=-\alpha_{k}(Ax^k-b)^\top S_{k} S_{k}^\top b+\beta_{k}\rho_{k}.
		\end{array}
		\right.
	\end{equation}

	%Since $d_k^\top \widehat{x}=(Ax^k-b)^\top S_{k} S_{k}^\top A\widehat{x}=(Ax^k-b)^\top S_{k} S_{k}^\top b$.
	%Since the convex function $f^*$ is continuously differentiable with Lipschitz continuous gradient $\frac{1}{\gamma}$
	%The ideal choices of the parameters $\alpha_k$ and $\beta_k$ would be the global minimizers of the following optimization problem
	%$$
	%\min g(\lambda)
	%$$
	
	%To improve the rate of convergence, Polyak proposed to modify GD by the introduction of a (heavy ball) momentum term, $\beta_k\left(\lambda^{k}-\lambda^{k-1}\right)$. This leads to the gradient descent method with momentum (mGD), popularly known as the heavy ball method:
	%$$
	%\lambda^{k+1}=\lambda^{k}-\alpha_{k} \nabla g\big(\lambda^{k}\big)+\beta_k\big(\lambda^{k}-\lambda^{k-1}\big) .
	%$$
	%Polyak \cite{polyak1964some} proved that for twice continuously differentiable objective function $g(\lambda)$ with $\mu$ strong convexity constant and $L$-Lipschitz gradient, mGD achieves a local accelerated linear convergence rate of $O\left(\log (\varepsilon^{-1})\sqrt{L / \mu}\right)$ (with the appropriate choice of the stepsize parameters $\alpha_{k}$ and momentum parameters $\lambda_k$).
	
	Now we are ready to present the ASDCD method, which is formally described in Algorithm \ref{amRDCD}. We note that unlike the ASHBM method \cite[Algorithm 4.1]{zeng2024adaptive}, which requires a specific condition on its parameters, Algorithm \ref{amRDCD} does not impose any restrictions on \(S_k\) to ensure \(\|d^k\|^2_2\|z^k-z^{k-1}\|^2_2 - \langle d^k, z^k-z^{k-1} \rangle^2 \neq 0\).
	\begin{algorithm}[htpb]
		\caption{Adaptive SDCD (ASDCD)\label{amRDCD}}
		\begin{algorithmic}
			\Require
			$A\in \mathbb{R}^{m\times n}$, $b\in \mathbb{R}^m$, probability spaces $\{(\Omega_k, \mathcal{F}_k, P_k)\}_{k\geq1}$, $k=1$ and initial points $\xi^0\in\mathbb{R}^m$, $z^0\in\text{Range}(A^\top)$. Set $z^1=z^0+A^\top\xi^0$, $\rho_1=\langle\xi^0, b\rangle$ , $x^0=\nabla f^*(z^{0})$ and $x^1=\nabla f^*(z^{1})$.
			\begin{enumerate}
				\item[1:] Randomly select a sampling matrix $S_{k}\in\Omega_k$.
				%Draw sample $S_k\thicksim\mathcal{D}_k$.
				\item[2:] Compute $d^k=A^\top S_{k} S_{k}^\top(Ax^k-b)$.
				\item[3:]  If
				$\|d^k\|^2_2\|z^k-z^{k-1}\|^2_2-\langle d^k,z^k-z^{k-1}\rangle^2=0$
				
				\qquad Compute $\alpha_k$ in \eqref{alp} with $\zeta=1$ and set  $\beta_k=0$.
				
				%\qquad Compute $
				%z^{k+1}=z^k-\alpha_k d_k.
				%$
				
				Otherwise,
				
				\qquad Compute the parameters $\alpha_k$ and $\beta_k$ in \eqref{adap-s-mm}.

				\item[4:] Compute
				$$
				\begin{aligned}
					z^{k+1}&=z^k-\alpha_k d^k+\beta_k(z^k-z^{k-1}),\\
					\rho_{k+1}&=-\alpha_{k}\langle S_{k} S_{k}^\top(Ax^k-b), b\rangle+\beta_{k}\rho_{k}.
				\end{aligned}
				$$
				\item[5:] Compute
				$$
				x^{k+1}=\nabla f^*(z^{k+1}).
				$$
				\item[6:] If the stopping rule is satisfied, stop and go to output. Otherwise, set $k=k+1$ and go to Step $1$.
			\end{enumerate}
			
			\Ensure
			The approximate solution.
		\end{algorithmic}
	\end{algorithm}

	\subsection{The relationship with conjugate gradient type methods} \label{Rel}
	This subsection aims to demonstrate that if the sample spaces $\Omega_k=\{I\}$ and $f(x)=\frac{\gamma}{2}\|x-u\|^2_2-v$, then Algorithm \ref{amRDCD} reduces to  the \emph{conjugate gradient normal equation error} (CGNE) method \cite[Section 11.3.9]{golub2013matrix}, which is a variant of the conjugate gradient method.
	%Here $u\in\mathbb{R}^n$ and $v\in\mathbb{R}$ are constants.
	The following lemma is useful in our discussion.
	%$(2)$ if $f$ is a general strongly convex (except for $f(x)=\frac{\gamma}{2}\|x-p\|^2_2-q$), Algorithm \ref{amRDCD} becomes a dual gradient descent method, i.e. the adaptive heavy ball momentum acceleration does not work. This reveals that one advantage of stochastic algorithms is their ability to incorporate adaptive heavy ball momentum acceleration (based on the majorization technique \cite{li2016majorized,chen2017efficient}), while the gradient descent with adaptive heavy ball momentum is equivalent to regular gradient descent.
	%Our result is mainly based on the following observation:
	
	\begin{lemma}\label{lemma-obs-xie}
		The inequality in \eqref{MT-xie} is always an equality if and only if  $f^*(z)=\frac{1}{2\gamma}\|z\|^2_2+u^\top z+v$, where $u\in\mathbb{R}^n$ and $v\in\mathbb{R}$ are constants, i.e. $f(x)=\frac{\gamma}{2}\|x-u\|^2_2-v$.
	\end{lemma}
	\begin{proof}
		Note that the inequality in \eqref{MT-xie} follows from \eqref{strongly-convex}. Hence,  the inequality in \eqref{MT-xie} is always an equality if and only if for any $x,y\in\mathbb{R}^n$,
		\begin{equation}\label{xie-equality}
			f^*(y) = f^*(x)+\left\langle\nabla f^*(x), y-x\right\rangle+\frac{1}{2 \gamma}\|y-x\|_{2}^2.
		\end{equation}
		On the one hand,  \eqref{xie-equality} can be rewritten as $f^*(z)=\frac{1}{2\gamma}\|z\|^2_2+u^\top z+v$, where $u$ and $v$ are constants. On the other hand, if $f^*(z)=\frac{1}{2\gamma}\|z\|^2_2+u^\top z+v$, one can verify that \eqref{xie-equality} holds. This completes the proof of this lemma.
	\end{proof}
	
	%		Since the $v$ and $\gamma$ do not effect the solution of the minimization problem, we can consider the case where the objective function $f(x)=\frac{\gamma}{2}\|x-u\|^2_2$. Additionally, note that the value of $\gamma$ also does not affect the solution of problem \eqref{main-prob}. Therefore, we can simplify the problem by considering the case where the objective function $f(x)=\frac{1}{2}\|x-u\|^2_2$.
	
	Since $v$ and $\gamma$ do not effect the solution of the minimization problem, we can simplify the problem by considering the case where the objective function $f(x)=\frac{1}{2}\|x-u\|^2_2$.
	%\subsection{The case where $f(x)=\frac{1}{2}\|x-p\|^2_2$}
	%It is follow from Lemma \ref{lemma-obs-xie} that the inequality in \eqref{MT-xie} now becomes equality, which implies  that \eqref{adap-s-m} gives the exactly solutions of the optimization problem \eqref{opt-prob}.
	Now the inequality in \eqref{MT-xie} becomes an equality, we know that \eqref{adap-s-m}  provides the exact solutions to the optimization problem \eqref{opt-prob}  if $\|d^k\|^2_2\|z^k-z^{k-1}\|^2_2-\langle d^k,z^k-z^{k-1}\rangle^2\neq 0$.
	Furthermore,  the sequences of iterates in Algorithm \ref{amRDCD} satisfy $z^k=x^k-u$ for $k\geq0$. Hence, we can rewrite the minimizers in \eqref{adap-s-m} as follows,
	\begin{equation}\label{adap-s-m11}
		\left\{
		\begin{array}{ll}
			\alpha_k=\frac{\langle d^{k},x^k-\widehat{x}\rangle\|x^k-x^{k-1}\|^2_2-\langle d^k,x^k-x^{k-1}\rangle\langle x^k-x^{k-1},x^k-\widehat{x}\rangle}
			{\|d^k\|^2_2\|x^k-x^{k-1}\|^2_2-\langle d^k,x^k-x^{k-1}\rangle^2}, \\[0.3cm]
			\beta_k=\frac{-\|d^{k}\|^2_2\langle x^k-\widehat{x},x^k-x^{k-1}\rangle+\langle d^k,x^k-x^{k-1}\rangle \langle d^k,x^k-\widehat{x}\rangle}
			{\|d^k\|^2_2\|x^k-x^{k-1}\|^2_2-\langle d^k,x^k-x^{k-1}\rangle^2}.
		\end{array}
		\right.
	\end{equation}
	When the sample spaces $\Omega_k=\{I\}$, we know that the iteration scheme of $z^{k+1}$ in Algorithm \ref{amRDCD} becomes
	$$
	x^{k+1}=x^k-\alpha_k A^\top (Ax^k-b)+\beta_k(x^k-x^{k-1}).
	$$
	It follows from \cite[Section 4]{zeng2024adaptive} that for $k\geq 1$, \eqref{adap-s-m11} can be simplified to
	\[
		\left\{\begin{array}{ll}
			\alpha_k
			=\frac{\| x^k-x^{k-1} \|_2^2 \|Ax^k-b\|_2^2}{\| A^\top(Ax^k-b) \|_2^2 \| x^k-x^{k-1} \|_2^2 - \langle A^\top(Ax^k-b), x^k-x^{k-1} \rangle^2},
			\\[0.3cm]
			\beta_k
			=
			\frac{\langle A^\top(Ax^k-b),x^k-x^{k-1} \rangle \|Ax^k-b\|_2^2}{\| A^\top(Ax^k-b) \|_2^2 \| x^k-x^{k-1} \|_2^2 - \langle A^\top(Ax^k-b), x^k-x^{k-1} \rangle^2}.
		\end{array}
		\right.
	\]
	Moreover, Algorithm \ref{amRDCD} can be expressed in the following  equivalent form.
	
	\begin{prop}[\cite{zeng2024adaptive}, Theorem 5.1]\label{SCGNE}
		Let $f(x)=\frac{1}{2}\|x-u\|^2_2$ and assume that for any $k\geq0$, the sample spaces $\Omega_k=\{I\}$. Suppose that $x^0\in u+\text{Range}(A^\top)$ is the  initial point in Algorithm \ref{amRDCD}  and set $r^0=Ax^0-b$, $p^0=-A^\top r^0$. Let $x^1$ be generated by Algorithm \ref{RDCD} with $\zeta=1$. Then for any $k\geq0$,  Algorithm \ref{amRDCD} can be  equivalently rewritten as %the following format
		\begin{equation}
			\label{EF}
			\left\{\begin{array}{ll}
				\delta_k =\|r^k\|_2^2 / \|p^k\|_2^2,
				\\[1.7mm]
				x^{k+1}=x^k+\delta_k p^k,
				\\[1.7mm]
				r^{k+1}= r^k+\delta_kAp^k,
				\\[1.7mm]
				\eta_k=\langle A^\top  r^{k+1}, p^k \rangle/ \|p^k\|_2^2=\|r^{k+1}\|^2_2/\|r^k\|^2_2,
				\\[1.7mm]
				p^{k+1}= - A^\top r^{k+1}+\eta_k p^k.
			\end{array}
			\right.
		\end{equation}
		%where $p_0=-\nabla f_{s_0}(x^0)$.
	\end{prop}
	The iteration scheme \eqref{EF} is  exactly the \emph{conjugate gradient normal equation error} (CGNE) method \cite[Section 11.3.9]{golub2013matrix}, a variant of the conjugate gradient method for solving
	$$AA^\top y= b, x=A^\top y,$$
	which is  equivalent to $Ax=b$.
	%Finally, we note that for general probability spaces $\{(\Omega_k, \mathcal{F}_k, P_k)\}_{k\geq0}$,
	%the above discussion indicates that Algorithm \ref{amRDCD} can be leveraged to establish a
	%novel stochastic conjugate gradient (SCG) method. We refer to \cite{zeng2023adaptive} for more details about the discussions above.
	%\begin{remark}
	%Assume that the objective function $f(x)=\frac{1}{2}\|x\|^2_2$. We note that if we choose the sample spaces $\Omega_k=\{I\}$ for all $k$, then Algorithm \ref{amRDCD} will become the conjugate gradient normal equation error (CGNE) method \cite[Section 11.3.9]{golub2013matrix}, a variant of the conjugate gradient method.
	It is worth noting that for general probability spaces $\{(\Omega_k, \mathcal{F}_k, P_k)\}_{k\geq0}$, if we require the sampling matrices $S_k$ to be chosen such that $S_k^\top (Ax^k-b)\neq 0$ for $k\geq 0$, then Algorithm \ref{amRDCD} can be utilized to establish a novel stochastic conjugate gradient (SCG) method. For further details on this topic, please refer to \cite{zeng2024adaptive}.
	%then Algorithm \ref{amRDCD} can be  leveraged to establish a novel stochastic conjugate gradient (SCG) method. We refer to \cite{zeng2023adaptive} for more details about the discussions above.
	
	\subsection{Extension to general $L$-smooth convex functions}
	\label{subsection-4-2}
	
	Since the objective function $f$ is $\gamma$-strongly convex, it follows from \eqref{strongly-convex} that the objection function $g$ in \eqref{opt-prob} is $\frac{\|A\|^2_2}{\gamma}$-smooth and convex. A natural and interesting question is that \emph{can our adaptive heavy ball momentum technique be extended to general $L$-smooth convex functions?}
	
	Similar to \eqref{opt-prob}, we consider the following optimization problem
	\begin{equation}\label{opt-prob-general}
		%{\boxed{
				\begin{aligned}
					\min\limits_{x}\ \ \varphi(x)\ \
					\text{subject to} \ \ x=x^{k}-\alpha S_kS_k^\top\nabla \varphi\big(x^{k}\big)+\beta\big(x^{k}-x^{k-1}\big), \ \alpha,\beta\in\mathbb{R},
				\end{aligned}
				%}}
	\end{equation}
	where $\varphi$ is $L$-smooth convex and $S_k$ is randomly chosen from $\Omega_k$.  We also use the majorization technique \cite{li2016majorized,chen2017efficient} to find an approximate solution of  \eqref{opt-prob-general}. We have
	%\begin{equation}\label{MT-xie}
	$$\begin{aligned}
		\varphi(x) \leq& \varphi(x^k)+\langle \nabla \varphi(x^k), x-x^k\rangle+\frac{L}{2}\|x-x^k\|^2_2
		\\
		=& \varphi(x^k)-\left\langle \nabla \varphi(x^k),\alpha S_{k} S_{k}^\top\nabla \varphi(x^k)-\beta(x^k-x^{k-1}) \right\rangle
		\\
		&+\frac{L}{2}\left\|\alpha S_{k} S_{k}^\top\nabla \varphi(x^k)-\beta(x^k-x^{k-1})\right\|^2_2.
		%=&\varphi(x^k)-\alpha\|S_{k}^\top\nabla \varphi(x^k)\|^2_2+\beta\langle \nabla \varphi(x^k),x^k-x^{k-1}\rangle
		%\\
		%&+\frac{L}{2}\left\|\alpha S_{k} S_{k}^\top\nabla \varphi(x^k)-\beta(x^k-x^{k-1})\right\|^2_2,
		%
		%\langle z^k,\widehat{x}\rangle- \left\langle x^k-\widehat{x},\alpha d_k -\beta(z^k-z^{k-1}) \right\rangle
		%+\frac{1}{2\gamma}\left\|\alpha d_k-\beta(z^k-z^{k-1})\right\|^2_2,
	\end{aligned}$$
	%\end{equation}
	The optimal value of the right hand is obtained when
	\begin{equation}\label{adap-general}
		\left\{
		\begin{array}{ll}
			\alpha_k=\frac{1}{L}\frac{\|S^\top_k\nabla \varphi(x^k)\|^2_2\|x^k-x^{k-1}\|^2_2-\langle S_k S^\top_k\nabla \varphi(x^k), x^k-x^{k-1}\rangle\langle \nabla \varphi(x^k),x^k-x^{k-1}\rangle}
			{\|S_kS^\top_k\nabla \varphi(x^k)\|^2_2\|x^k-x^{k-1}\|^2_2-\langle S_kS^\top_k\nabla \varphi(x^k) ,x^k-x^{k-1}\rangle^2}, \\[0.5cm]
			\beta_k=\frac{1}{L}\frac{-\|S_kS^\top_k\nabla \varphi(x^k)\|^2_2\langle \nabla \varphi(x^k),x^k-x^{k-1}\rangle+\|S^\top_k\nabla \varphi(x^k)\|^2_2\langle S_kS^\top_k\nabla \varphi(x^k) ,x^k-x^{k-1}\rangle}
			{\|S_kS^\top_k\nabla \varphi(x^k)\|^2_2\|x^k-x^{k-1}\|^2_2-\langle S_kS^\top_k\nabla \varphi(x^k) ,x^k-x^{k-1}\rangle^2}
		\end{array}
		\right.
	\end{equation}
	provided that $\|S_kS^\top_k\nabla \varphi(x^k)\|^2_2\|x^k-x^{k-1}\|^2_2-\langle S_kS^\top_k\nabla \varphi(x^k) ,x^k-x^{k-1}\rangle^2\neq 0$. However, in practice, it may be difficult to obtain the solutions $\alpha_k$ and $\beta_k$ because they require calculating $\langle \nabla \varphi(x^k),x^k-x^{k-1}\rangle$  and  the full gradient $\nabla \varphi(x^k)$ may not be easy to be obtained. In other words, if one is able to efficiently compute $\langle \nabla \varphi(x^k),x^k-x^{k-1}\rangle$, then the strategy provided by \eqref{adap-general} can be used to develop an adaptive stochastic heavy ball momentum method.
	
	When $\Omega_k=\{I\}$, \eqref{adap-general} reduces to $\alpha_k=1/L$ and $\beta_k=0$, which indicates that our approach reduces to the gradient method in this case.
	Since the selection of $\alpha_k$ and $\beta_k$ relies on solving the optimization problem \eqref{opt-prob-general},  our approach reconfirms the superiority of the traditional parameters of the regular gradient method.
	It also implies that in the context of stochastic methods, our adaptive heavy ball momentum technique could compensate for the loss of information caused by only partly using the gradients, via utilizing iteration information.

	Finally, we note that the adaptive HBM technique can be further extended to functions satisfying relative smoothness \cite{lu2018relatively}. Let \( h: \mathbb{R}^{n} \rightarrow \mathbb{R} \) be a differentiable convex function. We say that \( \varphi \) is \( L \)-smooth relative to \( h \) on \( \mathbb{R}^{n} \) if for all \( x, y \in \mathbb{R}^{n} \), it holds that
	\[
	\varphi(y) \leq \varphi(x) + \langle \nabla \varphi(x), y - x \rangle + L D_{h, \nabla h(x)}(x, y).
	\]
	In particular, if \( h(x) = \frac{1}{2L} \| x \|_{H}^{2} \), where \( H \in \mathbb{R}^{n \times n} \) is a positive definite matrix, then the above inequality reduces to
	\[
	\varphi(y) \leq \varphi(x) + \langle \nabla \varphi(x), y - x \rangle + \frac{1}{2} \| x - y \|_{H}^{2}.
	\]
	The computation of the optimal parameters \( \alpha_{k} \) and \( \beta_{k} \) in the resulting majorized optimization problem requires matrix-vector products involving \( H \). When \( H \) is dense, these computations can be expensive. However, if \( H \) is sparse, e.g. diagonal or scalar matrix, the cost is significantly reduced. Therefore, as long as \( H \) is chosen such that the computation of \( \alpha_{k} \) and \( \beta_{k} \) remains tractable, our adaptive HBM technique can be effectively applied in this more general setting.

	\subsection{Geometric viewpoint and convergence analysis}
	In this subsection, we first give a view of geometric
	interpretation of our approach and then establish the convergence of Algorithm \ref{amRDCD}.
	We first introduce some auxiliary variables. Recall that $d^k$ is defined as $d^k=A^\top S_k S_k^\top(Ax^k-b)$,
	we define two affine sets as
	$$\begin{aligned}
		\widetilde{\Pi}_k:=&x^k+\text{Span}\left\{d^k, z^k-z^{k-1}\right\},\\
		\Pi_k:=&z^k+\text{Span}\left\{d^k, z^k-z^{k-1}\right\},
	\end{aligned}$$
	and let
	\begin{equation}
		\label{def-wk}
		w^{k+1}:=\frac{1}{\gamma}(z^{k+1}-z^k)+x^k=x^k-\frac{\alpha_k}{\gamma}d^k+\frac{\beta_k}{\gamma}(z^k-z^{k-1}).
	\end{equation}
	Since the objective function in \eqref{xie-mo} can be  equivalently written as
	%Consider the case where $\|d_k\|^2_2\|z^k-z^{k-1}\|^2_2-\left(d_k^\top(z^k-z^{k-1})\right)^2\neq0$, due to
	$$
	h^k(\alpha, \beta)= \frac{\gamma}{2}\left\|x^k-\frac{\alpha}{\gamma} d^k+\frac{\beta}{\gamma}(z^k-z^{k-1}) -\widehat{x}\right\|^2_2-\frac{\gamma}{2}\left\|x^k-\widehat{x}\right\|^2_2,
	$$
	the majorized optimization problem \eqref{xie-mo} now becomes
	$$
	\min_{\alpha,\beta\in\mathbb{R}} \left\|x^k-\frac{\alpha}{\gamma} d^k+\frac{\beta}{\gamma}(z^k-z^{k-1}) -\widehat{x}\right\|^2_2,
	$$
	which implies that $w^{k+1}$ defined above is the orthogonal projection of $\widehat{x}$ onto the affine set $\widetilde{\Pi}_k$. We define
	$$\phi_{z^k}(x):=\frac{\gamma}{2}\left\|x+\frac{1}{\gamma}z^k-\nabla f^*(z^k)\right\|^2_2-\left(
	\frac{1}{2\gamma}\|z^k\|^2_2+f^*(z^k)-\left\langle z^k,\nabla f^*(z^k)\right\rangle\right),$$
	and hence
	$
	\phi_{z^k}^*(z)=f^*(z^k)+\langle \nabla f^*(z^k),z-z^k\rangle+\frac{1}{2\gamma}\|z-z^k\|^2_2.
	$
	%and hence .
	Since $f^*$ is $1/\gamma$-smooth convex, we know that
	$
	f^*(z)\leq \phi_{z^k}^*(z)
	$, i.e.  $\phi_{z^k}^*(z)$ is a quadratic approximation of $f^*(z)$.
	Note that $x^k=\nabla f^*(z^k)$, we have
	$$
	\nabla\phi_{z^k}(w^{k+1})=\gamma\left(w^{k+1}+\frac{1}{\gamma}z^{k}-x^k\right)=z^{k+1}.
	$$
	% or equivalently $z^{k+1}=\nabla \phi_{z^k}(w^{k+1})$.
	% \eqref{MT-xie}
	This means that the next iterate $z^{k+1}$ is determined by $z^{k+1}=\nabla \phi_{z^k}(w^{k+1})$.
	%Based on the discussions above, we can conclude that the next iterate $z^{k+1}$ is determined by $z^{k+1}=\nabla \phi_{z^k}(w^{k+1})$ with $w^{k+1}$ representing the orthogonal projection of $\widehat{x}$ onto the affine set $\widetilde{\Pi}_k$.
	%we can conclude that the next iterate $z^{k+1}$ arises such that $z^{k+1}=\nabla \phi_{z^k}(w^{k+1})$ with $w^{k+1}$ being the orthogonal projection of $\widehat{x}$ onto the affine set $\widetilde{\Pi}_k$.
	The geometric interpretation is presented in Figure \ref{GI1}. %Therefore, if $\phi_{z^k}^*(z)$ is a good approximation of $f^*(z)$, we can consider $z^{k+1}=\nabla\phi_{z^k}(w^{k+1})$ as a good approximation of $z^{k+1}_*\in \partial f(w^{k+1})$.
	Accordingly, if $\phi_{z^k}^*(z)$ serves as a reliable approximation of $f^*(z)$, we can consider $z^{k+1}=\nabla\phi_{z^k}(w^{k+1})$ as a suitable approximation of $z^{k+1}_*\in \partial f(w^{k+1})$.
	
	%  Particularly, if the strongly convex parameter  $\gamma=1$, our approach can be seen geometrically as an orthogonal projection method.
	
	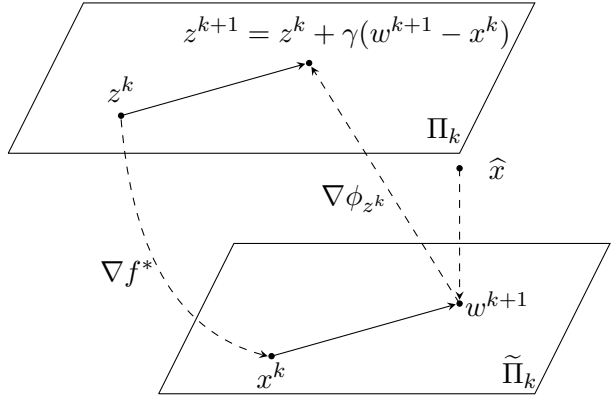
\begin{figure}[hptb]
		\centering
		\begin{tikzpicture}
			\draw (0,0)--(5,0)--(6,2)--(1,2)--(0,0);
			\draw (-2,3.2)--(4,3.2)--(5,5.2)--(-1,5.2)--(-2,3.2);
			\filldraw (1.5,0.5) circle [radius=1pt]
			%(2,1.35) circle [radius=1pt]
			(4,3) circle [radius=1pt]
			(4,1.2) circle [radius=1pt]
			(2,4.4) circle [radius=1pt]
			(-0.5,3.7) circle [radius=1pt];
			%(2,3.55) circle [radius=1pt];
			\draw (1.5,0.25) node {$x^k$};
			\draw (-0.5,4.05) node {$z^k$};
			\draw (4.5,3) node {$\widehat{x}$};
			%\draw (1.5,1.3) node {$y^{k+1}$};
			%\draw (1.5,3.5) node {$\tilde{z}^{k+1}$};
			\draw (4.5,1.2) node {$w^{k+1}$};
			\draw (2.5,4.8) node {$z^{k+1}=z^k+\gamma(w^{k+1}-x^k)$};
			\draw (4.8,0.3) node {$\widetilde{\Pi}_k$};
			\draw (2.6,2.6) node {$\nabla\phi_{z^k}$};
			\draw (-0.42,1.6) node {$\nabla f^*$};
			\draw (3.8,3.5) node {$\Pi_k$};
			\draw [dashed,-stealth] (4,3) -- (4,1.25);
			\draw [-stealth] (1.5,0.5) -- (3.95,1.19);
			\draw [-stealth] (-0.5,3.7) -- (1.95,4.39);
			\draw [dashed,-stealth] (4,1.2)-- (2.03,4.352);
			\draw [dashed,-stealth] (-0.5,3.65) to[out=275, in=165] (1.45,0.5);
		\end{tikzpicture}
		\caption{A geometric interpretation of Algorithm \ref{amRDCD}. The iterate $x^k=\nabla f^*(z^k)$ and $w^{k+1}$ is the orthogonal projection of $\widehat{x}$ onto the affine set $\widetilde{\Pi}_k$. Then the next iterate $z^{k+1}=\nabla \phi_{z^k}(w^{k+1})$.
			%The vector $w^{k+1}$ is the orthogonal projection of $\widehat{x}$ onto the affine set $\widetilde{\Pi}_k$. When the strongly convex parameter  $\gamma=1$, next iterate $z^{k+1}$  arises such that $z^{k+1}$ is the orthogonal projection of $\widehat{x}$ onto the affine set $\Pi_k$.
		}	
		\label{GI1}
	\end{figure}
	
	%When $\|d_k\|^2_2\|z^k-z^{k-1}\|^2_2-\left(d_k^\top(z^k-z^{k-1})\right)^2\neq0$, we define the affine set $\Pi_k=x^k+\text{Span}\left\{d_k, z^k-z^{k-1}\right\}$ and
	Next, we establish the convergence result for Algorithm \ref{amRDCD}.
	Define
	\begin{equation}\label{y}
		y^{k+1}:=x^k-\frac{L_{\text{adap}}^{k,\gamma}}{\gamma}d^k,
		%\end{aligned}
	\end{equation}
	where $L_{\mathrm{adap}}^{k,\gamma}$ is given by \eqref{Lk}. Let
	\[
	\mathcal{Q}_k := \left\{ S \in \Omega_k \mid S^\top (Ax^k - b) \neq 0 \right\}
	\]
	and define the vector
	\[
	u^k := \langle d^k, z^k - z^{k-1} \rangle d^k - \| d^k \|_2^2 (z^k - z^{k-1}).
	\]
	%where $u^k$ and $d^k$ form an orthogonal basis of $\tilde{\Pi}_k$. 
	Furthermore, let $\theta_k$ denote the angle between $y^{k+1}-\widehat{x}$ and $u^k$, i.e.
	\begin{equation}
		\label{def-delta}
		\theta_k:=\arccos\frac{\langle y^{k+1}-\widehat{x},u^k\rangle}{\|y^{k+1}-\widehat{x}\|_2\|u^k\|_2},
	\end{equation}
	where we define $\frac{0}{0}=0$. We now present convergence results for
	Algorithm \ref{amRDCD}. The detailed proof is provided in the Appendix \ref{secA2}.

	\begin{thm}
		\label{thm-mrdcd}
		Let $f$ be $\gamma$-strongly convex and strongly admissible. Suppose that the probability spaces $\{(\Omega_k, \mathcal{F}_k, P_k)\}_{k\geq 1}$ satisfy Assumption \ref{Ass}. Let $\{x^k\}_{k\geq1}$ and $\{z^k\}_{k\geq1}$ be the sequences of iterates generated by Algorithm \ref{amRDCD}. Then
		%For any given linear system $Ax=b$, let $\{x^k\}_{k=0}^{\infty}$ be the iteration sequence generated by Algorithm \ref{rrbk} with $x^0=0$. Then
		$$
		\mathbb{E}{_k}\left[D_{f,z^{k+1}}(x^{k+1},\widehat{x}) \right]\leq \left(1-\frac{\gamma\nu\lambda_{\min}\left(H_k\right) }{2\lambda_{\max}^{(k)}} \right)D_{f,z^{k}}(x^{k},\widehat{x})-\frac{\gamma}{2}\mathbb{E}{_k}\left[\cos^2\theta_k \|y^{k+1}-\widehat{x}\|_2^2\right],
		$$
		where  $\widehat{x}$ is the solution of \eqref{main-prob}, $\nu$, $H_k$, $\lambda_{\max}^{(k)}$, $y^{k+1}$ and $\theta_k$ are given by \eqref{nu-cons}, \eqref{matrix-H}, \eqref{lambda-max}, \eqref{y} and \eqref{def-delta}, respectively. %Here we define $\frac{0}{0}=0$.
	\end{thm}

	\begin{REMK}
		Upon comparison of Theorem \ref{thm-rdcd} and Theorem \ref{thm-mrdcd}, it can be observed that the ASDCD method exhibits convergence rate that is at least as fast as that of the SDCD method. Indeed, for certain objective function $f(x)$ and probability spaces $\{(\Omega_k, \mathcal{F}_k, P_k)\}_{k\geq 1}$, we can show that the convergence rate in Theorem \ref{thm-mrdcd} can be strictly smaller than that in Theorem \ref{thm-rdcd}. For example, for the case where $f(x)=\frac{1}{2}\|x\|^2_2$ and the sample spaces $\Omega_k=\{I\}$ for any $k\geq1$.  We refer to \cite[Remark 5.3]{zeng2024adaptive} for more details.
	\end{REMK}

\subsection{Efficient implementation for sparse data}
		
Algorithm \ref{amRDCD} exhibits a computational  disadvantage when applied to sparse matrices $A$. Indeed, the vectors $z^{k}$ and $z^{k-1}$ may be dense. Consequently, updating the momentum term requires full-dimensional vector operations, leading to a cost of $\mathcal{O}(n)$ arithmetic operations per iteration for obtaining  $z^{k+1}$. In contrast, the SDCD method can potentially circumvent such computational costs when $A$ is sparse, as  $d^{k}$ may remain sparse under this setting. Inspired by the idea of variable transformation adopted in \cite{lee2013efficient,fercoq2015accelerated}, we reformulate Algorithm \ref{amRDCD} into an equlvalent form, presented as Algorithm \ref{amRDCD_EI}, where we define $\frac{0}{0}=0$ by convention.

						\begin{algorithm}[htpb]
					\caption{ASDCD (written in a form facilitating efficient implementation) \label{amRDCD_EI}}
					\begin{algorithmic}
						\Require
						$A\in \mathbb{R}^{m\times n}$, $b\in \mathbb{R}^m$, probability spaces $\{(\Omega_{k}, \mathcal{F}_{k}, P_{k})\}_{k \geq 1}$, $k=1$, and initial points $z^{0} \in \text{Range}(A^{\top})$, $\xi^{0} \in \mathbb{R}^{m}$. Set $(h^{0},q^0,\delta_0)=(z^{0},0,1)$, $(h^{1},q^1,\delta_1)=\left(z^{0}+2A^{\top}\xi^{0},-2A^{\top}\xi^{0},\frac{1}{2}\right)$, $\delta^*_{0}=1$, $\theta_{-1}=\theta_{0}=\frac{1}{2}$, $\beta_0=0$, $l_{1}=\|q^{1}\|_{2}^{2}$, and $\tau_{1}=-2\langle\xi^0, b\rangle$.
						\begin{enumerate}
							\item[1:] Randomly select a sampling matrix $S_{k}\in\Omega_k$.
							%Draw sample $S_k\thicksim\mathcal{D}_k$.
							
							\item[2:] Compute $d_{1}^{k}=(S_{k}^{\top}A)\nabla f^{*}(h^{k}+\delta_{k} q^{k})-S_{k}^{\top}b$ and $d^{k}=(A^\top S_{k}) d_{1}^{k}$.
							
							\item[3:]  If
							$\|d^k\|^2_2 l_{k}-\langle d^k,q^{k}\rangle^2=0$ or $ \|d^{k}\|_{2}^{2}(\langle q^{k}, \nabla f^{*}(h^{k}+\delta_{k} q^{k})\rangle-\tau_{k})-\langle d^{k},q^{k} \rangle \|d_{1}^{k}\|_{2}^{2}=0$
							
							\quad Update $h^{k+1}$, $q^{k+1}$, $\delta_{k+1}$, $\delta_{k}^*$, $\theta_{k}$, $l_{k+1}$, and $\tau_{k+1}$ by Stage I.
							
							Otherwise,
							
							\quad Update $h^{k+1}$, $q^{k+1}$, $\delta_{k+1}$, $\delta_{k}^*$, $\theta_{k}$, $l_{k+1}$, and $\tau_{k+1}$ by Stage II.
							\item[4:] If the stopping rule is satisfied, stop and go to output. Otherwise, set $k=k+1$ and go to Step $1$.
						\end{enumerate}
						
						\Ensure
						The approximate solution $\nabla f^{*}(h^{k+1}+\delta_{k+1} q^{k+1})$.
					\end{algorithmic}
				\end{algorithm}

				\begin{table}[htpb]
					\centering
					\begin{tabular}{  |l|  }
						\hline
						\qquad \qquad \qquad \qquad \qquad \qquad \quad \qquad \textbf{Stage I} \qquad \qquad \qquad \qquad \qquad \qquad \quad  \qquad \\
						1: Set $\beta_{k}=0$, $\theta_{k}=\frac{1}{2}$, and $\delta^*_{k}=1$. \\
						2: Compute   \;  $z^{k}=h^{k}+ \delta_{k} q^{k}$ and $\alpha_{k}=\frac{\gamma \|d_{1}^{k}\|^2_2}{\|d^k\|^2_2}$.\\
%						\qquad\qquad \quad \quad
%						$
%						\alpha_{k}=
%						\begin{cases}
%							\frac{\gamma \|d_{1}^{k}\|^2_2}{\|d^k\|^2_2} & \text{if} \; d_{1}^{k} \neq 0; \\
%							0 & \text{otherwise}.
%						\end{cases}
%						$ \\
						3: Update \;\;
						$
						(h^{k+1},q^{k+1}, \delta_{k+1})=\left(z^{k}-2\alpha_{k}d^{k},2\alpha_{k}d^{k}, \frac{1}{2}\right) 
						$ and \\[1.5mm]
						\qquad\qquad\qquad\quad\;\;\;\;
						$
						\left(l_{k+1}, \tau_{k+1}\right)=\left(4\alpha_{k}^{2}\|d^{k}\|_{2}^{2}, 2\alpha_{k} \langle d_{1}^{k}, S_{k}^{\top}b\rangle \right).
						$ \\[1.5mm]
						\hline
					\end{tabular}
					
				\end{table}
				\begin{table}[htpb]
					\centering
					\begin{tabular}{  |l|  }
						\hline
						\qquad \qquad \qquad \qquad \qquad \qquad \quad \qquad \textbf{Stage II} \qquad \qquad \qquad \qquad \qquad \qquad \quad  \qquad\\[1.7mm]
						1: Compute
						$
						\alpha_{k}=\gamma \frac{\|d_{1}^{k}\|^2_2l_{k}-\langle d^{k}, q^{k} \rangle (\langle q^{k}, \nabla f^{*}(h^{k}+\delta_{k} q^{k}) \rangle-\tau_{k})}{\|d^k\|^2_2l_{k}-\langle d^k,q^{k}\rangle^2}
						$ and \\
						\qquad \qquad\quad\; $
						\beta_{k}=\frac{\gamma}{\theta_{k-1} \delta^*_{k-1}} \frac{\|d^{k}\|_{2}^{2}(\langle q^{k}, \nabla f^{*}(h^{k}+\delta_{k} q^{k} \rangle-\tau_{k}) - \langle d^{k},q^{k} \rangle \|d_{1}^{k}\|_{2}^{2}}{ \|d^k\|^2_2l_{k}-\langle d^k,q^{k}\rangle^2}.
						$ \\
						2: If $\theta_{k-1} \neq 1$ \\
						
						\qquad\quad\; Compute   $\delta^*_{k}=\delta_{k}$ and \\
						\qquad\qquad \qquad \; \; \; $
						\theta_{k} =
						\left\{\begin{array}{ll}
							\frac{\theta_{k-1}}{1-\theta_{k-1}} \beta_{k} \quad\;  \text{if} \; \theta_{k-2} \neq 1 \; \text{or} \; \beta_{k-1} = 0;
							\\
							- \beta_{k} \qquad\quad\;  \text{otherwise}.
						\end{array}
						\right.
						$ \\[2.4mm]
						\qquad\quad\;\; Update
						$
						(h^{k+1}, q^{k+1},\delta_{k+1})=
						\left(h^{k}-\frac{\alpha_{k}}{\theta_k} d^{k},  q^{k}+\frac{\alpha_{k}}{\delta^*_{k} \theta_{k}} d^{k},  (1-\theta_{k})\delta^*_{k}\right).
						$ \\
						\qquad\;\; Otherwise \\
						
						\qquad \quad\quad Set $\delta^*_{k}=2 \delta^*_{k-1} \beta_{k}$ and $\theta_{k}=\frac{1}{2}$. \\
						\qquad\quad\;\;\, Update $
						(h^{k+1},q^{k+1},\delta_{k+1})=
						\left(h^{k}, q^{k}+\frac{\alpha_{k}}{\delta^*_{k} \theta_{k}} d^{k},  -\theta_{k} \delta^*_{k}\right).
						$\\[1.5mm]
						3: Update
						$
						(l_{k+1}, \tau_{k+1})=\left(l_{k}+2\frac{\alpha_k}{\delta^*_{k}\theta_{k}} \langle d^{k},q^{k} \rangle+\frac{\alpha_k^{2}}{(\delta^*_{k})^{2}\theta_{k}^{2}} \|d^{k}\|_{2}^{2}, \tau_{k}+\frac{\alpha_{k}}{\delta^*_{k} \theta_{k}} \langle d_{1}^{k}, S_{k}^{\top}b \rangle \right).
						$ \\[1.5mm]
						\hline
					\end{tabular}
				\end{table}
				
			Since the equivalence between Algorithms \ref{amRDCD} and \ref{amRDCD_EI} is not immediately obvious, we formally state it as the following result. The detailed proof is provided in the Appendix \ref{secA3}.
				
				%To establish the equivalence between Algorithm \ref{amRDCD} and Algorithm \ref{amRDCD_EI}, we present the following result.
				\begin{prop} \label{Theo}
					Suppose that Algorithms  \ref{amRDCD} and \ref{amRDCD_EI} share the same sampling matrices \(\{S_k\}_{k\geq1}\) and initial points $z^{0}$ and $\xi^0$. Then, for any $k \geq 0$,  
					$$
					z^{k}=h^{k}+\delta_{k}q^{k}.
					$$
					That is, Algorithms  \ref{amRDCD} and \ref{amRDCD_EI} are equivalent.
				\end{prop}
				%Theorem \ref{Theo} indicates that explicit computation of $z^{k}$ is unnecessary. By introducing two new vectors, $h^{k}$ and $q^{k}$, we can express $z^{k}$ as $h^{k}+\delta_{k} q^{k}$. Utilizing \eqref{Theo_ref1}, we formally present an equivalent form of the ASDCD method in Algorithm \ref{amRDCD_EI}. In this algorithm, the newly introduced values $l_{k}$ and $\tau_{k}$ serve to compute the adaptive parameters $\alpha_{k}$ and $\beta_{k}$, where $l_{k}=\|q^{k}\|_{2}^{2}$ and $\tau_{k}=\langle q^{k}, x^{*} \rangle$.
				In Algorithm~\ref{amRDCD_EI}, explicit computation of \( z^{k} \) is unnecessary except when \( \beta_{k} = 0 \). Instead, two auxiliary vectors \( h^{k} \) and \( q^{k} \), along with a scalar parameter \(\delta_{k} \), are introduced to represent \( z^{k} \) via the decomposition \( z^{k} = h^{k} + \delta_{k} q^{k} \). The algorithm makes use of this representation through evaluating the term \((S_{k}^{\top}A)\nabla f^{*}(h^{k} + \delta_{k} q^{k})\) to determine the update direction, and the inner product \( \langle q^{k}, \nabla f^{*}(h^{k} + \delta_{k} q^{k}) \rangle \) to compute the step size parameters \( \alpha_{k} \) and \( \beta_{k} \).
				If \( \nabla f^{*}(h^{k} +\delta_{k} q^{k}) \) can be evaluated efficiently without explicitly forming \( z^k \), and given that \( A \) is sparse, then full-dimensional operations can be avoided when computing both \((S_k^\top A)\nabla f^{*}(h^k + \delta_k q^k)\) and \(\nabla f^{*}(h^k + \delta_k q^k)\). For examples of functions that admit such efficient computation, we refer the reader to \cite[Section~5]{fercoq2015accelerated}.  Indeed, when \( A \) is sparse, the vector \( d^{k} \) may also exhibit sparsity. Hence, both Stage~I and Stage~II of the algorithm can be carried out using sparse vector operations, thereby avoiding costly full-dimensional computations and making each iteration computationally efficient.
				
				Moreover, if the optimal solution \( \widehat{x} \) is sparse, and if \( \nabla f^{*}(h^{k} + \delta_{k} q^{k}) \) is close to \( \widehat{x} \), then the cost of computing the inner product \( \langle q^{k}, \nabla f^{*}(h^{k} + \delta_{k} q^{k}) \rangle \) may be significantly reduced.
				In particular, when the objective function takes the form \( f(x) = \frac{\gamma}{2}\|x - u\|_2^2 - v \), we have \( \langle z^{k} - z^{k-1}, \nabla f^{*}(z^{k}) - \widehat{x} \rangle =0\). Since the term \( \langle q^{k}, \nabla f^{*}(h^{k} + \delta_{k} q^{k}) \rangle \) is only introduced for computing \( \langle z^{k} - z^{k-1}, \nabla f^{*}(z^{k}) - \widehat{x} \rangle \), it becomes unnecessary to evaluate \( \langle q^{k}, \nabla f^{*}(h^{k} + \delta_{k} q^{k}) \rangle \) in this specific case.
				We present this specialized version of Algorithm~\ref{amRDCD_EI} as Algorithm~\ref{amRDCD_EI_Q}.
				%Notably, it offers an efficient implementation for the ASHBM method. Consequently, under these situations and assuming the matrix $A$ is sparse, ASDCD can be efficiently implemented to nearly avoid full-dimensional operations.

					\begin{center}
						\begin{minipage}{\linewidth}
							\begin{algorithm}[H]
								\caption{An efficient implementation of ASDCD for $f(x)=\frac{\gamma}{2}\|x-u\|_{2}^{2}-v$ \label{amRDCD_EI_Q}}
								\begin{algorithmic}
								\Require
												$A\in \mathbb{R}^{m\times n}$, $b\in \mathbb{R}^m$, probability spaces $\{(\Omega_{k}, \mathcal{F}_{k}, P_{k})\}_{k \geq 1}$, $k=1$, and initial points $z^{0} \in \text{Range}(A^{\top})$, $\xi^{0} \in \mathbb{R}^{m}$. Set $(h^{0},q^0,\delta_0)=(z^{0},0,1)$, $(h^{1},q^1,\delta_1)=\left(z^{0}+2A^{\top}\xi^{0},-2A^{\top}\xi^{0},\frac{1}{2}\right)$, $\delta^*_{0}=1$, $\theta_{-1}=\theta_{0}=\frac{1}{2}$, $\beta_0=0$, and $l_{1}=\|q^{1}\|_{2}^{2}$.
											\begin{enumerate}
												\item[1:] Randomly select a sampling matrix $S_{k}\in\Omega_k$.
												%Draw sample $S_k\thicksim\mathcal{D}_k$.
												
												\item[2:] Compute $d_{1}^{k}=\frac{1}{\gamma}(S_{k}^{\top}A)h^{k}+\frac{\delta_{k}}{\gamma}(S_{k}^{\top}A)q^{k}+(S_{k}^{\top}A)u-S_{k}^{\top}b$ and $d^{k}=(A^\top S_{k}) d_{1}^{k}$.
												
												\item[3:]  If
												$\|d^k\|^2_2 l_{k}-\langle d^k,q^{k}\rangle^2=0$ or $\langle d^{k},q^{k} \rangle=0$
												
												\quad Update $h^{k+1}$, $q^{k+1}$, $\delta_{k+1}$, $\delta_{k}^*$, $\theta_{k}$, $l_{k+1}$, and $\tau_{k+1}$ by Stage I.
												
												Otherwise,
												
												\quad Compute  $\alpha_{k}=\gamma \frac{\|d_{1}^{k}\|^2_2l_{k}}{\|d^k\|^2_2l_{k}-\langle d^k,q^{k}\rangle^2}$ and
												$
												\beta_{k}=-\frac{\gamma}{\theta_{k-1} \delta^*_{k-1}} \frac{\langle d^{k},q^{k} \rangle \|d_{1}^{k}\|_{2}^{2}}{\|d^k\|^2_2l_{k}-\langle d^k,q^{k}\rangle^2}.
												$								
												
												\quad Update $h^{k+1}$, $q^{k+1}$, $\delta_{k+1}$, $\delta_{k}^*$, $\theta_{k}$, $l_{k+1}$, and $\tau_{k+1}$ using Steps 2-3 in Stage II.
												
												\item[4:] If the stopping rule is satisfied, stop and go to output. Otherwise, set $k=k+1$ and go to Step $1$.
											\end{enumerate}
											
											\Ensure
											The approximate solution $\frac{h^{k+1}+\delta_{k+1}q^{k+1}}{\gamma}+u$.
								\end{algorithmic}
							\end{algorithm}
						\end{minipage}
			\end{center}

Finally, we note that alternative methods exist to mitigate the computational burden of full-dimensional vector operations from the momentum term. One such approach is the stochastic momentum technique introduced in \cite{loizou2020momentum}, where we can modify the update of \( z^{k+1} \) (Step 4 in Algorithm \ref{amRDCD}) to
\[
 z^{k+1} = z^k - \alpha_k d^k + \beta_k e_{i_k} e_{i_k}^\top (z^k - z^{k-1}),
 \]
where the index \( i_k \in [m] \) is sampled uniformly at random. Another relevant method is ProxSkip \cite{mishchenko2022proxskip}, which reduces computational complexity by probabilistically skipping the proximal operator. Inspired by this, one could consider computing the momentum term with a probability \( p \in (0,1] \), thereby reducing its evaluation frequency to once every \( 1/p \) iterations on average.

	\section{Numerical experiments}
	
	%In this section, we describe some numerical results for the RKAS method for inconsistent systems. We also compare RKAS with REK \cite{Zou12,Du19} on a variety of test problems.
	In this section, we report some numerical results that demonstrate the efficiency of the adaptive stochastic dual coordinate descent (ASDCD) method. Specifically, we will compare the performance of the methods for solving the following problem
	\begin{equation}\label{ell1-ell2}
		\min \mu\|x\|_1+\frac{1}{2}\|x\|^2_2 \ \ \text{subject to} \ \ Ax=b,
	\end{equation}
	which is a regularized version of the basis pursuit \cite{cai2009linearized,yin2008bregman,yin2010analysis}.
	%For the probability spaces, we consider the following row partition  \eqref{partition-sample} discussed in Remark \ref{remark-sampling}.
	
For the underlying sampling strategy, we adopt the row partition scheme discussed in Remark~\ref{remark-sampling}, which is formally described in equation~\eqref{partition-sample}. In this set of experiments, we do not consider uniform sampling. Although partition sampling and uniform sampling have the same computational cost per iteration, we observe that partition sampling consistently outperforms uniform sampling in terms of total CPU time.
	This performance gap is attributed to additional overhead incurred by uniform sampling during each iteration. Specifically, uniform sampling requires dynamically extracting rows from the matrix \(A\), which results in increased data movement and memory access latency. In contrast, partition sampling avoids this overhead by storing fixed submatrices of \(A\) in memory according to a predefined partition. This pre-processing step eliminates the need for repeated row extractions and enables more efficient access during the optimization process. Related discussions can be found in \cite{xie2025randomized,zeng2024adaptive}.

	For the SDCD method, we set $z^0=0$, and for the ASDCD method, we set $\xi^0=0$ and $ z^0=0$.
	
	All the methods are implemented in  {\sc Matlab} R2022a for Windows $11$ on a desktop PC with Intel(R) Core(TM) i7-1360P CPU @ 2.20GHz  and 32 GB memory. The code to reproduce our results can be found at  \href{https://github.com/xiejx-math/ASDCD}{https://github.com/xiejx-math/ASDCD}.
	
	\subsection{Choice of $\tau$}
	\label{section5-1}
	
	In this experiment, we utilize  Gaussian matrices, Bernoulli random matrices, and randomly subsampled Hardmard matrices as sensing matrices $A$.
	%These matrices are generated using $i.i.d.$ normal distributions $\mathcal{N}(0,1)$ (\texttt{randn(m,n)} in \textsc{Matlab}).
	%, whose entries are generated from $i.i.d.$ normal distributions $\mathcal{N}(0,1)$ (\texttt{randn(m,n)} in \textsc{Matlab}) with columns being  normalized, as the sensing matrices.
	We should mention that these matrices are well acknowledged to be efficient for sparse signal recovery in compressed sensing and have been widely used for numerical tests. To generate the $s$-sparse (the number of nonzero entries of a certain vector is less than or equal to $s$) vector $x \in \mathbb{R}^n$,
	we first sample a random vector $\hat{\lambda} \sim \mathcal{N}(0, I_n)$ from the standard normal distribution. We then compute $\hat{x} = S_\mu(A^\top \hat{\lambda})$, where $\mu$  is chosen as the $(s+1)$st largest absolute value among the entries of $A^\top \hat{\lambda}$. Afterward, we let $b = A\hat{x}$.
		Note that $(\hat{x}, \hat{\lambda})$ form a primal-dual pair for problem \eqref{main-prob}, satisfying
		$
		A\hat{x} = b$ and $\hat{x} = \nabla f^*(A^\top \hat{\lambda}),
		$
		which indicates that the constructed $\hat{x}$ is indeed an optimal solution. We apply the widely used stopping criterion that the relative solution error (RSE)
		$\frac{\|x^k-\hat{x}\|^2_2}{\|\hat{x}\|^2_2}\leq10^{-12}$.
	
	% we assume that the nonzero elements of $x$ are generated from a standard normal distribution. After this, we set $b=A x$. % where $\delta\geq0$ is the given noise level.
	
	%All computations are started from the initial vector $x^0=0$.

	Figures \ref{figue1}, \ref{figue2}, and \ref{figue3} illustrate the evolution of  the number of epochs and computational time (CPU) with respect to the block size $\tau$ for the SDCD and ASDCD methods. The bold line represents the median computed over $10$ independent runs. The lightly shaded area indicates the range between the minimum and maximum values, while the darker shaded region corresponds to the interquartile range, from the $25$th to the $75$th percentile. It can be observed that the ASDCD method consistently outperforms the SDCD method when $\tau < m$. In particular, for small values of $\tau$ (e.g., $\tau = 1, 2, 4$), ASDCD is approximately ten times faster than SDCD. When $\tau = m$, both ASDCD and SDCD reduce to dual full gradient methods and exhibit similar performance.
		This empirical behavior can be explained by the fact that the momentum parameter $\beta_k$ in ASDCD tends toward zero as $\tau$ approaches $m$. Indeed, from \eqref{adap-s-m}, we have
		\[
		\begin{aligned}
			\beta_{k} &= \frac{\langle A^{\top}(Ax^{k}-b), z^{k}-z^{k-1} \rangle \|Ax^{k}-b\|_{2}^{2} - \langle x^{k}-\hat{x}, z^{k}-z^{k-1} \rangle \|A^{\top}(Ax^{k}-b)\|_{2}^{2}}{\|A^{\top}(Ax^{k}-b)\|_{2}^{2} \|z^{k}-z^{k-1}\|_{2}^{2} - \langle A^{\top}(Ax^{k}-b), z^{k}-z^{k-1} \rangle^{2}} \\
			&= \frac{\langle x^{k}-\hat{x}, z^{k}-z^{k-1} \rangle_{A^\top A} \|x^{k}-\hat{x}\|_{A^\top A}^{2} - \langle x^{k}-\hat{x}, z^{k}-z^{k-1} \rangle \|x^{k}-\hat{x}\|_{(A^\top A)^2}^{2}}{\|A^{\top}(Ax^{k}-b)\|_{2}^{2} \|z^{k}-z^{k-1}\|_{2}^{2} - \langle A^{\top}(Ax^{k}-b), z^{k}-z^{k-1} \rangle^{2}},
		\end{aligned}
		\]
		which tends to zero as $A^\top A$ is almost a scalar matrix \cite[Theorem $4.6.1$]{vershynin2018high}. For the SDCD method, it can be observed that a larger $\tau$ leads to an increase in the number of epochs, yet a decrease in the total CPU time. This observation aligns with the analysis provided in Remark \ref{remark-sampling}. The underlying reason is that {\sc Matlab} engages automatic multithreading when computing matrix-vector products, which form the computational bottleneck in block sampling-based methods.
	
	\begin{figure}[hptb]
		\centering
		\begin{tabular}{cc}
			\includegraphics[width=0.4\linewidth]{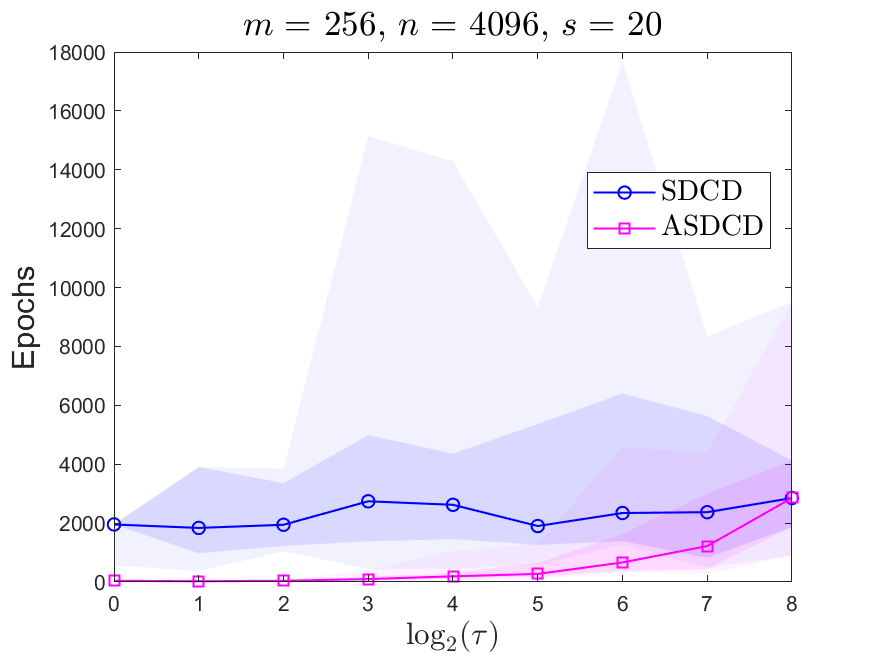}
			\includegraphics[width=0.4\linewidth]{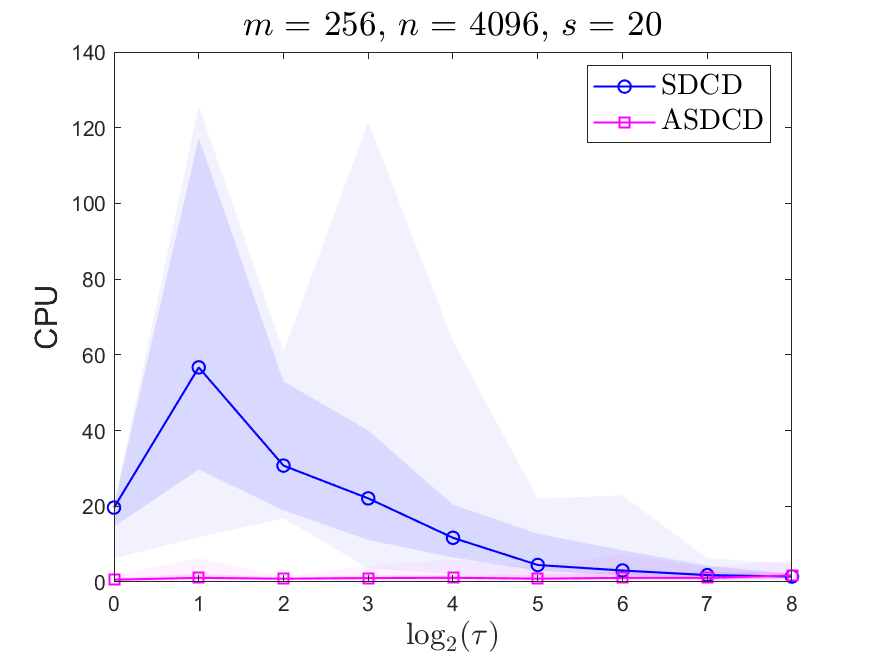}
		\end{tabular}
		\caption{Figures depict the evolution of the number of epochs and computational time (CPU)  with respect to the block size $\tau$  for  Gaussian matrices.  The title of each plot indicates the values of $m,n,s$. }
		\label{figue1}
	\end{figure}

	\begin{figure}[hptb]
		\centering
		\begin{tabular}{cc}
			\includegraphics[width=0.4\linewidth]{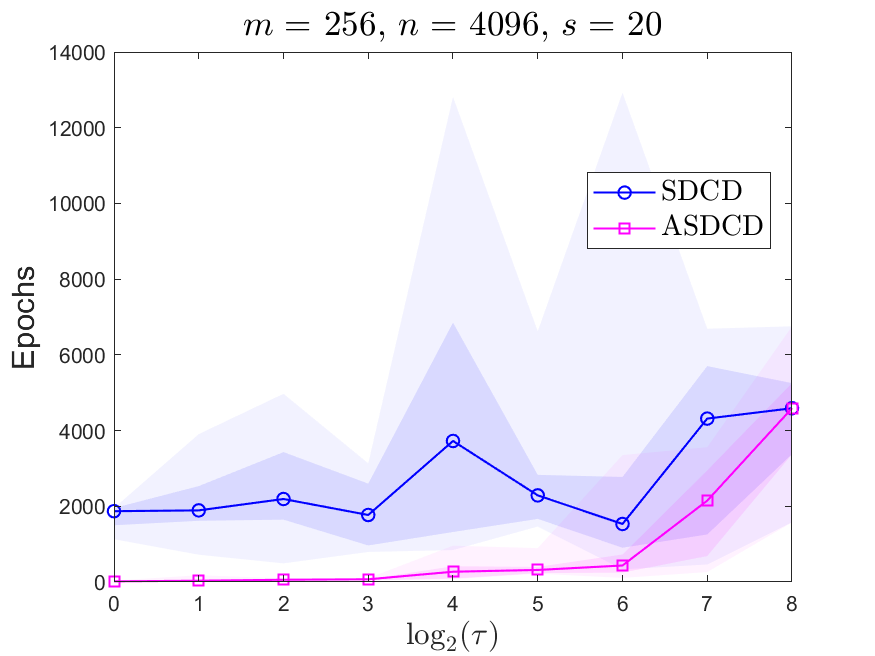}
			\includegraphics[width=0.4\linewidth]{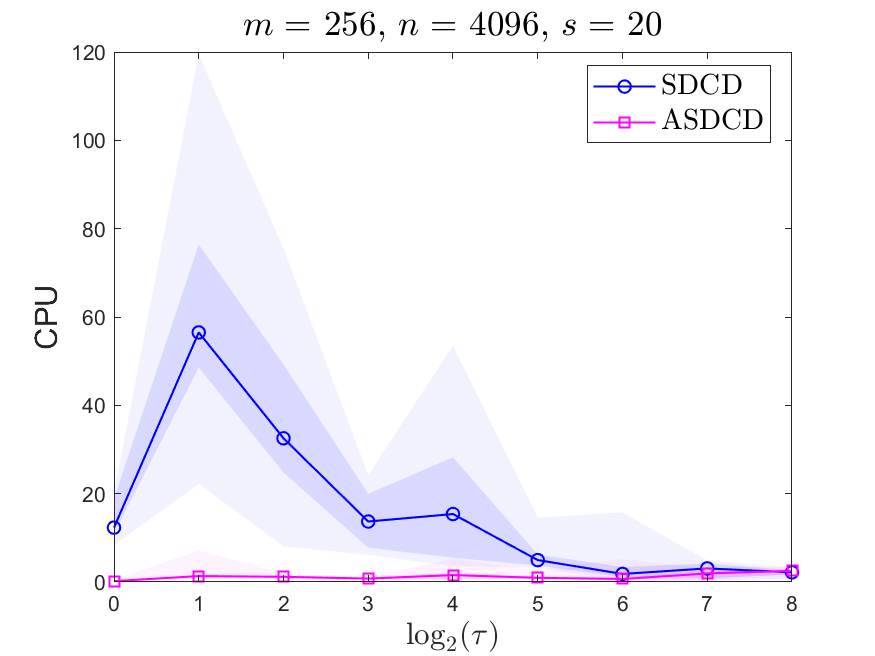}
		\end{tabular}
		\caption{Figures depict the evolution of the number of epochs and computational time (CPU) with respect to the block size $\tau$  for  Bernoulli random matrices.  The title of each plot indicates the values of $m,n,s$. }
		\label{figue2}
	\end{figure}
	
	\begin{figure}[hptb]
		\centering
		\begin{tabular}{cc}
			\includegraphics[width=0.4\linewidth]{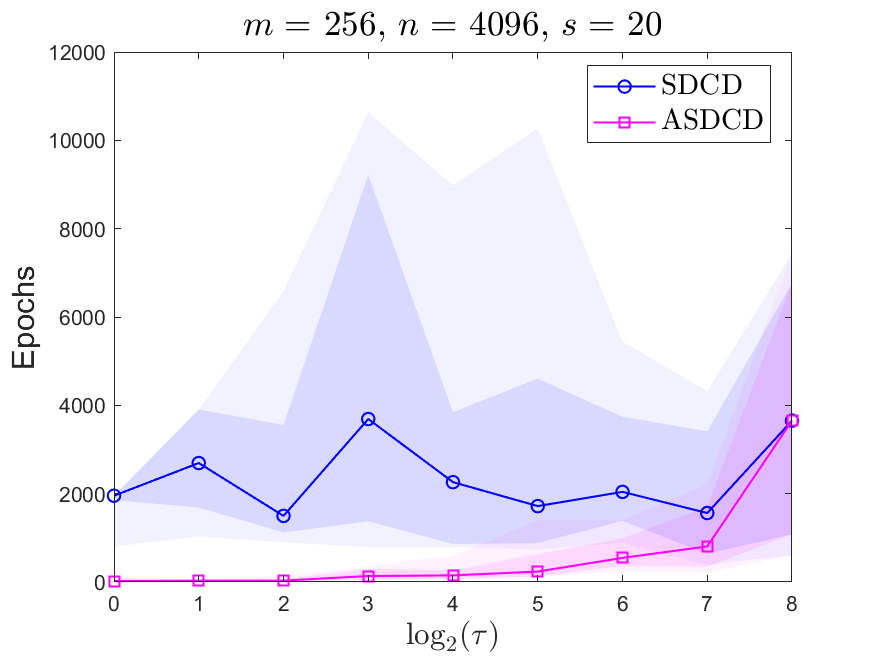}
			\includegraphics[width=0.4\linewidth]{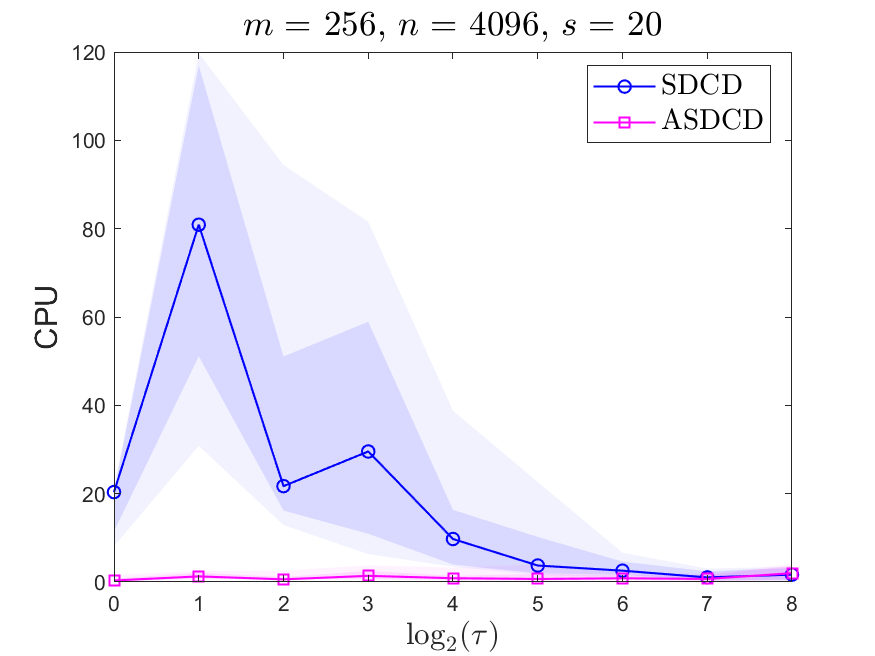}
		\end{tabular}
		\caption{Figures depict the evolution of the number of epochs and computational time (CPU)  with respect to the block size $\tau$  for  randomly subsampled Hardmard matrices.  The title of each plot indicates the values of $m,n,s$. }
		\label{figue3}
	\end{figure}

	\subsection{Comparison to the existing methods}
	We compare the performance of  the following methods for solving \eqref{ell1-ell2}: (1) alternating direction method of multipliers (ADMM) \cite{boyd2011distributed,han2022survey,yang2011alternating}; (2) linearized Bregman iteration \cite{cai2009linearized,cai2009convergence} (denoted by LB); (3) Nesterov accelerated linearized Bregman iteration \cite{huang2013accelerated} (denoted by ALB);  (4) our proposed methods (SDCD and ASDCD).  In particular,  we use the following iteration strategy  adopted from \cite[Remark 1]{yang2011alternating} for the ADMM method
	$$
	\begin{aligned}
		x^{k+1}&=S_{{\tilde{\mu}}/\beta }\left(x-\tau A^\top(Ax^k-b-y^k/\beta)\right),\\
		y^{k+1}&=y^k- \gamma\beta (Ax^{k+1}-b),
	\end{aligned}
	$$
	where $\beta>0$ is  a penalty parameter  and $\tilde{\mu},\gamma>0$ satisfy $\tilde{\mu}\|A\|^2_2+\gamma<2$. In our test, we set $\tilde{\mu}=\frac{1}{\|A\|_2^2},\gamma=0.99$, and $\beta=0.01$.
	The ALB method has the following iteration
	$$
	\begin{aligned}
		x^{k+1}&=S_{\mu }(\tilde{z}^k),\\
		z^{k+1}&=\tilde{z}^k- \alpha A^\top(Ax^{k+1}-b),\\
		\tilde{z}^{k+1}&=t_kz^{k+1}+(1-t_k)z^k,
	\end{aligned}
	$$
	where $\alpha=\frac{2}{\|A\|_2^2}$ and $t_{k-1}=1+\theta_k(\theta_{k-1}^{-1}-1)$ with $\theta_{-1}=1$ and $\theta_k=\frac{2}{k+2}$ for $k\geq0$; See \cite[Theorem 3.3]{huang2013accelerated} for more details. For the ADMM method, we set $x^0=0$ and $y^0=0$, and for the ALB method, we set $\tilde{z}^0=0$.
	
	Figures \ref{figue4}, \ref{figue5}, and \ref{figue6} compare the performance of ADMM, LB, ALB, SDCD, and ASDCD under different sensing matrices. In terms of epochs, ASDCD consistently outperforms all other methods across all matrix types. In terms of actual CPU time, however, ASDCD and ALB perform comparably, though both are more efficient than ADMM, LB, and SDCD. This divergence between epoch count and computational time arises because MATLAB leverages multithreading to accelerate matrix-vector products, which is the dominant cost in ADMM, LB, and ALB. While these methods require more iterations, they benefit from parallel computation, reducing their wall-clock time. Conversely, the epoch-efficient ASDCD derives less advantage from this low-level optimization, resulting in a relatively higher CPU time.

	\begin{figure}[hptb]
		\centering
		\begin{tabular}{cc}
			\includegraphics[width=0.4\linewidth]{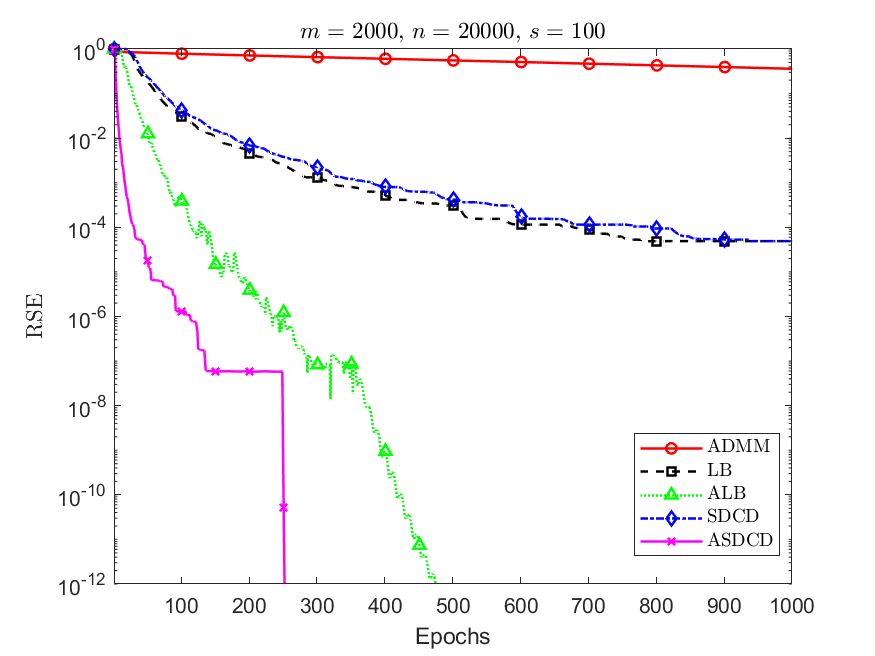}
			\includegraphics[width=0.4\linewidth]{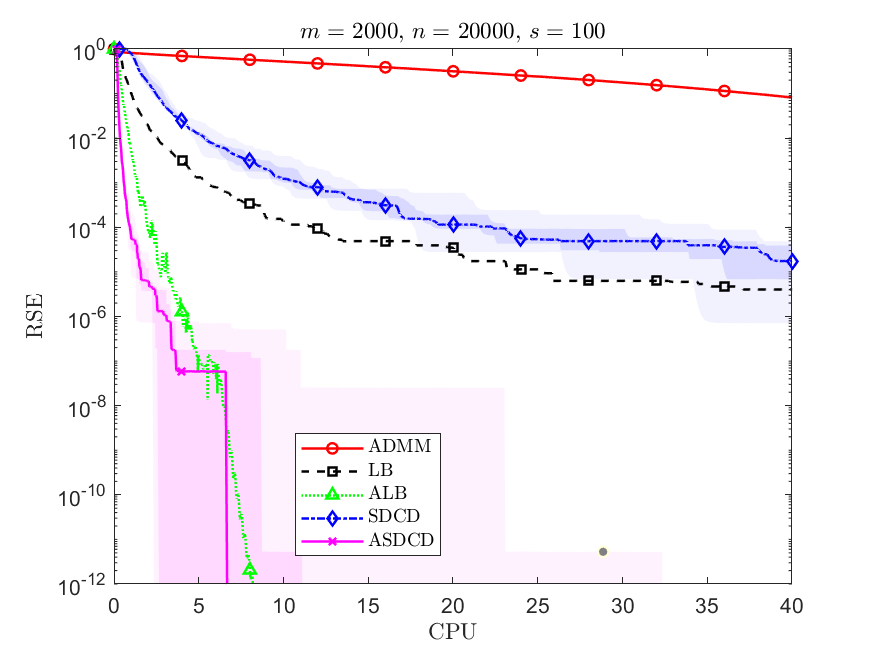}
		\end{tabular}
		\caption{The decrease of RSE across epochs and CPU time for ADMM, LB, ALB, SDCD, and ASDCD with Gaussian matrices. We set $\tau=100$. The title of each plot indicates the values of $m,n,s$. }
		\label{figue4}
	\end{figure}

	\begin{figure}[hptb]
		\centering
		\begin{tabular}{cc}
			\includegraphics[width=0.4\linewidth]{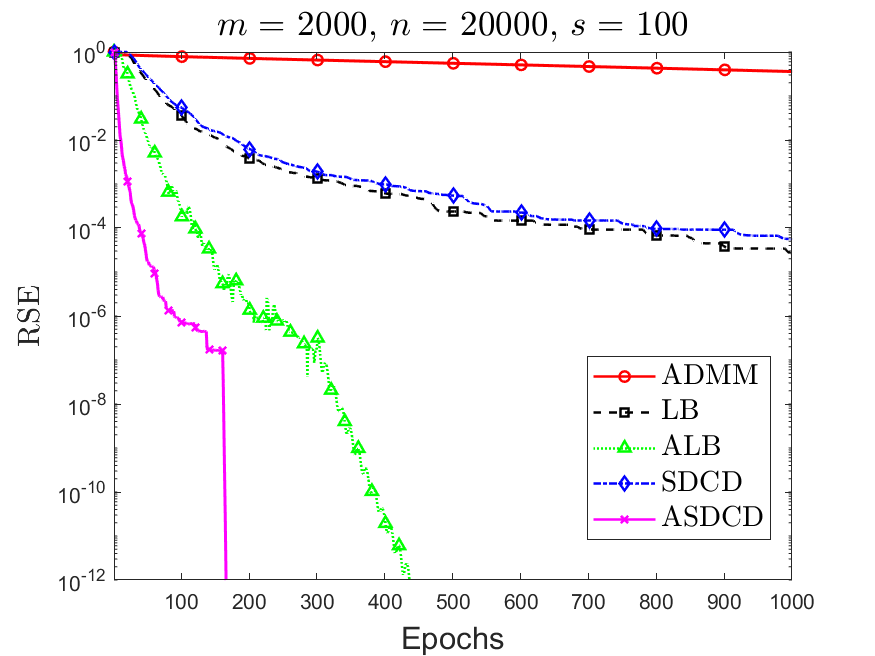}
			\includegraphics[width=0.4\linewidth]{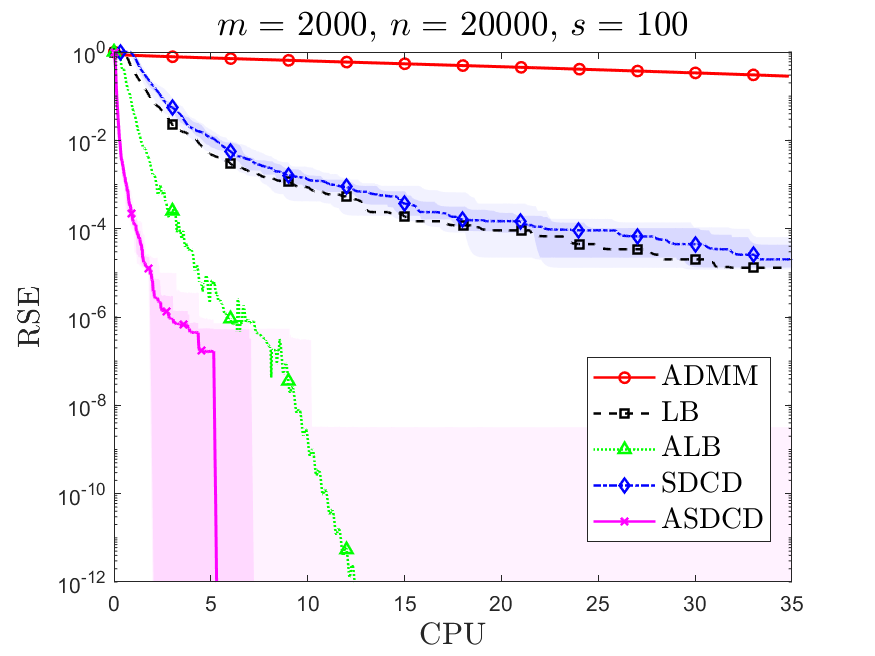}
		\end{tabular}
		\caption{The decrease of RSE across epochs and CPU time for ADMM, LB, ALB, SDCD, and ASDCD with Bernoulli random matrices. We set $\tau=50$. The title of each plot indicates the values of $m,n,s$. }
		\label{figue5}
	\end{figure}
	
	\begin{figure}[hptb]
		\centering
		\begin{tabular}{cc}
			\includegraphics[width=0.4\linewidth]{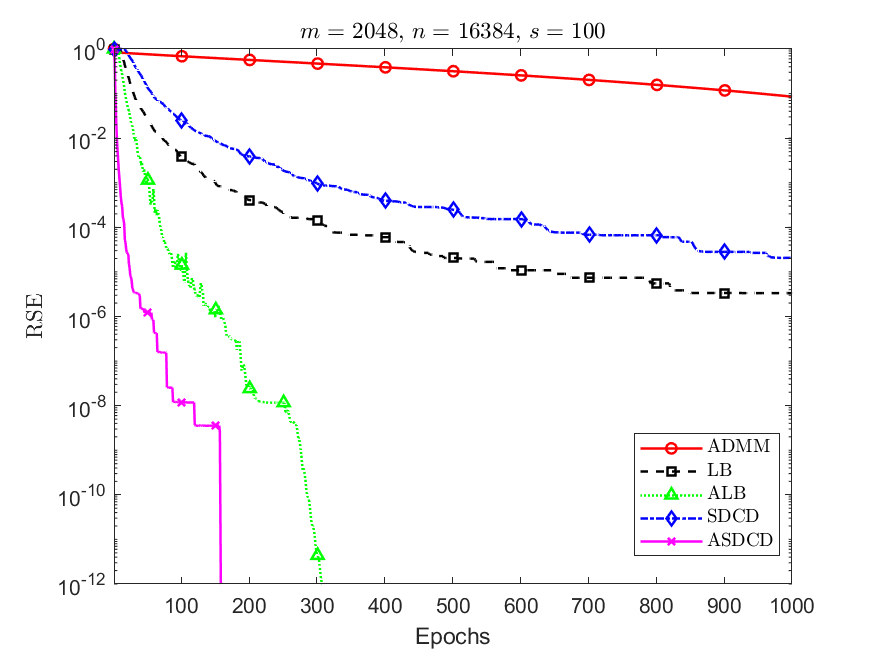}
			\includegraphics[width=0.4\linewidth]{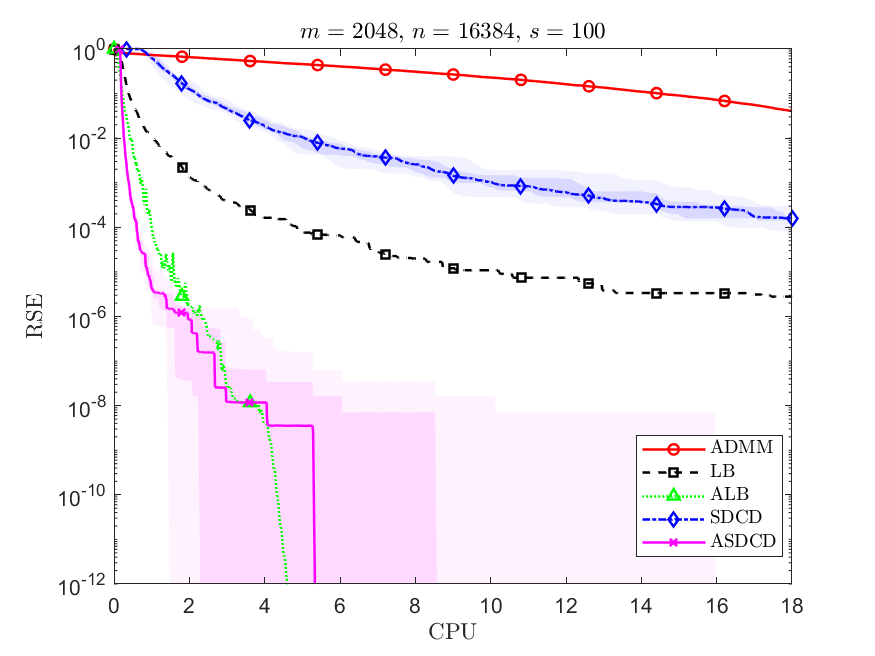}
		\end{tabular}
		\caption{The decrease of RSE across epochs and CPU time for ADMM, LB, ALB, SDCD, and ASDCD with randomly subsampled Hadamard matrices. We set $\tau=50$. The title of each plot indicates the values of $m,n,s$. }
		\label{figue6}
	\end{figure}
	
	%(\romannumeral1) normalized Gaussian matrices whose entries are generated from $i.i.d.$ normal distributions $\mathcal{N}(0,1)$ (\texttt{randn(m,n)} in \textsc{Matlab}) with columns being  normalized;
	%(\romannumeral2) partial discrete cosine transform (DCT) matrices whose $m$ rows are chosen randomly from the $n\times n$ DCT matrix.
	%We should mention that these matrices are well acknowledged to be efficient for sparse signal recovery in compressed sensing and have been widely used for numerical tests.
	%The $s$-sparse (the number of nonzero entries of a certain vector is less than or equal to $s$) vector $x\in\mathbb{R}^n$ and the noisy vector $\epsilon\in\mathbb{R}^m$ are generated based on the assumption that the elements of $x$ on its support and $\epsilon$ are generated from $\N(0,1)$.
	%After this, we set $b:=A x+\delta\epsilon$ where $\delta\geq0$ is the given noise level.
	%When $s$ is sufficiently small, we expect the $\ell_1/\ell_1$ minimization problem \eqref{l1-l1}, for which we solved by using the (inexact) ADMM algorithms, can yield a sparse solution $x_{sol}$ from the given data pair $(A,b)$.

	\section{Concluding remarks}
	
	This paper proposed an adaptive stochastic dual coordinate descent algorithmic framework, ASDCD, for minimizing a strongly convex objective function subject to linear constraints.  In particular, we incorporated the heavy ball momentum into our framework and proposed a novel strategy for adaptively learning the parameters $\alpha_k$ and $\beta_k$ using iteration information.
	If the objective function $f(x)=\frac{\gamma}{2}\|x-u\|^2_2-v$,  then the deterministic version of our method is serendipitously equivalent to  the conjugate gradient normal equation error (CGNE) method. We  discussed the extension and the geometric interpretation of our approach.
	Additionally, we have established that the ASDCD method can be reformulated into a computationally equivalent algorithm which, in certain cases, largely avoids the full-dimensional vector operations introduced by the momentum term.
	 Numerical results confirmed the efficiency of the ASDCD method.
	
	There are still many possible future avenues of research. The linearized Bregman method via split feasibility problems has been investigated in \cite{lorenz2014linearized},  which should be a valuable topic to explore the extensions of the adaptive heavy ball momentum approach for solving the general split feasibility problems. Recently, the Bregman-Kaczmarz method for solving nonlinear systems of equations was studied in \cite{gower2023bregman}. The convenience of extending our methods to nonlinear systems of equations would be a promising avenue for future research.  The stochastic heavy ball momentum has been studied in \cite{loizou2020momentum}, and it is also a valuable topic to investigate the stochastic coordinate descent with adaptive stochastic heavy ball momentum for minimizing the general $L$-smooth convex functions.  Furthermore, one can adopt the backtracking rule \cite{beck2009fast} to learn the parameter $L$.

%\backmatter

%\bmhead{Supplementary information}
%
%If your article has accompanying supplementary file/s please state so here. 
%
%Authors reporting data from electrophoretic gels and blots should supply the full unprocessed scans for key as part of their Supplementary information. This may be requested by the editorial team/s if it is missing.
%
%Please refer to Journal-level guidance for any specific requirements.

%\bmhead{Acknowledgements}
 
\bibliography{sn-bibliography}

%\noindent
%If any of the sections are not relevant to your manuscript, please include the heading and write `Not applicable' for that section. 
%
%%%===================================================%%
%%% For presentation purpose, we have included        %%
%%% \bigskip command. Please ignore this.             %%
%%%===================================================%%
%\bigskip
%\bigskip\noindent
%Springer journals and proceedings: \url{https://www.springer.com/gp/editorial-policies}
%
%\bigskip\noindent
%Nature Portfolio journals: \url{https://www.nature.com/nature-research/editorial-policies}
%
%\bigskip\noindent
%\textit{Scientific Reports}: \url{https://www.nature.com/srep/journal-policies/editorial-policies}
%
%\bigskip\noindent
%BMC journals: \url{https://www.biomedcentral.com/getpublished/editorial-policies}
%\end{flushleft}

\appendix
	
	\section{Proof of the main results}
\subsection{Proof of Theorem \ref{thm-rdcd}}
\label{secA1}
	Recall that the set $\mathcal{Q}_k$ is defined as 
\begin{equation}
	\label{xie-Qk}
	\mathcal{Q}_k=\{S\in \Omega_k \mid S^\top (Ax^k-b)\neq 0\},
\end{equation}
which represents the set of sampling matrices for which Algorithm \ref{RDCD} effectively executes one step such that $x^{k+1}\neq x^k$. Obviously,  $\{\mathcal{Q}_k, \mathcal{Q}_k^c\}$ forms a partition of $\Omega_k$. Given that $S_{k} \in \mathcal{Q}$, we denote
	$$
	\mathbb{E}_{k, S_{k} \in \mathcal{Q}} [\cdot]:=\mathbb{E}[\cdot | \mathcal{B}_{k}, S_{k} \in \mathcal{Q}].
	$$
	Note that for random variables $X$ and $Y$, if $X$ is measurable with respect to the $\sigma$-algebra generated by $\mathcal{B}_k$, which is denoted by $\sigma \langle \mathcal{B}_k \rangle$, we have the following equations holds \cite[Proposition 12.1.5 (ii)]{athreya2006measure}
	\begin{equation}
		\label{e:measure}
		\mathbb{E} [X|\mathcal{B}_k] = X \quad \text{and} \quad \mathbb{E} [XY|\mathcal{B}_k] = X  \mathbb{E} [Y|\mathcal{B}_k].
	\end{equation} 
	Since $z^k$ and $x^k$ are determined only by the elements in the sequence $(S_0, \cdots,S_{k-1})$, they are measurable with respect to $\sigma\langle \mathcal{B}_k \rangle$.

	\begin{proof}[Proof of Theorem \ref{thm-rdcd}] 
	Letting $\mathcal{Q}_k$ be defined as \eqref{xie-Qk} and supposing  the sample matrix $S_k \in \mathcal{Q}_k$, then we have
	\begin{equation}\label{xie-prf-0503}
		\begin{aligned}
			&D_{f,z^{k+1}}(x^{k+1},\widehat{x})=f(\widehat{x})+f^*(z^{k+1})-\langle z^{k+1},\widehat{x}\rangle\\
			\leq& f(\widehat{x})+f^*(z^{k})-\left\langle x^k,\alpha_k A^\top S_{k} S_{k}^\top(Ax^k-b) \right\rangle
			+\frac{1}{2\gamma}\left\|\alpha_k A^\top S_{k} S_{k}^\top(Ax^k-b)\right\|^2_2\\
			&-
			\left\langle z^k-\alpha_k A^\top S_{k} S_{k}^\top(Ax^k-b),\widehat{x}\right\rangle\\
			=&
			D_{f,z^{k}}(x^{k},\widehat{x})-\alpha_k \left\langle x^k-\widehat{x}, A^\top S_{k} S_{k}^\top(Ax^k-b) \right\rangle
			+\frac{\alpha_k^2}{2\gamma}\left\|A^\top S_{k} S_{k}^\top(Ax^k-b)\right\|^2_2
			\\
			=&D_{f,z^{k}}(x^{k},\widehat{x})-\frac{\zeta(2-\zeta) L_{\text{adap}}^{k,\gamma} }{2} \left\|S_{k}^\top(Ax^k-b)\right\|^2_2,
			%\leq&D_{f,z^{k}}(x^{k},\widehat{x})-\frac{\gamma\delta(2-\delta) L^{\text{block}}_{k} }{2} \left\|S_{k}^\top(Ax^k-b)\right\|^2_2,
	\end{aligned}\end{equation}
	where the first inequality follows from the $ \frac{1}{\gamma} $-smoothness of $ f^* $.
	%the definition of $\alpha_k$ and $L_{\text{adap}}^k$ in \eqref{alp} and \eqref{Lk}.
	%For convenience, we use $\mathop{\mathbb{E}}_{S_k \in \Omega_k}[\cdot]$ to denote $\mathop{\mathbb{E}}_{S_k \in \Omega_k}[\cdot | x^k]$.
	Thus
	{\begin{equation}
				\label{Exp}
				\begin{aligned}
					&\mathbb{E}_{k }\left[D_{f,z^{k+1}}(x^{k+1},\widehat{x})\right] \\
					=& \mathbb{P}(S_k \in \mathcal{Q}_k) \mathbb{E}_{k, S_k \in \mathcal{Q}_k}\left[D_{f,z^{k+1}}(x^{k+1},\widehat{x})\right]+\mathbb{P}(S_k \in \mathcal{Q}_k^{c}) \mathbb{E}_{k, S_k \in \mathcal{Q}_k^{c}} \left[D_{f,z^{k+1}}(x^{k+1},\widehat{x})\right] \\
					\leq& \mathbb{P}(S_k \in \mathcal{Q}_k) \mathbb{E}_{k, S_k \in \mathcal{Q}_k}\left[D_{f,z^{k}}(x^{k},\widehat{x})-\frac{\zeta(2-\zeta) L_{\text{adap}}^{k,\gamma} }{2} \left\|S_{k}^\top(Ax^k-b)\right\|^2_2\right]
					\\
					&+\mathbb{P}(S_k \in \mathcal{Q}_k^{c}) \mathbb{E}_{k, S_k \in \mathcal{Q}_k^{c}} \left[D_{f,z^{k}}(x^{k},\widehat{x})\right] \\
					=&\mathbb{E}_{k}\left[D_{f,z^{k}}(x^{k},\widehat{x})\right]-\frac{\zeta(2-\zeta)}{2} \mathbb{P}(S_k \in \mathcal{Q}_k) \mathbb{E}_{k, S_k \in \mathcal{Q}_{k}}\left[L_{\text{adap}}^{k,\gamma} \|S_k^\top (Ax^k-b)\|_2^2\right]\\
					=&D_{f,z^{k}}(x^{k},\widehat{x})-\frac{\zeta(2-\zeta)}{2} \mathbb{P}(S_k \in \mathcal{Q}_k) \mathbb{E}_{k, S_k \in \mathcal{Q}_{k}}\left[L_{\text{adap}}^{k,\gamma} \|S_k^\top (Ax^k-b)\|_2^2\right],				\end{aligned}
	\end{equation}}
	where the inequality follows from \eqref{xie-prf-0503} and the fact that if $S_k \in \mathcal{Q}_k^c$, then $z^{k+1}=z^k$ and $x^{k+1}=x^k$,
	and the last equality follows from the fact that $z^{k}$ and $x^{k}$ are measurable with respect to $\sigma\langle \mathcal{B}_k \rangle$ and \eqref{e:measure}.
	
	We consider the case where $\Omega_k$ is bounded. If $S_k \in \mathcal{Q}_k$, then we have
	%\begin{equation}
	%	\label{upp}
	$$				L_{\text{adap}}^{k,\gamma}=\frac{\gamma\|S_k^\top (Ax^k -b)\|_2^2}{\|A^\top S_k S_k^\top (Ax^k-b)\|_2^2}\geq \frac{\gamma}{\lambda_{\max}(A^\top S_kS_k^\top A)}\geq \frac{\gamma}{\lambda_{\max}^{(k)}}.$$
	%\end{equation}
	Substitute it into \eqref{Exp},  we can get
			\[
					\mathbb{E}_{k }\left[D_{f,z^{k+1}}(x^{k+1},\widehat{x})\right] 
					\leq
					D_{f,z^{k}}(x^{k},\widehat{x})-\frac{\gamma\zeta(2-\zeta) }{2\lambda_{\max}^{(k)}} \mathbb{P}(S_k \in \mathcal{Q}_k) \mathbb{E}_{k, S_k \in \mathcal{Q}_{k}}\left[\|S_k^\top (Ax^k-b)\|_2^2\right].
				\]
			Besides, note that $\mathbb{E}_{k, S_k \in \mathcal{Q}_{k}^{c}}\left[\|S_k^\top (Ax^k-b)\|_2^2\right]=0$ as $S_k^\top (Ax^k-b)=0$ for $S_{k} \in \mathcal{Q}_{k}^{c}$, we have
			\begin{equation} \nonumber
				\begin{aligned}
					&\mathbb{P}(S_k \in \mathcal{Q}_k) \mathbb{E}_{k, S_k \in \mathcal{Q}_{k}}\left[\|S_k^\top (Ax^k-b)\|_2^2\right] \\
					=&
					\mathbb{P}(S_k \in \mathcal{Q}_k) \mathbb{E}_{k, S_k \in \mathcal{Q}_{k}}\left[\|S_k^\top (Ax^k-b)\|_2^2\right]+\mathbb{P}(S_k \in \mathcal{Q}_{k}^{c}) \mathbb{E}_{k, S_k \in \mathcal{Q}_{k}^{c}}\left[\|S_k^\top (Ax^k-b)\|_2^2\right] \\
					=& \mathbb{E}_{k, S_k \in \Omega_{k}}\left[\|S_k^\top (Ax^k-b)\|_2^2\right].
				\end{aligned}
			\end{equation}
			Therefore,
	\begin{equation}\label{proof-xie0503-1}
		\begin{aligned}
			{\mathbb{E}_{k }\left[D_{f,z^{k+1}}(x^{k+1},\widehat{x})\right]}
			&{\leq
					D_{f,z^{k}}(x^{k},\widehat{x})-\frac{\gamma\zeta(2-\zeta) }{2\lambda_{\max}^{(k)}} \mathbb{E}_{k}\left[\left\| S_{k}^\top(Ax^k-b)\right\|^2_2\right]}\\
			&= D_{f,z^{k}}(x^{k},\widehat{x})-\frac{\gamma\zeta(2-\zeta) }{2\lambda_{\max}^{(k)}} \left\|(Ax^k-b)\right\|_{H_{k}}^{2}
			\\
			&\leq\left(1-\frac{\gamma\zeta(2-\zeta)\nu\lambda_{\min}\left(H_k\right) }{2\lambda_{\max}^{(k)}} \right)D_{f,z^{k}}(x^{k},\widehat{x}),
		\end{aligned}
	\end{equation}
	where
	the first equality follows from the fact that $x^{k}$ is measurable with respect to $\sigma\langle \mathcal{B}_k \rangle$ and \eqref{e:measure}, and the last inequality follows from that $H_k=\mathbb{E}_{S \in \Omega_k}[SS^\top]$ is positive definite.
	
	Next, we consider  the case where $\Omega_k$ is unbounded. If $S_k \in \mathcal{Q}_k$, we have
	%\begin{equation}
	%	\label{upp1}
	$$	L_{\text{adap}}^{k,\gamma}=\frac{\gamma\|S_k^\top (Ax^k-b)\|_2^2}{\|A^\top S_k S_k^\top (Ax^k-b)\|_2^2}\geq \frac{\gamma}{\lambda_{\max}\left(\frac{A^\top S_kS_k^\top A}{\|S_k\|_2^2}\right)} \frac{1}{\|S_k\|_2^2}\geq \frac{\gamma}{\lambda_{\max}^{(k)} \|S_k\|_2^2}.$$
	%	\end{equation}
Substitute it into \eqref{Exp} and use the similar arguments  to those  in \eqref{proof-xie0503-1}, we can get
$$
\begin{aligned}
	{\mathbb{E}_{k}\left[D_{f,z^{k+1}}(x^{k+1},\widehat{x})\right]}
	& {\leq D_{f,z^{k}}(x^{k},\widehat{x})-\frac{\gamma\zeta(2-\zeta) }{2\lambda_{\max}^{(k)}} \mathbb{E}_{k}\left[\frac{\left\| S_{k}^\top(Ax^k-b)\right\|^2_2}{\|S_k\|^2_2}\right]}\\
	&= D_{f,z^{k}}(x^{k},\widehat{x})-\frac{\gamma\zeta(2-\zeta) }{2\lambda_{\max}^{(k)}} \left\|(Ax^k-b)\right\|_{H_{k}}^{2}
	\\
	&\leq\left(1-\frac{\gamma\zeta(2-\zeta)\nu\lambda_{\min}\left(H_k\right) }{2\lambda_{\max}^{(k)}} \right)D_{f,z^{k}}(x^{k},\widehat{x}).
\end{aligned}$$
By taking the full expectation on both sides, we have
		$$
		\mathbb{E}\left[D_{f,z^{k}}(x^{k},\widehat{x})\right] \leq D_{f,z^{0}}(x^{0},\widehat{x})\prod \limits_{i=0}^{k-1} \left(1-\frac{\gamma\zeta(2-\zeta)\nu\lambda_{\min}\left(H_{i}\right) }{2\lambda_{\max}^{(i)}} \right).
		$$
		Furthermore, combining it with the inequality $D_{f,z^{k}}(x^{k},\widehat{x}) \geq \frac{\gamma}{2}\|x^{k}-\widehat{x}\|_{2}^{2}$, we can get
		$$
		\mathop{\mathbb{E}} \left[\| x^{k}-\hat{x}\|_2^2\right] \leq \frac{2 D_{f,z^{0}}(x^{0},\widehat{x})}{\gamma} \prod \limits_{i=0}^{k-1} \left(1-\frac{\gamma\zeta(2-\zeta)\nu\lambda_{\min}\left(H_{i}\right) }{2\lambda_{\max}^{(i)}} \right).
		$$
This completes the proof of this theorem.
%Taking expectation over the entire history and
%by induction on the iteration index $k$, we can obtain the desired result.
\end{proof}

\subsection{Proof of Theorem \ref{thm-mrdcd}}
\label{secA2}

The following lemma is essential for proving Theorem \ref{thm-mrdcd}.
	\begin{lemma}
	\label{lemma-xie-202510}
	Let $\{x^k\}_{k\geq1}$ be the sequences of iterates generated by Algorithm \ref{amRDCD}. Then
	\[
	\left\|w^{k+1} -x^k\right\|^2_2= \frac{L_{\text{adap}}^{k,\gamma}}{\gamma} \left\|S_k^{\top}(Ax^k-b)\right\|_2^2+\cos^2\theta_k \left\|y^{k+1}-\widehat{x}\right\|_2^2,
	\]
	where $L_{\mathrm{adap}}^{k,\gamma}$, $w^{k+1}$, $y^{k+1}$, and $\theta_k$ be given by \eqref{Lk}, \eqref{def-wk}, \eqref{y}, and \eqref{def-delta}, respectively. 
\end{lemma}
\begin{proof}
	Recall that $u^k = \langle d^k, z^k - z^{k-1} \rangle d^k - \| d^k \|_2^2 (z^k - z^{k-1})$. We define a candidate point as% we define
	$$
	\widetilde{w}^{k+1}:=y^{k+1}-\frac{\langle y^{k+1}-\widehat{x}, u^k \rangle}{\| u^k \|_2^2}u^{k}.
	$$
	Given that $y^{k+1}=x^k-\frac{L_{\text{adap}}^{k,\gamma}}{\gamma}d^k$, we have
	\[ 
		\begin{aligned}
			\langle y^{k+1} -x^k,\widetilde{w}^{k+1} -y^{k+1}\rangle&=\left\langle -\frac{L_{\text{adap}}^{k,\gamma}}{\gamma}d^k, -\frac{\langle y^{k+1}-\widehat{x}, u^k \rangle}{\| u^k \|_2^2}u^{k}\right\rangle\\
			&=\frac{L_{\text{adap}}^{k,\gamma}\langle y^{k+1}-\widehat{x}, u^k \rangle}{\gamma \| u^k \|_2^2}\langle d^k, u^{k} \rangle\\
			&=0.
		\end{aligned}
\]
Thus, it follows that
	\[\begin{aligned}
		\left\|\widetilde{w}^{k+1} -x^k\right\|^2_2=& \left\|y^{k+1} -x^k\right\|^2_2+\left\|\widetilde{w}^{k+1} -y^{k+1}\right\|^2_2\\
		=&\frac{L_{\text{adap}}^{k,\gamma}}{\gamma} \left\|S_k^{\top}(Ax^k-b)\right\|_2^2+\frac{\langle y^{k+1}-\widehat{x}, u^k \rangle^2}{\|u^k\|_2^2} \\
		=& \frac{L_{\text{adap}}^{k,\gamma}}{\gamma} \left\|S_k^{\top}(Ax^k-b)\right\|_2^2+\cos^2\theta_k \left\|y^{k+1}-\widehat{x}\right\|_2^2,
	\end{aligned}
	\]
where the second equality follows from the definitions of \( y^{k+1} \) and \( \widetilde{w}^{k+1} \), and the third equality follows from the definition of \( \theta_k \).
Our goal is now to show that \( \widetilde{w}^{k+1} = w^{k+1} \). Note that \( w^{k+1} \) is defined as the unique projection of \( \widehat{x} \) onto the affine subspace \( \widetilde{\Pi}_k = x^k + \text{Span}\{d^k, z^k - z^{k-1}\} \). It then suffices to prove that \( \widetilde{w}^{k+1} \) is indeed this projection. One can verify that 
	\begin{equation}
		\label{proof-xie-18-1}
	\widetilde{w}^{k+1}\in \widetilde{\Pi}_k, \langle \widetilde{w}^{k+1}-\widehat{x},d^k\rangle=0, \ \text{and}\ \langle \widetilde{w}^{k+1}-\widehat{x},u^k\rangle=0.
		\end{equation}
	 Noting that $\text{Span}\{d^k, z^k - z^{k-1}\}=\text{Span}\{d^k, u^k\}$, we have
	  \[\widetilde{\Pi}_k = x^k + \text{Span}\{d^k, u^k\}. \] 
	The  conditions in \eqref{proof-xie-18-1} therefore imply that \( \widetilde{w}^{k+1} \) is the orthogonal projection of \( \widehat{x} \) onto \( \widetilde{\Pi}_k \), which completes the proof.
\end{proof}

Now we are ready to prove Theorem \ref{thm-mrdcd}.

	\begin{proof}[Proof of Theorem \ref{thm-mrdcd}] 
	Consider the case where $\|d^k\|^2_2\|z^k-z^{k-1}\|^2_2-\langle d^k,z^k-z^{k-1}\rangle^2\neq0$,
	from \eqref{breg-dist}, \eqref{MT-xie} and the definition of $w^{k+1}$, we know that
	\[
		\begin{aligned}
			D_{f,z^{k+1}}(x^{k+1},\widehat{x})=&f(\widehat{x})+f^*(z^{k+1})-\langle z^{k+1},\widehat{x}\rangle\\
			\leq& f(\widehat{x})+f^*(z^k)-\langle z^k,\widehat{x}\rangle+\frac{1}{2\gamma}\left\|\alpha_k d^k-\beta_k(z^k-z^{k-1})\right\|^2_2
			\\&- \left\langle x^k-\widehat{x},\alpha_k d^k -\beta_k(z^k-z^{k-1}) \right\rangle\\
			=& D_{f,z^{k}}(x^{k},\widehat{x})+\frac{\gamma}{2} \left[\left\|w^{k+1} -\widehat{x}\right\|^2_2-\left\|x^k-\widehat{x}\right\|^2_2\right]\\
			=& D_{f,z^{k}}(x^{k},\widehat{x})-\frac{\gamma}{2} \left\|w^{k+1} -x^k\right\|^2_2,
		\end{aligned}
\]
	where the last equality follows from the fact that $w^{k+1}$ is the orthogonal projection of $\hat{x}$ onto the affine set $\widetilde{\Pi}_{k}=x^k+\text{Span}\{d^k, z^k-z^{k-1}\}$, which implies $\langle w^{k+1}-\hat{x},w^{k+1}-x^{k}\rangle=0$.
	%     Since $\|u^k\|^2_2=\| d^k\|^2_2\left(\| d^k \|_2^2 \| z^k-z^{k-1} \|_2^2 - \langle d^k, z^k-z^{k-1} \rangle^2\right)$, we know that $u^k\neq0$. Let $l_k:=\frac{\langle y^{k+1}-\widehat{x}, u^k \rangle}{\| u^k \|_2^2}$, one can verify  that
	% 	$$
	% 	w^{k+1}=y^{k+1}-l_k u^k.
	% 	$$
	% 	Hence
	% 	$$
	% 	\|w^{k+1}-y^{k+1}\|_2^2=l_k^2 \|u^k\|_2^2=\frac{\langle y^{k+1}-\widehat{x}, u^k \rangle^2}{\|u^k\|_2^2}=\cos^2\theta_k \| y^{k+1}-\widehat{x}\|^2_2,
	% 	$$
	% 	where $\theta_k$ is defined as \eqref{def-delta}.
	% Since $\left\langle d^k, u^k \right\rangle=0$, we have $\left\langle y^{k+1}-x^k, w^{k+1}-y^{k+1} \right\rangle=0$ and
	% 	\begin{align*}
		% 		\left\|w^{k+1} -x^k\right\|^2_2=& \left\|y^{k+1} -x^k\right\|^2_2+\left\|w^{k+1} -y^{k+1}\right\|^2_2\\
		% 		=& \frac{L_{\text{adap}}^{k,\gamma}}{\gamma} \left\|S_k^{\top}(Ax^k-b)\right\|_2^2+\cos^2\theta_k \left\|y^{k+1}-\widehat{x}\right\|_2^2.
		% 	\end{align*}
	From Lemma \ref{lemma-xie-202510}, we can get %Substituting it into \eqref{Breg1}, we can get
	$$
	D_{f,z^{k+1}}(x^{k+1},\widehat{x}) \leq D_{f,z^{k}}(x^{k},\widehat{x})-\frac{L_{\text{adap}}^{k,\gamma}}{2} \left\|S_k^{\top}(Ax^k-b)\right\|_2^2-\frac{\gamma \cos^2\theta_k }{2}  \left\|y^{k+1}-\widehat{x}\right\|_2^2.
	$$
	Consider the case where $\|d^k\|^2_2\|z^k-z^{k-1}\|^2_2-\langle d^k,z^k-z^{k-1}\rangle^2=0$, we have $u^k=0$ and hence $\cos^2\theta_k =0$. Thus, from Theorem \ref{thm-rdcd}, we can obtain the same inequality. Then, using the similar arguments  as that in the proof of Theorem \ref{thm-rdcd}, we can get this theorem.
	%Hence, if $\Omega_k$ is bounded, we have
	%\begin{align*}
	%\mathbb{E}\left[D_{f,z^{k+1}}(x^{k+1},\widehat{x}) \big | x^k \right]&\leq D_{f,z^{k}}(x^{k},\widehat{x})-\frac{\gamma\delta(2-\delta) }{2\lambda_{\max}^k} \mathbb{E}\left[\left\| S_{k}^\top(Ax^k-b)\right\|^2_2\right]\\
	%&= D_{f,z^{k}}(x^{k},\widehat{x})-\frac{\gamma\delta(2-\delta) }{2\lambda_{\max}^k} \left\| (Ax^k-b)\right\|^2_{H_k}
	%\\
	%&\leq\left(1-\frac{\gamma\delta(2-\delta)\nu\lambda_{\min}\left(H_k\right) }{2\lambda_{\max}^k} \right)D_{f,z^{k}}(x^{k},\widehat{x}).
	%\end{align*}
	%For the case where $\Omega_k$ is unbounded, we have
	%\begin{align*}
	%\mathbb{E}\left[D_{f,z^{k+1}}(x^{k+1},\widehat{x}) \big | x^k \right]&\leq D_{f,z^{k}}(x^{k},\widehat{x})-\frac{\gamma\delta(2-\delta) }{2\lambda_{\max}^k} \mathbb{E}\left[\frac{\left\| S_{k}^\top(Ax^k-b)\right\|^2_2}{\|S_k\|^2_2}\right]\\
	%&= D_{f,z^{k}}(x^{k},\widehat{x})-\frac{\gamma\delta(2-\delta) }{2\lambda_{\max}^k} \left\| (Ax^k-b)\right\|^2_{H_k}
	%\\
	%&\leq\left(1-\frac{\gamma\delta(2-\delta)\nu\lambda_{\min}\left(H_k\right) }{2\lambda_{\max}^k} \right)D_{f,z^{k}}(x^{k},\widehat{x}).
	%\end{align*}
	%Taking expectation over the entire history and
	%by induction on the iteration index $k$, we can obtain the desired result.
\end{proof}

\subsection{Proof of Proposition \ref{Theo}}
\label{secA3}

To prove Proposition \ref{Theo}, we first introduce two key lemmas.  
Let the parameters \(\{\overline{\alpha}_{k}, \overline{\beta}_{k}, \overline{S}_{k}\}_{k \geq 1}\) be given. Consider the following iteration scheme
	\begin{equation} \label{iteration-I}
		\begin{cases}
			\overline{d}^{k}&=A^\top \overline{S}_{k}\overline{S}_{k}^\top (A\nabla f^{*}(\overline{z}^{k})-b), \\
			\overline{z}^{k+1}&=
			\overline{z}^{k}-\overline{\alpha}_{k} \overline{d}^{k}+\overline{\beta}_{k}(\overline{z}^{k}-\overline{z}^{k-1}).
		\end{cases}
	\end{equation}
The initial points are chosen as \(\overline{z}^{0} \in \text{Range}(A^{\top})\) and \(\overline{\xi}^0 \in \mathbb{R}^m\), with \(\overline{z}^{1} = \overline{z}^{0} + A^\top \overline{\xi}^0\).

Next, given parameters \(\{\widehat{\alpha}_{k}, \widehat{\beta}_{k}, \widehat{S}_{k}\}_{k \geq 1}\), consider the following iteration scheme
\begin{equation}\label{iteration-II} 
	\begin{cases}
		\widehat{d}^{k}=A^\top \widehat{S}_{k}\widehat{S}_{k}^\top (A\nabla f^{*}(\widehat{z}^{k})-b), \\ 
		\text{If} \; \widehat{\beta}_{k}=0, \; \text{then}\; \widehat{\theta}_{k}=\frac{1}{2}, 
		\widehat{z}^{k+1}=\widehat{z}^{k}-\widehat{\alpha}_{k}\widehat{d}^{k},\; \text{and}\;
		\widehat{h}^{k+1}=\widehat{z}^{k}-2\widehat{\alpha}_{k} \widehat{d}^{k}.
		\\
		\text{If} \; \widehat{\beta}_{k}\neq0 \; \text{and} \; \widehat{\theta}_{k-1} \neq 1,  \; \text{then} \;
		\widehat{\theta}_{k}=
		\begin{cases}
			\frac{\widehat{\theta}_{k-1}}{1-\widehat{\theta}_{k-1}} \widehat{\beta}_{k} \, & \text{if} \; \widehat{\theta}_{k-2} \neq 1 \; \text{or} \; \widehat{\beta}_{k-1}=0; \\
			-\widehat{\beta}_{k} \qquad\, & \text{otherwise},
		\end{cases}\\
		\qquad \qquad \qquad \qquad \qquad \qquad \;  \;
		\widehat{z}^{k+1}=
		(1-\widehat{\theta}_{k})\widehat{z}^k+\widehat{\theta}_{k} \widehat{h}^{k}-\widehat{\alpha}_{k} \widehat{d}^{k}, \\
		\qquad \qquad \qquad \qquad \qquad \qquad \; \;
		\widehat{h}^{k+1}=
		\widehat{h}^{k}-\frac{\widehat{\alpha}_{k}}{\widehat{\theta}_k} \widehat{d}^{k}.\\
		\text{If} \; \widehat{\beta}_{k}\neq0 \; \text{and} \; \widehat{\theta}_{k-1} = 1,  \; \text{then} \;
		\widehat{\theta}_{k}=\frac{1}{2}, \\
		\qquad \qquad \qquad \qquad \qquad \qquad \; \;
		\widehat{z}^{k+1}=-\widehat{\beta}_{k}\widehat{z}^{k-1}+(1+\widehat{\beta}_{k})\widehat{h}^{k}-\widehat{\alpha}_{k} \widehat{d}^{k}, \\
		\qquad \qquad \qquad \qquad \qquad \qquad \; \;
		\widehat{h}^{k+1}=\widehat{h}^{k}.
	\end{cases}
\end{equation}
The initial conditions are \(\widehat{z}^{0} \in \text{Range}(A^{\top})\), \(\widehat{\xi}^{0} \in \mathbb{R}^{m}\), \(\widehat{z}^{1}=\widehat{z}^{0}+A^\top \widehat{\xi}^{0}\), \(\widehat{h}^{1}=\widehat{z}^{0}+2A^\top \widehat{\xi}^{0}\), \(\widehat{\theta}_{-1}=\widehat{\theta}_{0}=\frac{1}{2}\), and \(\widehat{\beta}_{0}=0\).
The recurrence for \(\widehat{\theta}_{k}\) ensures \(\widehat{\theta}_{k} \neq 0\) for all \(k \geq 1\), which guarantees that the vector \(\widehat{h}^{k}\) is well-defined throughout the iteration.

The following lemma establishes the equivalence between the iteration schemes \eqref{iteration-I} and \eqref{iteration-II}.
\begin{lemma} \label{Theo4}
	 Suppose that $\overline{z}^{0} = \widehat{z}^{0}$, $\overline{\xi}^{0} = \widehat{\xi}^{0}$, and $(\overline{\alpha}_{k}, \overline{\beta}_{k}, \overline{S}_{k}) = (\widehat{\alpha}_{k}, \widehat{\beta}_{k}, \widehat{S}_{k})$ for all $k \geq 1$.
	Then the sequences $\{\overline{z}^{k}\}_{k \geq 0}$ and $\{\widehat{z}^{k}\}_{k \geq 0}$, generated by \eqref{iteration-I} and \eqref{iteration-II} respectively, satisfy $\overline{z}^{k} = \widehat{z}^{k}$ for all $k \geq 0$.
\end{lemma}
		\begin{proof}
			Since $\overline{z}^{0}=\widehat{z}^{0}$ and $\overline{\xi}^{0}=\widehat{\xi}^{0}$,  it follows directly that
			$
			\widehat{z}^{1}=\widehat{z}^{0}+A^\top \widehat{\xi}^{0}=\overline{z}^{0}+A^\top \overline{\xi}^{0}=\overline{z}^{1}.
			$
			We now consider the update for $\widehat{z}^{2}$. Since $\widehat{\theta}_{0}=\frac{1}{2} \neq 1$, the recurrence proceeds based on the value of $\widehat{\beta}_{1}$. If $\widehat{\beta}_{1}=0$, then
			$$
			\widehat{z}^{2}=\widehat{z}^{1}-\widehat{\alpha}_{1}\widehat{d}^{1}=\widehat{z}^{1}-\widehat{\alpha}_{1}\widehat{d}^{1}+\widehat{\beta}_{1}(\overline{z}^{1}-\overline{z}^{0})=\overline{z}^{1}-\overline{\alpha}_{1} \overline{d}^{1}+\overline{\beta}_{1}(\overline{z}^{1}-\overline{z}^{0})=\overline{z}^{2}.
			$$
		 If $\widehat{\beta}_{1} \neq 0$, then $\widehat{\theta}_{1}=\frac{\widehat{\theta}_{0}}{1-\widehat{\theta}_{0}} \widehat{\beta}_{1}=\widehat{\beta}_{1}$ and
			\[
				\begin{aligned}
					\widehat{z}^{2}&=(1-\widehat{\theta}_{1})\widehat{z}^{1}+\widehat{\theta}_{1} \widehat{h}^{1}-\widehat{\alpha}_{1} \widehat{d}^{1} 
					=(1-\widehat{\beta}_{1})\widehat{z}^{1}+\widehat{\beta}_{1}(\widehat{z}^{0}+2A^\top \widehat{\xi}^{0})-\widehat{\alpha}_{1} \widehat{d}^{1} \\
					&=\widehat{z}^{1}-\widehat{\alpha}_{1} \widehat{d}^{1}+\widehat{\beta}_{1} (\widehat{z}^{0}+2A^\top \widehat{\xi}^{0}-\widehat{z}^{1}) 
					=\widehat{z}^{1}-\widehat{\alpha}_{1} \widehat{d}^{1}+\widehat{\beta}_{1} (\widehat{z}^{0}+2(\widehat{z}^{1}-\widehat{z}^{0})-\widehat{z}^{1}) \\
					&=\widehat{z}^{1}-\widehat{\alpha}_{1} \widehat{d}^{1}+\widehat{\beta}_{1} (\widehat{z}^{1}-\widehat{z}^{0}) 
					=\overline{z}^{1}-\overline{\alpha}_{1} \overline{d}^{1}+\overline{\beta}_{1} (\overline{z}^{1}-\overline{z}^{0}) 
					=\overline{z}^{2}.
				\end{aligned}
		\]
		 Consequently, we have $\widehat{z}^{2}=\overline{z}^{2}$ in both cases. Having established the base cases, we now proceed by induction. Assume that $\widehat{z}^{j} = \overline{z}^{j}$ holds for all $j \leq k$ and some $k \geq 2$. To complete the induction, we prove $\widehat{z}^{k+1} = \overline{z}^{k+1}$ by considering the values of $\widehat{\beta}_{k}$ and $\widehat{\theta}_{k-1}$.
		
		{\bf Case 1.} If $\widehat{\beta}_{k}=0$, then
		 $$
		 \widehat{z}^{k+1}=\widehat{z}^{k}-\widehat{\alpha}_{k}\widehat{d}^{k}=\widehat{z}^{k}-\widehat{\alpha}_{k}\widehat{d}^{k}+\widehat{\beta}_{k}(\widehat{z}^{k}-\widehat{z}^{k-1})=\overline{z}^{k}-\overline{\alpha}_{k} \overline{d}^{k}+\overline{\beta}_{k}(\overline{z}^{k}-\overline{z}^{k-1})=\overline{z}^{k+1}.
		 $$
		
		 {\bf Case 2.} If $\widehat{\beta}_{k} \neq 0$ and $\widehat{\theta}_{k-1} \neq 1$, then
		 $$
		 \widehat{z}^{k+1}=(1-\widehat{\theta}_{k})\widehat{z}^k+\widehat{\theta}_{k} \widehat{h}^{k}-\widehat{\alpha}_{k} \widehat{d}^{k}=\overline{z}^{k}-\overline{\alpha}_{k} \overline{d}^{k}+\widehat{\theta}_{k}(\widehat{h}^{k}-\overline{z}^{k}).
		 $$
		 Thus, it suffices to show that $\widehat{\theta}_{k}(\widehat{h}^{k}-\overline{z}^{k})=\overline{\beta}_{k}(\overline{z}^{k}-\overline{z}^{k-1})$. We verify this equality by examining the following subcases.
		
		{\bf Subcase 2.1.} If $\widehat{\beta}_{k} \neq 0$, $\widehat{\theta}_{k-1} \neq 1$, $\widehat{\beta}_{k-1} \neq 0$, and $\widehat{\theta}_{k-2} \neq 1$, then
		 \[
		 \begin{aligned}
			 \widehat{\theta}_{k}(\widehat{h}^{k}-\overline{z}^{k})&=\frac{\widehat{\theta}_{k-1}}{1-\widehat{\theta}_{k-1}} \widehat{\beta}_{k} \left(\widehat{h}^{k-1}-\frac{\widehat{\alpha}_{k-1}}{\widehat{\theta}_{k-1}}\widehat{d}^{k-1}-\overline{z}^{k}\right) \\
			 &=\frac{1}{1-\widehat{\theta}_{k-1}} \widehat{\beta}_{k} (\widehat{\theta}_{k-1}\widehat{h}^{k-1}-\widehat{\alpha}_{k-1} \widehat{d}^{k-1}-\widehat{\theta}_{k-1} \overline{z}^{k}) \\
			 &=\frac{1}{1-\widehat{\theta}_{k-1}} \widehat{\beta}_{k} (\overline{z}^{k}-(1-\widehat{\theta}_{k-1})\overline{z}^{k-1}-\widehat{\theta}_{k-1} \overline{z}^{k}) 
			 =\overline{\beta}_{k} (\overline{z}^{k}-\overline{z}^{k-1}),
			 \end{aligned}
		 \]
		 where the third equality follows from the inductive hypothesis, which gives $\overline{z}^{k}=(1-\widehat{\theta}_{k-1}) \overline{z}^{k-1}+\widehat{\theta}_{k-1} \widehat{h}^{k-1}-\widehat{\alpha}_{k-1} \widehat{d} ^{k-1}$ under the conditions $\widehat{\beta}_{k-1} \neq 0$ and $\widehat{\theta}_{k-2} \neq 1$.
		
		{\bf Subcase 2.2.} If $\widehat{\beta}_{k} \neq 0$, $\widehat{\theta}_{k-1} \neq 1$, $\widehat{\beta}_{k-1} \neq 0$, and $\widehat{\theta}_{k-2}=1$, then $\widehat{\theta}_{k-2}=1$ implies $\widehat{\beta}_{k-2} \neq 0$ and $\widehat{\theta}_{k-3} \neq 1$. Thus,
		 $$
		 \widehat{z}^{k-1}=(1-\widehat{\theta}_{k-2}) \widehat{z}^{k-2}+\widehat{\theta}_{k-2} \widehat{h}^{k-2}-\widehat{\alpha}_{k-2} \widehat{d}^{k-2}=\widehat{h}^{k-2}-\widehat{\alpha}_{k-2} \widehat{d}^{k-2}=\widehat{h}^{k-1}.
		 $$
		 Furthermore, we obtain
		 $$
		 \widehat{\theta}_{k}(\widehat{h}^{k}-\overline{z}^{k})=-\widehat{\beta}_{k} (\widehat{h}^{k-1}-\overline{z}^{k})=\overline{\beta}_{k}(\overline{z}^{k}-\overline{z}^{k-1}).
		 $$
		
		{\bf Subcase 2.3.} If $\widehat{\beta}_{k} \neq 0$, $\widehat{\theta}_{k-1} \neq 1$, and $\widehat{\beta}_{k-1} = 0$, then $\widehat{\beta}_{k-1} = 0$ implies $\widehat{\theta}_{k-1}=\frac{1}{2}$. Hence,
		 \[
		 \begin{aligned}
			 \widehat{\theta}_{k}(\widehat{h}^{k}-\overline{z}^{k})&=\frac{\widehat{\theta}_{k-1}}{1-\widehat{\theta}_{k-1}} \widehat{\beta}_{k}\left(\widehat{z}^{k-1}-2\widehat{\alpha}_{k-1} \widehat{d}^{k-1}-(\widehat{z}^{k-1}-\widehat{\alpha}_{k-1} \widehat{d}^{k-1})\right) \\
			 &=\widehat{\beta}_{k} (-\widehat{\alpha}_{k-1} \widehat{d}^{k-1})=\overline{\beta}_{k} (\overline{z}^{k}-\overline{z}^{k-1}).
			 \end{aligned}
		\]
		This establishes the equality $\widehat{\theta}_{k}(\widehat{h}^{k}-\overline{z}^{k})=\overline{\beta}_{k}(\overline{z}^{k}-\overline{z}^{k-1})$ for all subcases of Case 2.
		
		{\bf Case 3.} If $\widehat{\beta}_{k} \neq 0$ and $\widehat{\theta}_{k-1} = 1$, then $\widehat{\theta}_{k-1} = 1$ implies $\widehat{\beta}_{k-1} \neq 0$ and $\widehat{\theta}_{k-2} \neq 1$. From the recurrence, we have
		 $$
		 \widehat{z}^{k}=(1-\widehat{\theta}_{k-1}) \widehat{z}^{k-1}+\widehat{\theta}_{k-1} \widehat{h}^{k-1}-\widehat{\alpha}_{k-1} \widehat{d}^{k-1}=\widehat{h}^{k-1}-\widehat{\alpha}_{k-1} \widehat{d}^{k-1}=\widehat{h}^{k}.
		 $$
		 Furthermore, it follows that
		 \[
		 \begin{aligned}
			 \widehat{z}^{k+1}&=-\widehat{\beta}_{k}\widehat{z}^{k-1}+(1+\widehat{\beta}_{k})\widehat{h}^{k}-\widehat{\alpha}_{k} \widehat{d}^{k} 
			 =-\widehat{\beta}_{k}\widehat{z}^{k-1}+(1+\widehat{\beta}_{k})\widehat{z}^{k}-\widehat{\alpha}_{k} \widehat{d}^{k} \\
			 &=\widehat{z}^{k}-\widehat{\alpha}_{k} \widehat{d}^{k}+\widehat{\beta}_{k}(\widehat{z}^{k}-\widehat{z}^{k-1}) 
			 =\overline{z}^{k}-\overline{\alpha}_{k} \overline{d}^{k}+\overline{\beta}_{k}(\overline{z}^{k}-\overline{z}^{k-1}) = \overline{z}^{k+1}.
			 \end{aligned}
		 \]
		 Therefore, in all cases, we conclude that $\widehat{z}^{k+1}=\overline{z}^{k+1}$, which completes the induction and the proof of the lemma.
		\end{proof}
		%Building on Lemma \ref{Theo4}, we derive an equivalent formulation of \eqref{SDCD_fix} that incorporates information about $z^{k}$ without requiring its explicit computation, except in case where $\beta_{k}=0$.

		Let the parameters $\{\widetilde{\alpha}_{k}, \widetilde{\beta}_{k},\widetilde{S}_{k}\}_{k \geq 1}$ be given. Consider the following iteration scheme
		\begin{equation}\label{iteration-III} 
			\begin{cases}
				\widetilde{d}^{k} =A^\top \widetilde{S}_{k}\widetilde{S}_{k}^\top (A\nabla f^{*}(\widetilde{h}^{k}+\widetilde{\delta}_{k}\widetilde{q}^{k})-b),\\ 
				\text{If} \; \widetilde{\beta}_{k}=0, \; \text{then}\; \widetilde{\theta}_{k}=\frac{1}{2}, 
				\widetilde{\delta}_{k}^{*}  =1, \widetilde{z}^{k} =\widetilde{h}^{k}+\widetilde{\delta}_{k} \widetilde{q}^{k}, \; \text{and}\\
				\qquad \qquad \qquad \quad (\widetilde{h}^{k+1},\widetilde{q}^{k+1},\widetilde{\delta}_{k+1})=\left(\widetilde{z}^{k}-2\widetilde{\alpha}_{k}\widetilde{d}^{k},\;\,2\widetilde{\alpha}_{k}\widetilde{d}^{k},\;\, \frac{1}{2}\right).
				\\
				\text{If} \; \widetilde{\beta}_{k}\neq0 \; \text{and} \; \widetilde{\theta}_{k-1} \neq 1,  \; \text{then} \;
				\widetilde{\theta}_{k} =
				\begin{cases}
					\frac{\widetilde{\theta}_{k-1}}{1-\widetilde{\theta}_{k-1}} \widetilde{\beta}_{k}  & \text{if} \; \widetilde{\theta}_{k-2} \neq 1 \; \text{or} \; \widetilde{\beta}_{k-1}=0; \\
					-\widetilde{\beta}_{k} & \text{otherwise},
				\end{cases} \\
				\qquad \qquad \widetilde{\delta}_{k}^{*}  =\widetilde{\delta}_{k},  \; \text{and} \;
				(\widetilde{h}^{k+1}, \widetilde{q}^{k+1},\widetilde{\delta}_{k+1})=
				\left(\widetilde{h}^{k}-\frac{\widetilde{\alpha}_{k}}{\widetilde{\theta}_{k}} \widetilde{d}^{k}, \widetilde{q}^{k}+\frac{\widetilde{\alpha}_{k}}{\widetilde{\delta}_{k}^{*} \widetilde{\theta}_{k}} \widetilde{d}^{k},  (1-\widetilde{\theta}_{k})\widetilde{\delta}_{k}^{*}\right).\\
				\text{If} \; \widetilde{\beta}_{k}\neq0 \; \text{and} \; \widetilde{\theta}_{k-1} = 1,  \; \text{then} \;
				\widetilde{\theta}_{k} = \frac{1}{2}, \\
				\qquad \qquad \qquad \qquad \qquad \qquad \; \;\widetilde{\delta}_{k}^{*} =2 \widetilde{\delta}_{k-1}^{*} \widetilde{\beta}_{k}, \\
				\qquad \qquad \qquad \qquad \qquad \qquad \; \;(\widetilde{h}^{k+1}, \widetilde{q}^{k+1},\widetilde{\delta}_{k+1})=
				\left(\widetilde{h}^{k},  \widetilde{q}^{k}+\frac{\widetilde{\alpha}_{k}}{\widetilde{\delta}_{k}^{*} \widetilde{\theta}_{k}} \widetilde{d}^{k},  -\widetilde{\theta}_{k}\widetilde{\delta}_{k}^{*}\right).
			\end{cases}
		\end{equation}
		The initial conditions are $\widetilde{z}^{0} \in \text{Range}(A^{\top})$, $\widetilde{\xi}^{0} \in \mathbb{R}^{m}$,  $(\widetilde{h}^{0}, \widetilde{q}^{0},\widetilde{\delta}_{0})=(\widetilde{z}^{0},0,1)$, $(\widetilde{h}^{1},\widetilde{q}^{1},\widetilde{\delta}_{1})=\left(\widetilde{z}^{0}+2A^\top \widetilde{\xi}^{0},-2A^\top \widetilde{\xi}^{0}, \frac{1}{2}\right)$, $\widetilde{\delta}_{0}^{*}=1$, $\widetilde{\theta}_{-1}=\widetilde{\theta}_{0}=\frac{1}{2}$, and $\widetilde{\beta}_{0}=0$. 
		
		The recurrence for \(\widetilde{\theta}_{k}\) ensures \(\widetilde{\theta}_{k} \neq 0\) for all \(k \geq 1\).
	We prove by induction that \(\widetilde{\delta}_{k}^{*} \neq 0\) for all \(k \geq 1\). For \(k=1\), since \(\widetilde{\theta}_{0} = \frac{1}{2} \neq 1\), we have \(\widetilde{\delta}_{1}^{*} = 1\) if \(\widetilde{\beta}_{1} = 0\), or \(\widetilde{\delta}_{1}^{*} = \frac{1}{2}\) if \(\widetilde{\beta}_{1} \neq 0\). In both cases, \(\widetilde{\delta}_{1}^{*} \neq 0\).
	Assume \(\widetilde{\delta}_{k}^{*} \neq 0\) for some \(k \geq 1\). Then, (1) if \(\widetilde{\beta}_{k+1} = 0\), then \(\widetilde{\delta}_{k+1}^{*} = 1 \neq 0\);
		(2) if \(\widetilde{\beta}_{k+1} \neq 0\) and \(\widetilde{\theta}_{k} \neq 1\), then 
		\[ 
			\widetilde{\delta}_{k+1}^{*}=\widetilde{\delta}_{k+1}=
			\begin{cases}
				\frac{1}{2}, \qquad\qquad\; \text{if} \; \widetilde{\beta}_{k}=0, \\[1.7mm]
				(1-\widetilde{\theta}_{k})\widetilde{\delta}_{k}^{*} \quad\, \text{if} \; \widetilde{\beta}_{k} \neq 0 \; \text{and} \; \widetilde{\theta}_{k-1} \neq 1, \\[1.7mm]
				-\widetilde{\theta}_{k} \widetilde{\delta}_{k}^{*} \quad\quad\;\;\;\, \text{if} \; \widetilde{\beta}_{k} \neq 0 \; \text{and} \; \widetilde{\theta}_{k-1} = 1.
			\end{cases}
		\]
		Since $\widetilde{\theta}_{k} \notin \{0,1\}$ and $\widetilde{\delta}_{k}^{*} \neq 0$, we have $\widetilde{\delta}_{k+1}^{*} \neq 0$;
	(3) if \(\widetilde{\beta}_{k+1} \neq 0\) and \(\widetilde{\theta}_{k} = 1\), then \(\widetilde{\delta}_{k+1}^{*} = 2 \widetilde{\delta}_{k}^{*} \widetilde{\beta}_{k+1} \neq 0\).
		Hence, \(\widetilde{\delta}_{k+1}^{*} \neq 0\), and the sequences \(\{\widetilde{h}^{k}\}_{k \geq 0}\) and \(\{\widetilde{q}^{k}\}_{k \geq 0}\) in \eqref{iteration-III}  are well-defined.
	
Based on Lemma \ref{Theo4}, we establish the following result, which shows that the iteration schemes \eqref{iteration-I} and \eqref{iteration-III} are equivalent.
		\begin{lemma} \label{Theo3}
		Suppose that $\overline{z}^{0} = \widetilde{z}^{0}$, $\overline{\xi}^{0} = \widetilde{\xi}^{0}$, and $(\overline{\alpha}_{k}, \overline{\beta}_{k}, \overline{S}_{k}) = (\widetilde{\alpha}_{k}, \widetilde{\beta}_{k}, \widetilde{S}_{k})$ for all $k \geq 1$. Let the sequences $\{\overline{z}^{k}\}_{k \geq 0}$ and $\{\widetilde{h}^{k}, \widetilde{q}^{k}, \widetilde{\delta}_{k}, \widetilde{\delta}_{k}^{*}, \widetilde{\theta}_{k}\}_{k \geq 0}$ be generated by the iteration schemes \eqref{iteration-I} and \eqref{iteration-III}, respectively. Then for any $k \geq 0$, we have
			 $\overline{z}^{k}=\widetilde{h}^{k}+\widetilde{\delta}_{k} \widetilde{q}^{k}$ and
			 $\overline{z}^{k+1}-\overline{z}^{k}=-\widetilde{\theta}_{k} \widetilde{\delta}_{k}^{*} \widetilde{q}^{k+1}$.
		\end{lemma}
		\begin{proof} 
			We first prove the identity $\overline{z}^{k} = \widetilde{h}^{k} + \widetilde{\delta}_{k} \widetilde{q}^{k}$ by induction.
			We begin with the base cases. By the initial conditions, we have  $(\widetilde{h}^{0}, \widetilde{q}^{0},\widetilde{\delta}_{0})=(\widetilde{z}^{0},0,1)$ and $(\widetilde{h}^{1},\widetilde{q}^{1},\widetilde{\delta}_{1})=\left(\widetilde{z}^{0}+2A^\top \widetilde{\xi}^{0},-2A^\top \widetilde{\xi}^{0}, \frac{1}{2}\right)$, and thus, $\widetilde{h}^{0}+\widetilde{\delta}_{0} \widetilde{q}^{0}=\widetilde{z}^{0}=\overline{z}^{0}$ and
			$$
			\widetilde{h}^{1}+\widetilde{\delta}_{1} \widetilde{q}^{1}=(\widetilde{z}^{0}+2A^\top \widetilde{\xi}^{0})+\frac{1}{2} \cdot (-2A^\top \widetilde{\xi}^{0})=\widetilde{z}^{0}+A^\top \widetilde{\xi}^{0}=\overline{z}^{0}+A^\top \overline{\xi}^{0}=\overline{z}^{1}.
			$$
			We now consider the update for \((\widetilde{h}^{2}, \widetilde{q}^{2}, \widetilde{\delta}_{2})\).  Since $\widetilde{\theta}_{0}=\frac{1}{2} \neq 1$, the recurrence proceeds based on the value of $\widetilde{\beta}_{1}$.  If $\widetilde{\beta}_{1}=0$, then $(\widetilde{h}^{2},\widetilde{q}^{2},\widetilde{\delta}_{2})=\left(\widetilde{z}^{1}-2\widetilde{\alpha}_{1}\widetilde{d}^{1},\;\,2\widetilde{\alpha}_{1}\widetilde{d}^{1},\;\, \frac{1}{2}\right)$. Thus, we have
			\[
				\begin{aligned}
					\widetilde{h}^{2}+\widetilde{\delta}_{2} \widetilde{q}^{2}&=(\widetilde{z}^{1}-2\widetilde{\alpha}_{1}\widetilde{d}^{1})+\frac{1}{2} \cdot 2\widetilde{\alpha}_{1}\widetilde{d}^{1} 
					=\widetilde{z}^{1}-\widetilde{\alpha}_{1} \widetilde{d}^{1} 
					=(\widetilde{h}^{1}+\widetilde{\delta}_{1} \widetilde{q}^{1})-\widetilde{\alpha}_{1} \widetilde{d}^{1}+\widetilde{\beta}_{1}(\overline{z}^{1}-\overline{z}^{0}) \\
					&=\overline{z}^{1}-\overline{\alpha}_{1} \overline{d}^{1}+\overline{\beta}_{1}(\overline{z}^{1}-\overline{z}^{0}) 
					=\overline{z}^{2}.
				\end{aligned}
			\]
			 If $\widetilde{\beta}_{1} \neq 0$, then
			\[ 
					(\widetilde{h}^{2}, \widetilde{q}^{2},\widetilde{\delta}_{2})=
					\left(\widetilde{h}^{1}-\frac{\widetilde{\alpha}_{1}}{\widetilde{\theta}_{1}} \widetilde{d}^{1},  \widetilde{q}^{1}+\frac{\widetilde{\alpha}_{1}}{\widetilde{\delta}_{1}^{*} \widetilde{\theta}_{1}} \widetilde{d}^{1},  (1-\widetilde{\theta}_{1})\widetilde{\delta}_{1}^{*}\right) 
					=\left(\widetilde{h}^{1}-\frac{\widetilde{\alpha}_{1}}{\widetilde{\beta}_{1}} \widetilde{d}^{1},  \widetilde{q}^{1}+2\frac{\widetilde{\alpha}_{1}}{\widetilde{\beta}_{1}} \widetilde{d}^{1}, \frac{1-\widetilde{\beta}_{1}}{2}\right).
			\]
			Thus, we have
			\[
				\begin{aligned}
					\widetilde{h}^{2}+\widetilde{\delta}_{2} \widetilde{q}^{2}&=\widetilde{h}^{1}-\frac{\widetilde{\alpha}_{1}}{\widetilde{\beta}_{1}} \widetilde{d}^{1}+\frac{1-\widetilde{\beta}_{1}}{2} \left(\widetilde{q}^{1}+2\frac{\widetilde{\alpha}_{1}}{\widetilde{\beta}_{1}} \widetilde{d}^{1}\right) 
					=\widetilde{h}^{1}-\widetilde{\alpha}_{1} \widetilde{d}^{1}+\frac{1-\widetilde{\beta}_{1}}{2} \widetilde{q}^{1} \\
					&=(\widetilde{z}^{0}+2A^\top \widetilde{\xi}^{0})-\widetilde{\alpha}_{1} \widetilde{d}^{1}-(1-\widetilde{\beta}_{1})A^\top \widetilde{\xi}^{0} 
					=(\widetilde{z}^{0}+A^\top \widetilde{\xi}^{0})-\widetilde{\alpha}_{1} \widetilde{d}^{1}+\widetilde{\beta}_{1}A^\top \widetilde{\xi}^{0} \\
					&=(\overline{z}^{0}+A^\top \overline{\xi}^{0})-\overline{\alpha}_{1} \overline{d}^{1}+\overline{\beta}_{1}A^\top \overline{\xi}^{0} 
					=\overline{z}^{1}-\overline{\alpha}_{1} \overline{d}^{1}+\overline{\beta}_{1} (\overline{z}^{1}-\overline{z}^{0}) 
					=\overline{z}^{2}.
				\end{aligned}
			\]
			In either case, we have \(\overline{z}^2 = \widetilde{h}^2 + \widetilde{\delta}_2 \widetilde{q}^2\), thereby completing the base cases. Now, assume by induction that \(\overline{z}^j = \widetilde{h}^j + \widetilde{\delta}_j \widetilde{q}^j\) holds for all \(j \leq k\) and some \(k \geq 2\). We show that it also holds at step \(k+1\), by considering the three cases based on the values of \(\widetilde{\beta}_k\) and \(\widetilde{\theta}_{k-1}\):
				
			{\bf Case 1.} If $\widetilde{\beta}_{k}=0$, then
			\[
				\begin{aligned}
					\widetilde{h}^{k+1}+\widetilde{\delta}_{k+1} \widetilde{q}^{k+1}&=(\widetilde{z}^{k}-2\widetilde{\alpha}_{k} \widetilde{d}^{k})+\frac{1}{2} \cdot 2\widetilde{\alpha}_{k} \widetilde{d}^{k} 
					=\widetilde{z}^{k}-\widetilde{\alpha}_{k} \widetilde{d}^{k} \\
					&=(\widetilde{h}^{k}+\widetilde{\delta}_{k} \widetilde{q}^{k})-\widetilde{\alpha}_{k} \widetilde{d}^{k}+\widetilde{\beta}_{k}(\overline{z}^{k}-\overline{z}^{k-1}) \\
					&=\overline{z}^{k}-\overline{\alpha}_{k} \overline{d}^{k}+\overline{\beta}_{k}(\overline{z}^{k}-\overline{z}^{k-1}) 
					=\overline{z}^{k+1}.
				\end{aligned}
			\]
		
		 {\bf Case 2.} If $\widetilde{\beta}_{k} \neq 0$ and $\widetilde{\theta}_{k-1} \neq 1$, then
			\[
				\begin{aligned}
					\widetilde{h}^{k+1}+\widetilde{\delta}_{k+1}\widetilde{q}^{k+1}&=\widetilde{h}^{k}-\frac{\widetilde{\alpha}_{k}}{\widetilde{\theta}_k} \widetilde{d}^{k}+(1-\widetilde{\theta}_{k})\widetilde{\delta}_{k}^{*} \left(\widetilde{q}^{k}+\frac{\widetilde{\alpha}_{k}}{\widetilde{\delta}_{k}^{*} \widetilde{\theta}_{k}} \widetilde{d}^{k}\right) \\
					&=(\widetilde{h}^{k}+\widetilde{\delta}_{k}^{*} \widetilde{q}^{k})-\widetilde{\alpha}_{k} \widetilde{d}^{k}-\widetilde{\theta}_{k} \widetilde{\delta}_{k}^{*} \widetilde{q}^{k} 
					=(\widetilde{h}^{k}+\widetilde{\delta}_{k} \widetilde{q}^{k})-\widetilde{\alpha}_{k} \widetilde{d}^{k}-\widetilde{\theta}_{k} \widetilde{\delta}_{k} \widetilde{q}^{k} \\
					&=\overline{z}^{k}-\widetilde{\alpha}_{k} \widetilde{d}^{k}-\widetilde{\theta}_{k} (\overline{z}^{k}-\widetilde{h}^{k}) 
					=(1-\widetilde{\theta}_{k})\overline{z}^k+\widetilde{\theta}_{k} \widetilde{h}^{k}-\overline{\alpha}_{k} \overline{d}^{k}.
				\end{aligned}
			\]
			Furthermore, by Lemma \ref{Theo4}, we have $\overline{z}^{k+1}=(1-\widetilde{\theta}_{k})\overline{z}^k+\widetilde{\theta}_{k} \widetilde{h}^{k}-\overline{\alpha}_{k} \overline{d}^{k}$. Thus, it holds that $\overline{z}^{k+1}=\widetilde{h}^{k+1}+\widetilde{\delta}_{k+1} \widetilde{q}^{k+1}$.
			
			{\bf Case 3.} If $\widetilde{\beta}_{k} \neq 0$ and $\widetilde{\theta}_{k-1} = 1$, then $\widetilde{\theta}_{k-1} = 1$ implies that $\widetilde{\beta}_{k-1} \neq 0$ and $\widetilde{\theta}_{k-2} \neq 1$. Thus, we have $\widetilde{\delta}_{k}=(1-\widetilde{\theta}_{k-1})\widetilde{\delta}_{k-1}^{*}=0$. Furthermore, we can get
			\begin{equation} \label{h_xie191}
				\begin{aligned}
					\widetilde{h}^{k+1}+\widetilde{\delta}_{k+1}\widetilde{q}^{k+1}&=\widetilde{h}^{k}-\widetilde{\theta}_{k} \widetilde{\delta}_{k}^{*} \left(\widetilde{q}^{k}+\frac{\widetilde{\alpha}_{k}}{\widetilde{\delta}_{k}^{*} \widetilde{\theta}_{k}} \widetilde{d}^{k}\right) 
					=(\widetilde{h}^{k}+\widetilde{\delta}_{k} \widetilde{q}^{k})-\widetilde{\alpha}_{k} \widetilde{d}^{k}-\widetilde{\theta}_{k} \widetilde{\delta}_{k}^{*} \widetilde{q}^{k} \\
					&=\overline{z}^{k}-\widetilde{\alpha}_{k} \widetilde{d}^{k}-\widetilde{\delta}_{k-1}^{*}\widetilde{\beta}_{k} \widetilde{q}^{k}.
				\end{aligned}
			\end{equation}
			In addition, since $\widetilde{\theta}_{k-1} = 1$, $\widetilde{\beta}_{k-1} \neq 0$, and $\widetilde{\theta}_{k-2} \neq 1$, we have $\widetilde{h}^{k}=\widetilde{h}^{k-1}-\frac{\widetilde{\alpha}_{k-1}}{\widetilde{\theta}_{k-1}}\widetilde{d}^{k-1}=\widetilde{h}^{k-1}-\widetilde{\alpha}_{k-1} \widetilde{d}^{k-1}$. Thus,
			\[
				\begin{aligned}
					\overline{z}^{k}-\overline{z}^{k-1}&=(\widetilde{h}^{k}+\widetilde{\delta}_{k}\widetilde{q}^{k})-(\widetilde{h}^{k-1}+\widetilde{\delta}_{k-1}\widetilde{q}^{k-1}) 
					=\widetilde{h}^{k}-\widetilde{h}^{k-1}-\widetilde{\delta}_{k-1}\widetilde{q}^{k-1} \\
					&=-\widetilde{\alpha}_{k-1} \widetilde{d}^{k-1}-\widetilde{\delta}_{k-1}^{*} \widetilde{q}^{k-1} 
					=-\widetilde{\delta}_{k-1}^{*} \left(\widetilde{q}^{k-1}+\frac{\widetilde{\alpha}_{k-1}}{\widetilde{\delta}_{k-1}^{*}} \widetilde{d}^{k-1}\right) 
					=-\widetilde{\delta}_{k-1}^{*} \widetilde{q}^{k}.
				\end{aligned}
			\]
			Substitute it into \eqref{h_xie191}, we can get
			$$
			\widetilde{h}^{k+1}+\widetilde{\delta}_{k+1}\widetilde{q}^{k+1}=\overline{z}^{k}-\widetilde{\alpha}_{k} \widetilde{d}^{k}+\widetilde{\beta}_{k}(\overline{z}^{k}-\overline{z}^{k-1})=\overline{z}^{k}-\overline{\alpha}_{k} \overline{d}^{k}+\overline{\beta}_{k}(\overline{z}^{k}-\overline{z}^{k-1})=\overline{z}^{k+1}.
			$$
			Therefore, by induction, $\overline{z}^{k} = \widetilde{h}^{k} + \widetilde{\delta}_{k} \widetilde{q}^{k}$ for all $k \geq 0$.

			Next, we prove the identity  $\overline{z}^{k+1}-\overline{z}^{k}=-\widetilde{\theta}_{k} \widetilde{\delta}_{k}^{*} \widetilde{q}^{k+1}$ by induction.
			Since $\widetilde{\theta}_{0}=\frac{1}{2}$, $\widetilde{\delta}_{0}^{*}=1$, and $\widetilde{q}^{1}=-2A^{\top} \widetilde{\xi}^{0}$, we have
			$$
			\overline{z}^{1}-\overline{z}^{0}=A^{\top} \overline{\xi}^{0}=A^{\top} \widetilde{\xi}^{0}=-\overline{\theta}_{0} \widetilde{\delta}_{0}^{*} \widetilde{q}^{1}.
			$$
			Now, assume by induction that $\overline{z}^{j+1}-\overline{z}^{j}=-\widetilde{\theta}_{j} \widetilde{\delta}_{j}^{*} \widetilde{q}^{j+1}$ holds for all \(j \leq k\) and some \(k \geq 0\). We show that it also holds at step \(k+1\), by considering the three cases based on the values of \(\widetilde{\beta}_k\) and \(\widetilde{\theta}_{k-1}\):

			{\bf Case I.} If $\widetilde{\beta}_{k}=0$, then
			\[
				\begin{aligned}
					\overline{z}^{k+1}-\overline{z}^{k}&=(\widetilde{h}^{k+1}+\widetilde{\delta}_{k+1}\widetilde{q}^{k+1})-(\widetilde{h}^{k}+\widetilde{\delta}_{k}\widetilde{q}^{k}) 
					=\widetilde{z}^{k}-\widetilde{\alpha}_{k} \widetilde{d}^{k}-(\widetilde{h}^{k}+\widetilde{\delta}_{k}\widetilde{q}^{k}) \\
					&=-\widetilde{\alpha}_{k} \widetilde{d}^{k} 
					=-\widetilde{\theta}_{k} \widetilde{\delta}_{k}^{*} \widetilde{q}^{k+1}.
				\end{aligned}
			\]
			
			{\bf Case II.} If $\widetilde{\beta}_{k} \neq 0$ and $\widetilde{\theta}_{k-1} \neq 1$, then
			\[
				\begin{aligned}
					\overline{z}^{k+1}-\overline{z}^{k}&=(\widetilde{h}^{k+1}+\widetilde{\delta}_{k+1}\widetilde{q}^{k+1})-(\widetilde{h}^{k}+\widetilde{\delta}_{k}\widetilde{q}^{k}) 
					=\widetilde{h}^{k+1}-\widetilde{h}^{k}+\widetilde{\delta}_{k+1}\widetilde{q}^{k+1}-\widetilde{\delta}_{k}\widetilde{q}^{k} \\
					&=-\frac{\widetilde{\alpha}_{k}}{\widetilde{\theta}_k} \widetilde{d}^{k}+(1-\widetilde{\theta}_{k})\widetilde{\delta}_{k}^{*} \widetilde{q}^{k+1}-\widetilde{\delta}_{k}\widetilde{q}^{k} \\
					&=-\frac{\widetilde{\alpha}_{k}}{\widetilde{\theta}_k} \widetilde{d}^{k}+\widetilde{\delta}_{k}^{*} \left(\widetilde{q}^{k}+\frac{\widetilde{\alpha}_{k}}{\widetilde{\delta}_{k}^{*} \widetilde{\theta}_{k}} \widetilde{d}^{k}\right)-\widetilde{\theta}_{k} \widetilde{\delta}_{k}^{*} \widetilde{q}^{k+1}-\widetilde{\delta}_{k}^{*}\widetilde{q}^{k} 
					=-\widetilde{\theta}_{k} \widetilde{\delta}_{k}^{*} \widetilde{q}^{k+1}.
				\end{aligned}
			\]
			
			{\bf Case III.} If $\widetilde{\beta}_{k} \neq 0$ and $\widetilde{\theta}_{k-1} = 1$, then $\widetilde{\theta}_{k-1} = 1$ implies that $\widetilde{\beta}_{k-1} \neq 0$ and $\widetilde{\theta}_{k-2} \neq 1$. Thus, we have $\widetilde{\delta}_{k}=(1-\widetilde{\theta}_{k-1})\widetilde{\delta}_{k-1}^{*}=0$. Furthermore, we can get
			$$
			\overline{z}^{k+1}-\overline{z}^{k}=(\widetilde{h}^{k+1}+\widetilde{\delta}_{k+1}\widetilde{q}^{k+1})-(\widetilde{h}^{k}+\widetilde{\delta}_{k}\widetilde{q}^{k})=\widetilde{h}^{k}-\widetilde{\theta}_{k} \widetilde{\delta}_{k}^{*} \widetilde{q}^{k+1}-\widetilde{h}^{k}=-\widetilde{\theta}_{k} \widetilde{\delta}_{k}^{*} \widetilde{q}^{k+1}.
			$$
			Therefore, by induction,  $\overline{z}^{k+1}-\overline{z}^{k}=-\widetilde{\theta}_{k} \widetilde{\delta}_{k}^{*} \widetilde{q}^{k+1}$ for all $k\geq0$. 
	 This completes the proof of this lemma.
		\end{proof}
		
		Now, we are ready to prove Proposition \ref{Theo}.
		
		\begin{proof}[Proof of Proposition \ref{Theo}] 
			Since Algorithms \ref{amRDCD} and \ref{amRDCD_EI} share the same sampling matrices \(\{S_k\}_{k\geq1}\) and initial points \(z^{0}\) and \(\xi^0\), it follows from Lemma \ref{Theo3} that to establish the identity \(z^{k} = h^{k} + \delta_{k} q^{k}\) for all \(k \geq 0\), it suffices to show that the parameter sequences \(\{\alpha_{k}, \beta_{k}\}_{k \geq 1}\) in Algorithm \ref{amRDCD} are identical to those in Algorithm \ref{amRDCD_EI}.

			We begin by rewriting the parameter selection rule for \(\alpha_k\) and \(\beta_k\) in Algorithm~\ref{amRDCD} in an equivalent form.
			Recall that in Algorithm~\ref{amRDCD}, if
			\(
			\|d^k\|_2^2 \cdot \|z^k - z^{k-1}\|_2^2 - \langle d^k, z^k - z^{k-1} \rangle^2 = 0,
			\)
			then the parameters are selected as
			$
			\alpha_k = \frac{\gamma \|d_1^k\|_2^2}{\|d^k\|_2^2}$, $ \beta_k = 0
			$, where we define $\frac{0}{0}=0$ by convention. 
			Otherwise, we have
			\(
			\alpha_k = \gamma \cdot \frac{\|d_1^k\|_2^2 \cdot \|z^k - z^{k-1}\|_2^2 - \langle z^k - z^{k-1}, d^k \rangle \cdot \langle z^k - z^{k-1}, x^k - \widehat{x} \rangle}{\|d^k\|_2^2 \cdot \|z^k - z^{k-1}\|_2^2 - \langle d^k, z^k - z^{k-1} \rangle^2},
			\)
			and
			$
			\beta_k = \gamma \cdot \frac{ -\|d^k\|_2^2 \cdot \langle z^k - z^{k-1}, x^k - \widehat{x} \rangle + \langle d^k, z^k - z^{k-1} \rangle \cdot \| d_1^k \|_2^2 }{ \|d^k\|_2^2 \cdot \|z^k - z^{k-1}\|_2^2 - \langle d^k, z^k - z^{k-1} \rangle^2 }.
			$
			In the second case, if the numerator of \(\beta_k\) becomes zero, that is,
			$
			-\|d^k\|_2^2 \cdot \langle z^k - z^{k-1}, x^k - \widehat{x} \rangle + \langle d^k, z^k - z^{k-1} \rangle \cdot \| d_1^k \|_2^2 = 0,
			$
			then it also holds that \(\beta_k = 0\), and the expression for \(\alpha_k\) simplifies to
			$
			\alpha_k = \frac{\gamma \|d_1^k\|_2^2}{\|d^k\|_2^2}.
			$
			Summarizing the above cases, the parameter selection in Algorithm~\ref{amRDCD} can be equivalently rewritten as follows:
			\begin{itemize}
				\item[(1)] If 
				$
				\|d^k\|_2^2 \cdot \|z^k - z^{k-1}\|_2^2 - \langle d^k, z^k - z^{k-1} \rangle^2 = 0
				$
				or
				$
				-\|d^k\|_2^2 \cdot \langle z^k - z^{k-1}, x^k - \widehat{x} \rangle + \langle d^k, z^k - z^{k-1} \rangle \cdot \| d_1^k \|_2^2 = 0,
				$
				then
				$
				\alpha_k = \frac{\gamma \| d_1^k \|_2^2}{\| d^k \|_2^2} $, $ \beta_k = 0.
				$
				
				\item[(2)] Otherwise,
				\begin{equation}\nonumber
					\left\{
					\begin{array}{ll}
						\alpha_{k}= \gamma \cdot \frac{ \|d_1^k\|_2^2 \cdot \|z^k - z^{k-1}\|_2^2 - \langle d^k, z^k - z^{k-1} \rangle \cdot \langle z^k - z^{k-1}, x^k - \widehat{x} \rangle }{ \|d^k\|_2^2 \cdot \|z^k - z^{k-1}\|_2^2 - \langle d^k, z^k - z^{k-1} \rangle^2 }, \\[0.5cm]
						\beta_{k}=\gamma \cdot \frac{ -\|d^k\|_2^2 \cdot \langle z^k - z^{k-1}, x^k - \widehat{x} \rangle + \langle d^k, z^k - z^{k-1} \rangle \cdot \| d_1^k \|_2^2 }{ \|d^k\|_2^2 \cdot \|z^k - z^{k-1}\|_2^2 - \langle d^k, z^k - z^{k-1} \rangle^2 }(\neq0).
					\end{array}
					\right.
				\end{equation}
			\end{itemize}
		We now prove by induction that for all \(k \geq 1\), the parameters \(\alpha_k\) and \(\beta_k\) in Algorithm \ref{amRDCD} are identical to those in Algorithm \ref{amRDCD_EI}, under the assumption that both algorithms share the same sampling matrices \(\{S_k\}_{k\geq1}\) and initial points \(z^{0}\) and \(\xi^0\).
			
			For the base case \(k = 1\), note that
			$
			z^{1} - z^{0} = A^\top \xi^0 = -\theta_0 \delta_0^{*} q^1.
			$
			From this and the definitions \(l_1 = \|q^1\|_2^2\) and \(\tau_1 = -2 \langle \xi^0, b \rangle = \langle -2A^\top \xi^0, \widehat{x} \rangle = \langle q^1, \widehat{x} \rangle\), we obtain
			$
			\|z^1 - z^0\|_2^2 = \theta_0^2 (\delta_0^*)^2 l_1$, $\langle d^1, z^1 - z^0 \rangle = -\theta_0 \delta_0^* \langle d^1, q^1 \rangle,
			$
			and
			\[
			\langle z^1 - z^0, x^1 - \widehat{x} \rangle = -\theta_0 \delta_0^* \langle q^1, x^1 - \widehat{x} \rangle = -\theta_0 \delta_0^* ( \langle q^1, x^1 \rangle - \tau_1 ).
			\]
			As a result, the selection of \(\alpha_1\) and \(\beta_1\) in Algorithm~\ref{amRDCD} can be re-expressed in terms of the variables \(d_{1}^{1}\), \(d^1\), \(h^1\), \(q^1\), and the scalars \(\theta_0\), \(\delta_0^*\), \(\delta_1\), \(l_1\), \(\tau_1\) from Algorithm~\ref{amRDCD_EI}, as follows.
				If 
			$
			\|d^1\|_2^2 l_1 - \langle d^1, q^1 \rangle^2 = 0$ or $ \|d^1\|_2^2 (\langle q^1, \nabla f^*(h^1 + \delta_1 q^1) \rangle - \tau_1) - \langle d^1, q^1 \rangle \|d_1^1\|_2^2 = 0,
			$
			then
			$
			\alpha_1 = \frac{\gamma \|d_1^1\|_2^2}{\|d^1\|_2^2},  \beta_1 = 0.
			$
			Otherwise, $\alpha_1 = \gamma \cdot \frac{ \|d_1^1\|_2^2 l_1 - \langle d^1, q^1 \rangle \big( \langle q^1, \nabla f^*(h^1 + \delta_1 q^1) \rangle - \tau_1 \big)}{ \|d^1\|_2^2 l_1 - \langle d^1, q^1 \rangle^2 }$,
			$\beta_1 = \frac{\gamma}{\theta_0 \delta_0^{*}} \cdot \frac{ \|d^1\|_2^2 \left( \langle q^1, \nabla f^*(h^1 + \delta_1 q^1) \rangle - \tau_1 \right) - \langle d^1, q^1 \rangle \|d_1^1\|_2^2 }{ \|d^1\|_2^2 l_1 - \langle d^1, q^1 \rangle^2 }.$
			This formulation exactly matches the parameter selection strategy in Algorithm \ref{amRDCD_EI} for \(k = 1\).
			
			Now, assume by induction that for all \(j \leq k\) and some \(k \geq 2\), the parameters \(\alpha_{j-1}\) and \(\beta_{j-1}\) in Algorithm \ref{amRDCD} are identical to those in Algorithm \ref{amRDCD_EI}. We now prove that this equivalence holds for \(\alpha_k\) and \(\beta_k\). By the inductive hypothesis and Lemma \ref{Theo3}, we have
			$
			z^{k}-z^{k-1}=-\theta_{k-1} \delta_{k-1}^{*} q^{k}
			$.
			Thus, we have
			$
					\|z^{k}-z^{k-1}\|^2_2=\theta_{k-1}^{2} (\delta_{k-1}^{*})^{2} \|q^{k}\|_{2}^{2}$, 
			$		\langle d^{k},z^{k}-z^{k-1}\rangle=-\theta_{k-1} \delta_{k-1}^{*} \langle d^{k},q^{k}\rangle,
		$
			and
			$$
			\langle z^{k}-z^{k-1},x^{k}-\widehat{x}\rangle=-\theta_{k-1} \delta_{k-1}^{*}\langle q^{k},x^{k}-\widehat{x}\rangle=-\theta_{k-1} \delta_{k-1}^{*}(\langle q^{k},x^{k}\rangle-\langle q^{k},\widehat{x}\rangle).
			$$
		We claim that the auxiliary variables in Algorithm \ref{amRDCD_EI} satisfy \(l_k = \|q^{k}\|_2^2\) and \(\tau_k = \langle q^{k}, \widehat{x} \rangle\). Granting this claim for now, the parameter selection rule for \(\alpha_k\) and \(\beta_k\) in Algorithm \ref{amRDCD} becomes equivalent to
			\begin{itemize}
				\item[(1)] If $\|d^k\|^2_2 l_{k}-\langle d^k,q^{k}\rangle^2=0$ or $ \|d^{k}\|_{2}^{2}(\langle q^{k}, \nabla f^{*}(h^{k}+\delta_{k} q^{k})\rangle-\tau_{k})-\langle d^{k},q^{k} \rangle \|d_{1}^{k}\|_{2}^{2}=0$, then
				$
				\alpha_{k}=	\frac{\gamma\left\|d_{1}^{k}\right\|^2_2}{\left\| d^{k}\right\|_2^2}$, $\beta_{k}=0.
				$
			
				\item[(2)] Otherwise,
				\begin{equation}\nonumber
					\left\{
					\begin{array}{ll}
						\alpha_{k}=\gamma \frac{\|d_{1}^{k}\|^2_2l_{k}-\langle d^{k}, q^{k} \rangle (\langle q^{k}, \nabla f^{*}(h^{k}+\delta_{k} q^{k}) \rangle-\tau_{k})}{\|d^k\|^2_2l_{k}-\langle d^k,q^{k}\rangle^2}, \\[0.5cm]
						\beta_{k}=\frac{\gamma}{\theta_{k-1} \delta_{k-1}^{*}} \frac{\|d^{k}\|_{2}^{2}(\langle q^{k}, \nabla f^{*}(h^{k}+\delta_{k} q^{k} \rangle-\tau_{k}) - \langle d^{k},q^{k} \rangle \|d_{1}^{k}\|_{2}^{2}}{ \|d^k\|^2_2l_{k}-\langle d^k,q^{k}\rangle^2}.
					\end{array}
					\right.
				\end{equation}
			\end{itemize}
		This formulation is identical to the parameter selection strategy in Algorithm \ref{amRDCD_EI} for iteration \(k\).
		%	Consequently, the parameter sequences $\{\alpha_{k}, \beta_{k}\}_{k \geq 1}$ in Algorithm \ref{amRDCD} are identical to those in Algorithm \ref{amRDCD_EI}.
			
			It remains to prove the claim that \(l_k = \|q^{k}\|_2^2\) and \(\tau_k = \langle q^{k}, \widehat{x} \rangle\). We proceed by induction on \(k\).
		For the base case \(k=1\), we have already verified that \(l_1 = \|q^{1}\|_2^2\) and \(\tau_1 = \langle q^{1}, \widehat{x} \rangle\).
		Now assume that \(l_{k-1} = \|q^{k-1}\|_2^2\) and \(\tau_{k-1} = \langle q^{k-1}, \widehat{x} \rangle\) for some \(k \geq 2\). We prove the identities for \(k\) by case analysis on the definition of \(q^{k}\).
			If $\beta_{k-1}=0$, then $q^{k}=2\alpha_{k-1} d^{k-1}$. Thus, we have
			$$
			\|q^{k}\|_{2}^{2}=\|2\alpha_{k-1} d^{k-1}\|_{2}^{2}=4\alpha_{k-1}^{2}\|d^{k-1}\|_{2}^{2}=l_{k},
			$$
			and
			$$
			\langle q^{k},\widehat{x}\rangle=\langle 2\alpha_{k-1} d^{k-1},\widehat{x}\rangle=2\alpha_{k-1}\langle A^{\top}S_{k-1}d_{1}^{k-1}, \widehat{x} \rangle=2\alpha_{k-1}\langle d_{1}^{k-1}, S_{k-1}^{\top}b\rangle=\tau_{k}.
			$$
			If $\beta_{k-1} \neq 0$, then $q^{k}=q^{k-1}+\frac{\alpha_{k-1}}{\delta_{k-1}^{*} \theta_{k-1}} d^{k-1}$. Thus, we have
			$$
			\|q^{k}\|_{2}^{2}=l_{k-1}+2\frac{\alpha_{k-1}}{\delta_{k-1}^{*}\theta_{k-1}} \langle d^{k-1},q^{k-1} \rangle+\frac{\alpha_{k-1}^{2}}{(\delta_{k-1}^{*})^{2} \theta_{k-1}^{2}} \|d^{k-1}\|_{2}^{2}=l_{k},
			$$
			and
			$$
			\langle q^{k},\widehat{x}\rangle=\left \langle q^{k-1}++\frac{\alpha_{k-1}}{\delta_{k-1}^{*} \theta_{k-1}} d^{k-1},\widehat{x} \right \rangle=\tau_{k-1}+\frac{\alpha_{k-1}}{\delta_{k-1}^{*} \theta_{k-1}} \langle d_{1}^{k-1}, S_{k-1}^{\top}b \rangle=\tau_{k}.
			$$
		 Therefore, the claim holds for all \(k \geq 1\). This completes the induction and the proof of the theorem.
		\end{proof}

%An appendix contains supplementary information that is not an essential part of the text itself but which may be helpful in providing a more comprehensive understanding of the research problem or it is information that is too cumbersome to be included in the body of the paper.




%\bibliographystyle{plain} 

%\bibliographystyle{abbrv}

\end{CJK}
\end{document}